\documentclass[reqno]{amsart}
\usepackage{tkz-euclide}
\usepackage{hyperref}
\usepackage{tikz-3dplot}
\usepackage{amsmath}  
\usepackage{amsfonts}
\usepackage{amssymb} 
\usepackage{mathrsfs} 
\usepackage{graphicx}  
\usepackage{bm}
\usepackage{tikz}
\usepackage[normalem]{ulem} 
\usetikzlibrary{calc} 
\usepackage{xcolor}  
\usepackage{tikz}
\usepackage{soul}

\usepackage{tkz-euclide}
\usepackage{amsmath}  
\usepackage{amsfonts}
\usepackage{amssymb} 
\usepackage{tikz-3dplot}
\usepackage{mathrsfs} 
\usepackage{graphicx}  
\usepackage{bm}
\usepackage{tikz} 
\usetikzlibrary{svg.path} 
\usepackage{amssymb}
\usepackage{pgfplots}

\usepackage{pgfplots}
\pgfplotsset{compat = newest}

\usetikzlibrary{calc} 
\usepackage{xcolor}  
\usetikzlibrary{arrows.meta} 
\usepackage{amsthm} 
\usepackage{float}
\usepackage{enumitem}
\usepackage[english]{babel}
\usepackage[font=footnotesize, labelfont=bf]{ caption}
\usepackage{ esint }
\usepackage{stackengine,scalerel,graphicx}

\usepackage{amsthm} 
\usepackage{float}
\usepackage{enumitem}
\usepackage[english]{babel}
\usepackage[font=footnotesize, labelfont=bf]{ caption}
\usepackage{ esint }
\usepackage{stackengine,scalerel,graphicx}
\stackMath

\definecolor{trueblue}{rgb}{0.0, 0.45, 0.81}

\definecolor{forestgreen}{rgb}{0.13, 0.55, 0.13}

\newcommand{\NNN}{\color{black}} 
\newcommand{\MMM}{\color{black}} 
\newcommand{\BBB}{\color{black}} 
\newcommand{\GGG}{\color{black}} 
 
\newcommand{\EEE}{\color{black}} 
 
\newcommand{\OOO}{\color{black}} 
\newcommand{\CCC}{\color{black}} 
\newcommand{\AAA}{\color{black}}

\newcommand{\eps}{\varepsilon}

\theoremstyle{plain}
\begingroup
\newtheorem{theorem}{Theorem}[section]
\newtheorem{lemma}[theorem]{Lemma}

\newtheorem{remark}[theorem]{Remark}
\newtheorem{proposition}[theorem]{Proposition}
\newtheorem{corollary}[theorem]{Corollary}
\endgroup

\newenvironment{step}[1]{\underline{Step #1}.}{}

\theoremstyle{definition}
\begingroup
\newtheorem{definition}[theorem]{Definition}
\endgroup

\renewcommand{\tilde}{\widetilde}

\DeclareMathOperator{\dist}{dist}

\numberwithin{equation}{section}
\newcommand{\N}{\mathbb{N}}
\newcommand{\Z}{\mathbb{Z}}
\newcommand{\F}{\mathcal{F}}
\newcommand{\E}{\mathcal{E}}
\newcommand{\R}{\mathbb{R}}

\newcommand{\I}{\mathcal{I}}
\renewcommand{\S}{\mathbb{S}}
\renewcommand{\L}{\mathcal{L}}
\renewcommand{\H}{\mathcal{H}}

\newcommand{\x}{{\times}}
\newcommand{\defas}{:=}

\newcommand{\Q}{\mathcal{Q}}

\usepackage[normalem]{ulem}

\begin{document}
	
\title[Derivation of Kirchhoff-type plate theories for elastic materials with voids]{Derivation of Kirchhoff-type plate theories for elastic materials with voids}
	
\author[Manuel Friedrich]{Manuel Friedrich} 
\address[Manuel Friedrich]{Department of Mathematics, FAU Erlangen-N\"urnberg. Cauerstr.~11,
D-91058 Erlangen, Germany}
\email{manuel.friedrich@fau.de}

\author{Leonard Kreutz}
\address[Leonard Kreutz]{Department of Mathematics, School of Computation, Information and Technology, 
Technical University of Munich,
Boltzmannstr. 3, 85748 Garching, Germany}
\email{leonard.kreutz@tum.de}

\author{Konstantinos Zemas}
\address[Konstantinos Zemas]{Institute for Applied Mathematics, University of Bonn\\
Endenicher Allee 60, 53115 Bonn, Germany}
\email{zemas@iam.uni-bonn.de}

\begin{abstract}
We rigorously derive a Blake-Zisserman-Kirchhoff theory for thin plates with material voids, starting from a three-dimensional model with elastic bulk and interfacial energy  featuring a {Willmore-type curvature penalization}. The effective two-dimensional model comprises a classical elastic bending energy and surface terms which reflect the possibility that voids can persist in the limit, that the limiting plate can be broken apart into several pieces, or that the plate can be folded.  Building upon and extending the techniques used in  the authors' recent work \cite{KFZ:2023} on the derivation of one-dimensional theories for thin brittle rods with voids, the present contribution generalizes the results of \cite{SantiliSchmidt2022}, by considering general geometries on the admissible set of voids and constructing recovery sequences for all admissible limiting configurations.
\end{abstract}
	
\maketitle	
	
\section{Introduction}\label{introduction}
The rigorous derivation of variational theories for lower dimensional elastic objects, e.g.,  membranes, plates, shells, beams, and rods, has been one of the fundamental and challenging questions in the mathematical development of continuum mechanics. A common aspect in the analysis is always an appropriate model of three-dimensional nonlinear elasticity, and the key feature for the resulting theory is the energy scaling with respect to the \OOO thickness \EEE parameter of the initial domain. Despite the longstanding interest in such questions \cite{Antman1, Antman2}, early results were usually relying on some a priori \textit{ansatzes}, often leading to contrasting theories. Over the last decades though, modern techniques from the Calculus of Variations and Applied Analysis have been  implemented  very successfully for the rigorous derivation of effective models for thin elastic objects. From a technical point of view, the fundamental ingredient to perform these rigorous justifications has been the celebrated geometric rigidity estimate in the seminal work by {\sc G.~Friesecke, R.D.~James}, and {\sc S.~M\"uller}  \cite{friesecke2002theorem}, which has had a striking number of applications in dimension-reduction problems in the context of pure (hyper)elasticity. The interested reader is referred for instance to \OOO\cite{ContiDolzmann, DelgadoSchmidt, friesecke2002theorem, hierarchy, HornungNeukammVelcic, LewickaLucic, Mora4, Mora, MoraMullerSchultz,  schmidt_atomistic_dimension_reduction, schmidt_multilayers}\EEE, for a by far non-exhaustive list of references   regarding   dimension-reduction results related to plate or rod theories in the bending regime.  
	
However, concerning the investigation of phenomena   beyond the perfectly elastic regime, for instance the behavior of solids with defects and impurities  such as  \textit{plastic slips, dislocations}, \textit{cracks}, or \textit{stress-induced voids}, the mathematical understanding is far less well settled. In the case of thin elastic materials with voids, the natural variational formulation involves energies driven by the competition between bulk elastic and interfacial energies of perimeter type. Such variational models describe \textit{stress driven rearrangement instabilities} (SDRI) in elastic solids, and have recently been a focal point of considerable attention both from the mathematical and the \OOO physics \EEE community, see for example \cite{BonCha02, BraChaSol07, CrismaleFriedrich, FonFusLeoMil11, KFZ:2021, KFZ:2022, KFZ:2023, GaoNix99, Grin86, Grin93, KhoPio19,  KrePio19, SantiliSchmidt-old,  SieMikVoo04}. 
	
As far as \OOO dimension-reduction \EEE results beyond the purely elastic setting are concerned, we now present a concise summary of some important recent developments. In the framework of plasticity, we refer the reader, e.g., to  \cite{elisa2, elisa1, elisa3, liero-mielke, maggiani}. Regarding models for brittle  \MMM fracture, \EEE despite a significant recent progress on brittle plates and shells in the linear setting  \cite{Almi-etal, almitasso,  Baba-dim2, ginsglad}\EEE, the   theory in the nonlinear framework is mainly restricted to static and evolutionary models in the membrane regime \cite{solo-mem,  Baba-dim, Braides-fonseca}. Smaller energy regimes are less well studied, the only rigorous available result for now appearing to be for a two-dimensional thin brittle beam \cite{schmidt2017griffith}. The main result in that work is the derivation of an effective \textit{Griffith-Euler-Bernoulli} energy defined on the midline of the possibly fractured beam, which takes into account possible jump discontinuities of the limiting deformation and its derivative. From a technical perspective, the key tool in \cite{schmidt2017griffith} is an appropriate generalization of \cite{friesecke2002theorem}, namely a \textit{quantitative piecewise geometric rigidity theorem for SBD functions} \cite{friedrich_rigidity}. Up to now, such a general result is  available only in two dimensions, which is the main obstruction for the generalization of \OOO dimension-reduction \EEE results to settings of three-dimensional fracture.  Let us mention, however, that similar rigidity results in higher dimensions  are available in models for nonsimple materials \cite{friedrich_nonsimple}, where a singular perturbation term depending on the second gradient of the deformation is incorporated in the elastic energy.  
	
In the setting of SDRI models for solids with material voids, a first analysis  on plate \OOO theories \EEE with surface discontinuities in the bending (Kirchhoff) energy regime has been performed in \cite{SantiliSchmidt2022}. However, the results therein are conditional in two aspects. Firstly, only voids with restrictive assumptions on their distribution and geometry (satisfying the so-called \textit{minimal droplet assumption}) are considered,  which allows to resort to the classical rigidity theorem of \cite{friesecke2002theorem}. Secondly, recovery sequences are only constructed under a specific regularity property for the  outer Minkowski-content of voids and discontinuities. In our recent work \cite{KFZ:2023}, a related result for thin rods without  restriction on the void geometry was accomplished, generalizing the results of \cite{Mora} from the purely elastic setting. The cornerstone of our approach was a novel \OOO piecewise \EEE rigidity result in the realm of SDRI-models \cite{KFZ:2021}, which is based on a curvature regularization of the surface term. 
	
The goal of the present article is  to  extend the methods used in \cite{KFZ:2023} in order to show that   the \textit{Blake-Zisserman-Kirchhoff} model of \cite{SantiliSchmidt2022} is the $\Gamma$-limit of the three-dimensional model, without restricting the void geometries and without restricting to a special class of  configurations for the construction of recovery sequences.

We now describe our setting in more detail. We consider a three-dimensional thin plate with reference configuration 
\begin{equation*}
\Omega_h  := S\times (-\tfrac{h}{2},\tfrac{h}{2})\subset \R^3 
\end{equation*} 
of thickness $0<h\ll 1$, where the midsurface is represented by a bounded Lipschitz domain $S \subset \R^2$. Variational models for thin plates describing the formation of  material voids which are not a priori prescribed, fall into the framework of \textit{free discontinuity problems} \cite{Ambrosio-Fusco-Pallara:2000}, leading to an energy of the form
\begin{equation}\label{typical_energies_1}
\mathcal F^h_{\rm{el, per}}(v,E):=\int_{\Omega_h\setminus \overline{E}} W(\nabla v)\,\mathrm{d}x+\beta_h\int_{\partial E\cap \Omega_h}\varphi(\nu_E)\,\mathrm{d}\H^2\,.
\end{equation}
Here, $E \subset \Omega_h$ represents the (sufficiently regular) void set within an elastic plate with reference configuration $ \Omega_h \subset \mathbb{R}^3$, and $v$ is the corresponding elastic deformation. The first term in \eqref{typical_energies_1} represents the nonlinear elastic energy with density $W$ (see Section \ref{model_main_result} for details), whereas the second one depends on a parameter $\beta_h>0$ and a possibly anisotropic norm $\varphi$ evaluated at the outer unit normal $\nu_E$ to $ \partial E\cap\Omega_h$. For purely expository reasons, we restrict our analysis to the isotropic case, i.e., $\varphi(\cdot)\equiv|\cdot|_{2}$, see Remark \ref{choice_of_model} for some comments.
	
At a heuristic level, it is well known that an elastic energy scaling of the order $h^3$ corresponds to the \textit{bending Kirchhoff theory}, leaving the midsurface $S$ unstretched. At the same time, the surface area of voids completely separating the plate is of order $h$. Now, depending on the choice of the parameter $\beta_h$, one can expect different limiting models, after rescaling \eqref{typical_energies_1} with $\max \lbrace h^{-3}, (\beta_h h)^{-1}\rbrace$: the case $\beta_h \gg h^{2}$ will result in a purely elastic plate model, while the choice $\beta_h \ll h^{2}$ will lead to a model of purely brittle fracture.  The critical regime $\beta_h \sim h^{2}$ is the most interesting and mathematically most challenging one, since the elastic and surface contributions in this case compete at the same order.

Hence, from here on we set for simplicity $\beta_h:=h^{2}$. After rescaling the total energy in \eqref{typical_energies_1} by $h^{-3}$, our aim is to rigorously derive effective two-dimensional theories by means of $\Gamma$-convergence   \cite{Braides:02, DalMaso:93}. \MMM As \EEE in \cite{KFZ:2023}, the presence of a priori unprescribed voids in the model hinders the use of the classical rigidity result \cite{friesecke2002theorem}. Indeed, the distribution of voids in the material might possibly exhibit highly complicated geometries, for instance densely packed thin spikes or microscopically small components with small surface measure on different length scales,  see Figure~\ref{fig:spikes}.

\begin{figure}[htp]
\begin{tikzpicture}
			
\draw[clip] plot [smooth cycle] coordinates {(0,0)(2,0)(2,2)(0,2)};
			
\begin{scope}[shift={(1,1.75)},scale=.25]
\draw[fill=gray, domain=-.5:.5, smooth, variable=\x, xscale=.5,yscale=1] plot ({\x}, {40*\x*\x*\x*\x});
\end{scope}
			
\begin{scope}[shift={(.6,1.7)},scale=.2]
\draw[fill=gray, domain=-.8:.8, smooth, variable=\x, xscale=.5,yscale=1] plot ({\x}, {40*\x*\x*\x*\x});
				
\end{scope}
			
\begin{scope}[shift={(1.6,1.4)},scale=.1]
\draw[fill=gray, domain=-1.8:1.8, smooth, variable=\x, xscale=.5,yscale=1] plot ({\x}, {40*\x*\x*\x*\x});
\end{scope}
			
\begin{scope}[shift={(1.3,1.6)},scale=.1]
\draw[fill=gray, domain=-1.8:1.8, smooth, variable=\x, xscale=.5,yscale=1] plot ({\x}, {40*\x*\x*\x*\x});		
\end{scope}
			
\begin{scope}[shift={(1,.85)},scale=.1,rotate=170]
				
\draw [fill=gray] plot [smooth cycle] coordinates {(0,0) (1,1) (3,1) (1,0) (2,-1)};
				
\end{scope}
			
\begin{scope}[shift={(.5,.8)},scale=.1]
\draw [fill=gray] plot [smooth cycle] coordinates {(0,0) (1,1) (3,1) (1,0) (2,-1)};				
				
\draw [fill=gray] plot [smooth cycle] coordinates {(10,2) (11,3) (12,1) (11,-2) (10,0)};
							
\end{scope}
			
\begin{scope}[shift={(-.2,.8)}]
				
\draw[fill=gray](1.2,.58) circle(.01);
\draw[fill=gray](1.22,.48) circle(.015);
\draw[fill=gray](1,.5) circle(.04);
\draw[fill=gray](1.4,.45) circle(.03);
\draw[fill=gray](1.25,.23) circle(.02);
				
\draw[fill=gray](1.8,.35) circle(.025);
\draw[fill=gray](1.15,.23) circle(.015);
						
\draw[fill=gray](1.35,.63) circle(.02);
				
\end{scope}
					
\begin{scope}[shift={(.5,.8)}]				
\draw[fill=gray](1,.5) circle(.05);
\draw[fill=gray](1.3,.55) circle(.03);
\draw[fill=gray](1.25,.43) circle(.02);
\draw[fill=gray](1.35,.63) circle(.02);
				
\end{scope}

\begin{scope}[shift={(1.45,1.6)},scale=.14]
\draw[fill=gray, domain=-1.8:1.8, smooth, variable=\x, xscale=.5,yscale=1] plot ({\x}, {40*\x*\x*\x*\x});
\end{scope}
			
\draw plot [smooth cycle] coordinates {(0,0)(2,0)(2,2)(0,2)};

\draw[fill=gray](1,.5) circle(.05);
\draw[fill=gray](1.2,.3) circle(.04);
\draw[fill=gray](1.1,.45) circle(.035);
\draw[fill=gray](1.3,.55) circle(.03);
\draw[fill=gray](1.25,.43) circle(.02);
\draw[fill=gray](1.15,.63) circle(.02);
\draw[fill=gray](1.12,.53) circle(.01);
\draw[fill=gray](1.2,.58) circle(.01);
\draw[fill=gray](1.22,.48) circle(.015);
			
\end{tikzpicture}
\caption{Densely packed thin spikes and microscopically small components leading to loss of rigidity. For simplicity, the figure illustrates a two-dimensional example.}
\label{fig:spikes}
\end{figure}
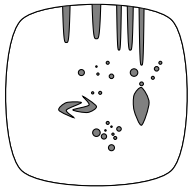
	
To remedy these phenomena, motivated by our recent works \cite{KFZ:2021, KFZ:2023}, we introduce a \textit{curvature regularization} of the form
\begin{equation}\label{vanishing_curvature_term}
\F^h_{\rm{curv}}(E):= h^{2}\kappa_h \int_{\partial E\cap \Omega_h}|\bm{A}|^{2}\,\mathrm{d}\H^2\,,
\end{equation}
where $\bm{A}$ denotes the second fundamental form of $\partial E\cap \Omega_h$ and   $\kappa_h$ satisfies  $ \kappa_h\to 0^{+}$ as $h\to 0^{+}$ at a \OOO sufficient rate \EEE(see  \eqref{rate_1_gamma_h} for details). The presence of this additional Willmore-type energy penalization allows us to use the \textit{piecewise rigidity estimate} \cite[Theorem 2.1]{KFZ:2021} in the analysis. It is a \textit{singular perturbation} for the void set $E$ and not for the deformation $v$, i.e., no higher-order derivatives of $v$ are involved in the model.  We also refer the interested reader to \cite{KFZ:2022}, where a related atomistic model is studied and additional explanations for the presence of a microscopic analog of the term in \eqref{vanishing_curvature_term} are given, see \cite[Subsection 2.5]{KFZ:2022}. As mentioned therein, curvature regularizations of similar type are commonly used in the mathematical and physical literature of SDRI models, for instance in the description of heteroepitaxial growth of elastically stressed thin films or material voids, see \cite{AnGurt,  FonFusLeoMor14, FonFusLeoMor15, GurtJabb, Herr,   voigt,  SieMikVoo04}. Despite the possible modeling relevance, we mention that the presence of the curvature contribution in our model is only for mathematical reasons as a regularization term. In particular, it does not affect the structure of the effective limiting model.

The total energy of a pair $(v,E)$ is then given by the sum of the two terms in \eqref{typical_energies_1} and \eqref{vanishing_curvature_term}, i.e.,  
\begin{equation*}
\mathcal F^h(v,E):=\mathcal F^h_{\rm{el, per}}(v,E)+\mathcal F^h_{\rm{curv}}(E)\,,
\end{equation*}
having set  $\beta_h := h^{2}$ and  $\varphi(\nu)\equiv 1$ for all $\nu\in \S^2$. The main outcome of this work is then Theorem \ref{thm:main_kirchhoff_gamma_conv}, where we show that the rescaled energies $\big(h^{-3}\F^h(\cdot,\cdot)\big)_{h>0}$ $\Gamma$-converge in an appropriate topology to an effective two-dimensional model that is of the form
\begin{equation}\label{limiting_energy_intro}
\frac{1}{24}\int_{S\setminus V}\Q_2\big(\mathrm{II}_y(x')\big)\,\mathrm{d}x'+\H^1\big(\partial^*V\cap S\big)+2\H^1\big(J_{(y, \nabla'y)}\setminus \partial^*V\big)\,. 
\end{equation}
Here, $V\subset S$ denotes a set of finite perimeter in $S\subset \mathbb{R}^2$   and represents the void part in the limiting two-dimensional plate. As in \cite{SantiliSchmidt2022}, the limiting admissible deformations turn out to be \textit{possibly fractured and creased flat isometric immersions}, i.e., $y\in SBV^{2,2}_{\mathrm{isom}}(S;\mathbb{R}^3)$, see Section \ref{subsec:main_results} for precise definitions. The approximate gradient of $y$ is denoted here by $\nabla'y$ and $\mathrm{II}_y$ denotes the induced second fundamental form on $y(S)$. The density of the limiting elastic energy, which should be conceived as a curvature energy on isometric immersions of the \OOO midsurface \EEE $S$, depends on a quadratic form $\Q_2$ which is defined through the quadratic form $D^2W({\rm Id})$ of linearized elasticity via a suitable minimization problem, see \eqref{def:Q_2} for its precise definition. The second term in \eqref{limiting_energy_intro} accounts for the presence of limiting voids by measuring the total length of their boundaries. The last term in \eqref{limiting_energy_intro} takes the fact  into account that, in the limit, voids might collapse into \OOO discontinuities \EEE of the limiting deformation $y$ or its derivative $\nabla'y$, corresponding to cracks or folds of the limiting plate, respectively. Exactly due to their origin, the length of those should be counted twice in the energy.
	
Let us once again mention that one fundamental difference between our work and that of \cite{SantiliSchmidt2022} lies in the assumptions on the admissible void sets. While we allow for voids with general geometry employing a mild curvature regularization, \cite{SantiliSchmidt2022} is based on specific restrictive assumptions on the void geometry, namely the  so-called \textit{$\psi$-minimal droplet assumption}, cf.~\cite[Equation (6)]{SantiliSchmidt2022}. This can be interpreted as an $L^{\infty}$-diverging bound on the curvature of the boundary of the voids in the initial thin plate. In our setting, the \MMM curvature penalization term \eqref{vanishing_curvature_term} \EEE can be thought of as imposing an $L^2$-diverging bound on the curvature: its nature allows for the voids to concentrate at arbitrarily small scales (independently  of $h$), while also allowing for a diverging (with $h$) number of connected components of voids, see Remark~\ref{choice_of_model}(iv). Of course, our more general model comes at the expense of using more sophisticated geometric rigidity estimates \cite{KFZ:2021} compared to the classical one of \cite{friesecke2002theorem}. As in \cite{KFZ:2023}, our strategy relies on modifying the deformations and their gradients on a small part of the domain  such that  the new deformations  are actually Sobolev except for the boundaries of a controllable number of cubes, with a good control on the elastic energy. Concerning \OOO the \EEE derivation of compactness and $\Gamma$-liminf, the estimate for the surface parts  in \eqref{limiting_energy_intro} is the most delicate step and requires a fine control on the jump height of these modifications. By means of a \OOO contradiction argument based on a \EEE blow up method, we are able to reduce the problem to the setting of thin rods, which allows us to directly use \OOO the compactness result from  \cite[Theorem~2.1]{KFZ:2023}.\EEE 

Besides the geometry of voids, our work differs from \cite{SantiliSchmidt2022} by the fact that therein  recovery sequences are only provided under specific regularity properties for the  outer Minkowski-content of $\partial^* V$ and $J_{(y, \nabla' y)}$. By means of the coarea formula, the latter assumption allows to construct a sequence of three-dimensional voids with the required regularity, in particular satisfying the minimal droplet assumption. In the general case, however, a density result for boundaries of void sets and jump sets of $SBV$-functions appears to be required. Although many results are available in this direction, see, e.g., \cite{Cortesani-Toader:1999}, to the best of our knowledge they are all incompatible with the isometry constraint, i.e., with  $y\in SBV^{2,2}_{\mathrm{isom}}(S;\mathbb{R}^3)$. As approximation results are usually built on convolution techniques, it indeed cannot be expected to obtain a density result satisfying exactly an isometry constraint. Yet, we are able to obtain a density result which can control the deviation from an isometry in a controlled way, quantified in terms of the thickness $h$. This is then enough to construct recovery sequences by adapting the \emph{ansatz} from the purely elastic case \cite{friesecke2002theorem}. We regard this part as the most original technical novelty of the current paper, and we  believe that the technique may be applicable also in other related settings.

\subsection{Organization of the paper}\label{plan_of_paper} 
Our paper is organized as follows. In Section \ref{model_main_result} we introduce the model and state the main compactness and $\Gamma$-convergence results, i.e., Theorems \ref{thm:main_kirchhoff_cptness} and \ref{thm:main_kirchhoff_gamma_conv}, respectively, together with some additional \MMM  modeling \EEE remarks. In Section~\ref{preparatory_modifications}, we collect the necessary technical ingredients for the proofs, namely   the nonlinear piecewise rigidity estimates from \cite{KFZ:2021} and the Korn-Poincar{\'e} inequality for $SBV^2$-functions with small jump set from \cite{Cagnetti-Chambolle-Scardia} adapted to our setting, \OOO see \EEE Subsection \ref{se: rigi}, as well as the construction of almost Sobolev replacements for sequences of deformations with equibounded energy, \OOO see \EEE Subsections \ref{sec: locest}--\ref{sec: global_constructions_proofs}. Section \ref{compactness} contains the proof of Theorem \ref{thm:main_kirchhoff_cptness}, while Sections \ref{gamma_liminf} and \ref{sec: gamma_limsup} contain the proof of Theorem~\ref{thm:main_kirchhoff_gamma_conv}(i),(ii) respectively. Finally, in Appendices \ref{sec: aux estimates} and \ref{sec: standard_liminf} we give the proofs of some auxiliary facts that are themselves not new but \GGG that are \EEE presented here for completeness only.
	
\subsection{Notation}\label{Notation}
We close the introduction with some basic notation. Given \OOO$d\in \N$, \EEE $U\subset \R^{\OOO d\EEE}$ open, we denote by $\mathfrak{M}(U)$ the collection of all measurable subsets of $U$, and by $\mathcal{P}(U)$ the one of subsets of finite perimeter in $U$.  Given $A, B \in \MMM \mathfrak{M}(U) \EEE$, we write $\chi_A$ for the characteristic function of $A$, $A \triangle B$ for their symmetric difference, $A\subset \subset B$ iff $\overline{A} \subset B$, and $\mathrm{dist}_{\H}(A,B)$ for the Hausdorff distance between $A$ and $B$. For $v \in \mathbb{R}^d$ we denote by  $\MMM |v|_\infty \EEE := \max\{|v_k| \colon k=1,\dots,d\}$ its $\ell_\infty$-norm, while for the $|\cdot|_\infty$-distance of a point $x$ \GGG (respectively a set $B$)  to a set $A$, we write ${\rm dist}_\infty(x,A)$ \GGG (respectively $\mathrm{dist}_\infty(A,B)$)\OOO. For every $A\subset \R^d$ and $\delta>0$, we \MMM define \EEE  
\begin{equation}\label{neigh-def}
(A)_\delta:=\{x\in \R^d\colon \mathrm{dist}(x,A)<\delta\}\,.
\end{equation} 
For $E\in \mathcal{P}(U)$   we denote by  $\partial^*E$ the essential boundary of $E$, see \cite[Definition~3.60]{Ambrosio-Fusco-Pallara:2000}. \OOO For $d=3$ we \EEE also denote by $\mathcal{A}_{\rm reg}(U)$ the collection of all open sets $E \subset U$ such that $\partial E \cap U $ is a two-dimensional $C^2$-surface in $\mathbb{R}^3$. Surfaces and functions of $C^2$-regularity will be called $C^2$-regular or just regular in the following. For $E \in \mathcal{A}_{\rm reg}(U)$ we denote by $\bm{A}$ the second fundamental form of $\partial E\cap U$, i.e., $|\bm{A}| = \sqrt{\kappa_1^2 + \kappa_2^2}$, where $\kappa_1$ and $\kappa_2$ are the corresponding principal curvatures.  By $\nu_{E}$ we indicate the outer unit normal to  $\partial E\cap U$. 

 The inner product of two vectors $a,b\in \R^3$ will be denoted by $a\cdot b$, and their exterior product by $a\wedge b$. We further write 
$\mathbb{S}^2:= \lbrace \nu \in \R^3\colon \, |\nu|=1\rbrace$.
By ${\rm id}$   we denote the identity mapping on $\R^3$ and by ${\rm Id} \in \R^{3\times 3}$  the identity matrix.  For each $F \in \R^{3 \times 3}$ we let ${\rm sym}(F) := \frac{1}{2}\left(F+F^T\right)$
and we also \MMM introduce \EEE
$SO(3) := \lbrace F\in\R^{3 \times 3}\colon F^TF = {\rm Id}, \, \det F = 1\rbrace$. 
Moreover, we denote by $\R^{3\times 3}_{\rm sym}$ and $\R^{3 \times 3}_{\rm skew}$ the space of symmetric and skew-symmetric matrices, respectively. 
For $\sigma>0$, we denote by $T_\sigma$ the linear transformation in $\R^3$ with matrix representation given by 
\begin{equation}\label{anisotropic_dilation}
T_\sigma:=\mathrm{diag}(1,1,\sigma)
\end{equation} 
with respect to the canonical basis $\{e_1,e_2,e_3\}$ of $\mathbb{R}^3$. For $d,k\in \N$, we indicate by $\mathcal{L}^d$ and $\mathcal{H}^{k}$ the $d$-dimensional Lebesgue measure and the $k$-dimensional Hausdorff measure, respectively.  
\EEE

For \MMM  $U \subset \R^n$ open, for \EEE $p\in [1,\infty]$ and $d, k \in \N$ we denote by $L^p(U;\R^d)$ and $W^{k,p}(U;\R^d)$ the standard Lebesgue and Sobolev spaces, respectively.  Partial derivatives of a function $f\colon U \to \R^d$ will be denoted by \OOO $(\partial_if)_{i=1,2,3}$\EEE. We use standard notation for $SBV$-functions, cf.~\cite[Chapter~4]{Ambrosio-Fusco-Pallara:2000} for the definition and a detailed presentation of the properties of this space. In particular,  for a function $u\in SBV(U;\R^d)$,  we write $\nabla u$ for the approximate gradient, $J_u$ for its jump set, and $u^{\pm}$ for the one-sided traces on $J_u$. 
Finally, 
\begin{equation*}
SBV^2(U;\R^d):=\Big\{u\in SBV(U;\R^d)\colon \int_{U}|\nabla u|^2\,\mathrm{d}x+\H^{d-1}(J_u\cap U)<+\infty\Big\}\,.
\end{equation*}

\section{The models and the main results}\label{model_main_result}

In this section we introduce the model and present our main results. \EEE

\subsection{The three-dimensional model}\label{rescaling_of_models} 
We  \EEE denote the reference configuration of the thin plate by  
\begin{align}\label{eq: Omegah}
\Omega_{h}:=S\times (-\tfrac{h}{2},\tfrac{h}{2})\subset \mathbb{R}^3\,,
\end{align}
where $S\subset \mathbb{R}^2$ is \MMM a bounded \EEE Lipschitz domain describing the midsurface of the thin plate, and $0<h\ll 1$ denotes its infinitesimal thickness. For a large but fixed constant $M\gg 1$, the set of \emph{admissible pairs} of function and set is given by 
\begin{equation}\label{initial_admissible_configurations}
\mathcal{A}_h:=\left\{(v,E)\colon\ E\in \mathcal{A}_{\rm reg}(\Omega_h),\ v\in \NNN W^{1,2}\EEE(\Omega_h\setminus \overline{E};\mathbb{R}^3)\,,\ v|_E\equiv\mathrm{id}\,,\ \|v\|_{L^\infty(\Omega_h)} \le   M\right\}\,.
\end{equation}
The third condition in \eqref{initial_admissible_configurations} is for definiteness only. While the last one therein is merely of technical nature to ensure compactness, it is also justified from a physical viewpoint since it corresponds to the assumption that the solid under consideration is confined in a bounded region. For each pair $(v,E) \in \mathcal{A}_h$,  we consider the energy
\begin{align}\label{initial_energy}
\F^{h}(v,E):= 
\int_{\Omega_h\setminus \overline{E}} W(\nabla v)\,\mathrm{d}x+h^2 \H^2(\partial E\cap \Omega_h)+  h^2\kappa_h  \int_{\partial E\cap \Omega_h}|\bm{A}|^2\,\mathrm{d}\mathcal{H}^2\,.
\end{align}
The first two terms correspond to the \emph{elastic} and the \emph{surface energy} of perimeter-type, respectively, while the third one is a \emph{curvature regularization} of Willmore-type, where $\bm A$ denotes the second fundamental form of $\partial  E\cap \Omega_h$ and $\kappa_h$ is a suitable infinitesimal parameter specified in \eqref{rate_1_gamma_h} below. The factor $h^2$ in front of the surface terms ensures that the elastic and the surface energy are of same order, since the elastic energy per unit volume is of order $h^{2}$, see the introduction for some heuristic explanations of the model.

The function $W \colon \mathbb{R}^{3\times 3}\to\mathbb{R}_+$ in \eqref{initial_energy} represents the \emph{stored elastic energy density}, satisfying   standard assumptions of nonlinear elasticity. In particular, we suppose that $W\in C^0(\R^{3\times 3};  \R_+ )$ satisfies 
\begin{align}\label{eq: nonlinear energy}
\begin{split}
\rm{(i)} & \ \  \text{Frame indifference: $W(RF) = W(F)$ for all $R \in SO(3)$ and $F\in \mathbb{R}^{3\times 3}$}\,,\\ 
\rm{(ii)} & \ \ \text{Single energy-well structure:}\ \{W=0\} \equiv SO(3)\,,\\
\rm{(iii)} & \ \ \text{Regularity:}  \  \text{$W$ is $ C^2$\OOO-regular \EEE in a neighborhood of $SO(3)$}\,,\\
\rm{(iv)} & \ \ \text{Coercivity:}  \   \text{There exists $c>0$ such that for all $F \in \mathbb{R}^{3\times 3}$ it holds that} \\ 
& \quad   \quad \quad  \quad \quad \quad \, W(F) \geq c\, \mathrm{dist}^2(F,SO(3))\,,\\
\rm{(v)} & \ \ \text{Growth condition:}  \   \text{There exists $C>0$ such that for all $F \in \mathbb{R}^{3\times 3}$ it holds that} \\ 
& \quad   \quad \quad  \quad \quad \quad \, W(F) \le  C\, \mathrm{dist}^2(F,SO(3))\,.  \EEE
\end{split}
\end{align}
We note that  condition (v) excludes the natural assumption $W(F) \to +\infty$ as $\det F \to 0^+$, but it is needed in our analysis for the construction of recovery sequences. The choice of an isotropic perimeter energy is purely for simplicity of the exposition, \MMM and \EEE more general anisotropic perimeters can be chosen in the model without substantial   changes in the proofs, see also Remark \ref{choice_of_model} below. As for the parameter $\kappa_h>0$ in the curvature regularization, we require  
\begin{equation}\label{rate_1_gamma_h}
\kappa_h h^{-2}\to 0\,, \quad \kappa_h  h^{-52/25} \to +\infty \quad \text{ as $h\to 0$}\,. 
\end{equation}
Similarly to its role in \cite[\OOO Equation (2.5)\EEE]{KFZ:2023}, it is a technical assumption that has been chosen for simplicity rather than optimality. Its choice is related to the application of suitable piecewise rigidity results \cite{KFZ:2021} \MMM and Korn inequalities \EEE  \cite{Cagnetti-Chambolle-Scardia}, and will become apparent along the proof, see in particular~\eqref{parameters_for_uniform_bounds}.

As is by now customary in  dimension-reduction problems, we perform an anisotropic change of variables to reformulate the \MMM problem \EEE in a fixed reference domain: recalling \eqref{anisotropic_dilation}, we rescale our variables and set 
\begin{equation}\label{order_1_domain}
\Omega:=\Omega_1\,, \quad V:=T_{1/h}(E)=\{x\in \Omega: \OOO T_h x\EEE
\in E\}\,.
\end{equation}
Accordingly, the rescaled deformations are defined by $y\colon\Omega\to \R^3$ via
\begin{equation}\label{from_v_to_y}
y(x)\OOO:=v(T_hx)\EEE\,.
\end{equation}
The total energy is rescaled by a factor $h^3$, hence we set 
\begin{equation}\label{rescaled_energy}
\mathcal{E}^{h}(y,V):=h^{-3}\F^{h}(v,E)\,,
\end{equation}
where the pair $(y,V)$ is related to $(v, E)$ via \eqref{order_1_domain} and \eqref{from_v_to_y}. In this rescaling,  one  factor $h$  corresponds to the change of volume  and the other factor $h^{2}$  corresponds to the average elastic energy per unit volume in the bending regime for plates.  

The corresponding rescaled gradient will be denoted as usual by 
\begin{equation*}
\nabla_hy(x):=\Big(\partial_1 y,\partial_2 y, \frac{1}{h}\partial_3y\Big)(x)=\nabla v(\OOO T_hx\EEE)\,. 
\end{equation*}
Therefore, by a change of variables we find  
\begin{align}\label{eq: newenergy}
\mathcal{E}^{h}(y,V) = h^{-2}\int_{\Omega\setminus \overline{V}} W(\nabla_hy(x))\,\mathrm{d}x+\int_{\partial V\cap \Omega}\big|\big(\nu^1_{V}(z), \nu^2_{V}(z), h^{-1}\nu^3_{V}(z)\big)\big|\, \mathrm{d}\H^2(z) + \mathcal{E}^h_{\rm  curv}(V) \,,
\end{align}
where $\nu_{V}(z):=\big(\nu^1_{V}(z),\nu^2_{V}(z), \nu^3_{V}(z)\big)$ denotes the outer unit normal to $\partial V\cap \Omega$ at the point $z$. Note that for the rescaling of the perimeter part of the energy, one can test with smooth functions and use the divergence theorem, as we also commented in \cite[Equation (2.10)]{KFZ:2023}.

\EEE

Regarding the term $\mathcal{E}^h_{\rm  curv}(V)$ which denotes the curvature-related energetic contribution for the rescaled set $V$, its precise expression after the change of variables will not be of \OOO specific \EEE 
use in the subsequent analysis. Hence, we refrain from \MMM giving its explicit form.  \EEE

In view of \eqref{anisotropic_dilation} and  \eqref{initial_admissible_configurations}, the space of rescaled admissible pairs \OOO of \EEE deformations and voids is given by
\begin{equation}\label{admissible_configurations_h_level}
\hat{\mathcal{A}}_h:=\big\{(y,V)\colon V\in \mathcal{A}_{\mathrm{reg}}(\Omega)\,,\ y\in \NNN W^{1,2}\EEE(\Omega\setminus \overline{V}; \R^3)\,, \ y|_{V}\equiv T_h(\mathrm{id})\,, \ \|y\|_{L^\infty(\Omega)}\leq M\big\}\,. 
\end{equation}
\subsection{Limiting model and main result}\label{subsec:main_results}
As in the purely elastic case considered in \cite{friesecke2002theorem}, the density of the limiting elastic energy will depend on the quadratic form $\Q_3\colon \R^{3\times 3}\to \R$, which is defined as 
\begin{equation}\label{def:Q_3}
\Q_3(G):=D^2W(\mathrm{Id})[G,G]\,.
\end{equation}
Due to \eqref{eq: nonlinear energy}, $\mathcal{Q}_3$ vanishes on $\R^{3\times 3}_{\mathrm{skew}}$ and is strictly positive-definite on $\R^{3\times 3}_{\mathrm{sym}}$. We also define $\Q_2\colon\R^{2\times 2}\to \R$ as
\begin{equation}\label{def:Q_2}
\Q_2(A):= \min_{c\in \R^3} \Q_3(\hat A+c\otimes e_3)\,,
\end{equation}
by minimizing over stretches in the $x_3$-direction. Here, for a matrix $A \in \R^{2 \times 2}$ we denote by $\hat A$ its extension to a $(3\times 3)$-matrix by adding zeros in the third row and column.

As in \cite{SantiliSchmidt2022}, where a similar model under more restrictive geometric assumptions on the voids was studied, the  limiting energy will be defined on the space
\begin{equation}\label{limiting_admissible_pairs}
\begin{split}
\hspace{-0.7em}\mathcal{A}&:=\Big\{(y,V)\in SBV_{\mathrm{isom}}^{2,2}(S;\R^3)\times\mathcal{P}(S)\colon y|_V=\mathrm{id}|_V\,,   \|y\|_{L^\infty(\Omega)}\leq M\Big\}\,,
\end{split}
\end{equation} 
where 
\begin{equation*}
SBV_{\mathrm{isom}}^{2,2}(S;\R^3):=	\{y\in SBV^2(S;\R^3)\colon \nabla'y\in SBV^2(S;\R^{3\times 2}), (\nabla'y,\OOO \partial_1 y\EEE\wedge \OOO \partial_2y \EEE)\in SO(3)\ \text{a.e.}\}\,,
\end{equation*}
and $\nabla'y:=(\OOO \partial_1 y, \partial_2y\EEE)$, \MMM i.e., the \EEE limiting admissible deformations are isometric away from the jump set
\begin{equation}\label{eq:jump_set}
J_{(y,\nabla'y)}:=J_{y}\cup J_{\nabla'y}\,.
\end{equation}
\MMM We \EEE denote by \EEE $\tilde y\colon \Omega \to \mathbb{R}^3$  maps of the form 
\begin{equation}\label{eq:y_ident}
\tilde y(x)=y(x_1,x_2) \ \forall \, x \in \Omega\,, \ \text{for some } y\in SBV^{2,2}_{\mathrm{isom}}(S;\R^3)\,,
\end{equation}
and, similarly, by  $\tilde V\subset \Omega$  sets of the form  
\begin{equation}\label{eq:V_ident}
\tilde V= V\times \left(-\frac{1}{2}, \frac{1}{2}\right) \ \ \text{for some } V\in \mathcal{P}(S)\,.
\end{equation} 
In what follows, the pair $(\tilde y,\tilde V)$ will always be associated to $(y,V)$ via \eqref{eq:y_ident} and \eqref{eq:V_ident} whenever it appears. \MMM For \EEE  a mapping $y\in SBV^{2,2}_{\mathrm{isom}}(S;\R^3)$, we also introduce its second fundamental form via
\begin{equation}\label{def:second_FF}
\mathrm{II}_y:=\big(\OOO \partial_iy\cdot\partial_j(\partial_1y\wedge \partial_2y)\EEE\big)_{1\leq i,j\leq 2}\,.
\end{equation}
With  the definitions \eqref{def:Q_2} and \eqref{eq:jump_set} in mind, for each $(y,V)\in \mathcal{A}$, the limiting two-dimensional energy of Blake-Zisserman-Kirchhoff-type (cf.~also \cite{Blake-Zisserman, Carriero-Leaci-BZ, SantiliSchmidt2022}) is defined as
\begin{equation}\label{limiting_two_dimensional_energy}
\mathcal{E}^{0}\big(y,V\big):=
\frac{1}{24}\int_{S\setminus V}\Q_2(\mathrm{II}_y(x'))\,\mathrm{d}x'+\H^1\big(\partial^*V\cap S\big)+2\H^1\big(J_{(y, \nabla'y)}\setminus \partial^*V\big)\,. 
\end{equation}

As mentioned also in the introduction, the limiting two-dimensional model features the classical bending-curvature energy derived in \cite{friesecke2002theorem} and two surface terms related to the presence of voids. The middle term on the right-hand side of \eqref{limiting_two_dimensional_energy} corresponds to the energy contribution of the limiting void $V$, whereas the last one therein is associated to discontinuities or folds of the deformation, represented by $J_y$ and $J_{\nabla'y}$, respectively. The origin of this term is due to the fact that voids may collapse into discontinuity curves in the limit, and thus appears with a factor $2$, see Figure~\ref{fig:dimred}.

\begin{figure}[htp]
\includegraphics[width=1\linewidth]{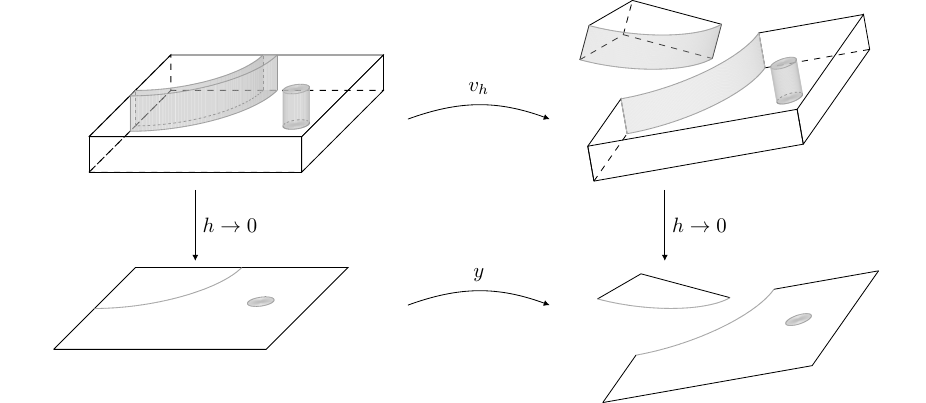}
\caption{Bending a plate with voids \MMM and cracks: \EEE Collapsing voids lead to discontinuity curves for $J_y$. \MMM Folds corresponding to the presence of $J_{\nabla'y}$ are  not depicted for simplicity.}
\label{fig:dimred}
\end{figure} 

\EEE

With these definitions and notations, our main results can be summarized as follows.

\begin{theorem}\label{thm:main_kirchhoff_cptness}
$\text{(\underline {Compactness})}$ Let $(h_j)_{j\in \N}\subset (0,\infty)$ with $h_j\searrow 0$ and $(y_{h_j}, V_{h_j}) \in \hat{\mathcal{A}}_{h_j}$ (cf.~\eqref{admissible_configurations_h_level}) be such that 
\begin{equation}\label{uniform_rescaled_energy_bound}
\sup_{j\in \mathbb{N}}\E^{h_j}(y_{h_j},V_{h_j})<+\infty\,. 
\end{equation}
Then, there exists $(y,V)\in \mathcal{A}$ (cf.~\eqref{limiting_admissible_pairs}) such that, up to a non-relabeled subsequence,
\begin{align}\label{compactness_properties}
\begin{split}
\rm{(i)}\quad&\chi_{V_{h_j}}\longrightarrow \chi_{\tilde V}\ \text{ in } L^1(\Omega)\,,\\
\rm{(ii)}\quad& y_{h_j}\longrightarrow \tilde y  \text{ in  } L^1(\Omega;\R^3)  \,,\\
\rm{(iii)}\quad& \nabla_{h_j}y_{h_j}\longrightarrow  \big(\nabla'\tilde y,\OOO \partial_1 \tilde y\wedge \partial_2\tilde y\EEE\big)  \ \text{ strongly in } L^2(\Omega;\R^{3\times 3}) \,,
\end{split}
\end{align} 
where  $(\tilde y, \tilde V)$ is associated to $(y,V)$ via \eqref{eq:y_ident} and \eqref{eq:V_ident}.
\end{theorem}

\begin{definition}\label{eq:tau_convergence}
We say that $(y_{h_j},V_{h_j})\overset{\tau}{\longrightarrow}\big(y,V\big)$ as $j\to \infty$ if and only if  \eqref{compactness_properties} holds. 
\end{definition}

Note that \eqref{admissible_configurations_h_level} implies that $\sup_{j\in \N}\|y_{h_j}\|_{L^\infty(\Omega)}\leq M$, and therefore the convergence in \eqref{compactness_properties}(ii) actually holds in $L^p(\Omega;\R^3)$ for every $p\in [1,+\infty)$. We are now ready to state the main $\Gamma$-convergence result.

\begin{theorem}\label{thm:main_kirchhoff_gamma_conv}
$(\underline{ {\Gamma}\text{-convergence}})$  Let $(h_j)_{j\in \N}\subset (0,\infty)$ with $h_j\searrow 0$. The sequence of functionals $(\E^{h_j})_{j\in \N}$ $\Gamma(\tau)$-converges to the functional $\E^{0}$ of \eqref{limiting_two_dimensional_energy}, i.e., the following two inequalities hold true.\\[-7pt]

\noindent $\rm{(i)}$ $(\underline{ {\Gamma}\text{-liminf inequality}})$ Whenever $(y_{h_j},V_{h_j})\overset{\tau}{\longrightarrow}\big(y,V\big)$, then
\begin{equation*}
 \GGG\liminf_{j\to +\infty}\E^{h_j}(y_{h_j},V_{h_j})\geq \E^0\big(y,V\big)\,.\EEE
\end{equation*} 
$\rm{(ii)}$ $(\underline{ {\Gamma}\text{-limsup inequality}})$ For every $\big(y,V\big)\in \mathcal{A}$ there exists $ (y_{h_j},V_{h_j})\in \hat{\mathcal{A}}_{h_j}$ for each $j \in \mathbb{N}$ such that $(y_{h_j},V_{h_j})\overset{\tau}{\longrightarrow}\big(y,V\big)$ and 
\begin{equation*}
\limsup_{j\to +\infty} \E^{h_j}(y_{h_j},V_{h_j})\leq \E^0\big(y,V\big)\,. 
\end{equation*} 
\end{theorem}

\begin{remark}[Possible extensions and variants]\label{choice_of_model}
\normalfont (i) One could   consider more general perimeter energies of the form
$$\beta_h\int_{\partial E\cap \Omega_h}\varphi(\nu_E)\, \mathrm{d}\H^2\,,$$
where $\lim_{h\to 0} (h^{-2}\beta_h)=\beta>0$ and $\varphi$ is a norm in $\R^3$. Similarly \MMM to \cite{KFZ:2023}, \EEE for simplicity of the exposition we have chosen $\beta_h := h^2$ and $\varphi$ to be the standard isotropic Euclidean norm in $\R^3$. The general case is analogous in its treatment, where the limiting surface energy in \eqref{limiting_two_dimensional_energy} is given by 
\[\int_{\OOO \partial^*V\cap S\EEE}\varphi_0(\nu_V)\,\mathrm{d}\H^2+2\int_{J_{(y,\nabla'y)}\setminus\partial^*V}\varphi_0(\nu_{J_{(y,\nabla'y)}})\,\mathrm{d}\H^2\,,\]
with 
\[\varphi_0(x_1, x_2) := \min_{c\in \R} \varphi(x_1 , x_2, c)\,,\]
see \cite[Sections 2 and 6]{SantiliSchmidt2022} for details, especially related to some technical adjustments needed in the construction of recovery sequences for the void sets in the presence of anisotropic surface energy densities.

\smallskip
(ii) Completely analogously to \cite[Remark~2.1(ii)]{KFZ:2023}, any singular perturbation of the form 
$$h^2\kappa_h\int_{\partial E\cap \Omega_h}|\bm{A}|^q\,\mathrm{d}\H^2$$
with $q\geq 2$ would be a legitimate choice for a curvature regularization, up to adjusting the condition for $\kappa_h$ in \eqref{rate_1_gamma_h} (which \OOO would \EEE then depend also on $q$).  Let us nevertheless also mention here that the choice $q \ge 2$ is essential, see \cite[Lemma \OOO 2.12 \EEE and Example \OOO 2.13\EEE]{KFZ:2021}. For simplicity, we have chosen $q=2$, which corresponds to a curvature regularization of classical Willmore-type.  

\smallskip
(iii)  We also remark that compressive boundary conditions and body forces can be included into the $\Gamma$-convergence statement. Although  we omit the details here, we refer the interested reader to \cite[Section 6]{friesecke2002theorem} for a discussion in this direction, in the \MMM  purely \EEE elastic case.

\smallskip
 (iv) Finally, \cite[Example 2.1]{KFZ:2023} can easily be  adjusted to our setting, the only difference coming from the energy rescaling by a factor $h^{-3}$ instead of $h^{-4}$ in this case. This allows  to exhibit   configurations $(v_h, E_h) \in \mathcal{A}_h$ with
\begin{align*}
\sup_{h >0} h^{-3} \mathcal{F}^h(v_h, E_h) <+\infty\,,
\end{align*}
where $E_h$ consists of balls which concentrate on arbitrarily small scales (independently of $h$), and whose number is diverging at a rate faster than $h^{-1}$. In contrast, in  the setting of \cite{SantiliSchmidt2022} and for void sets consisting of a disjoint union of balls, the minimal droplet assumption implies a lower bound of order $h$ on the radius of each ball and an upper bound of order $h^{-1}$ for the total number of balls, cf.~\cite[Remark 3.1]{SantiliSchmidt2022}\OOO.
\EEE
\end{remark}

\section{Piecewise rigidity and Sobolev modification of deformations}\label{preparatory_modifications}

In this section we collect some preparatory ingredients which are of utter importance in the proofs of the compactness result of Theorem \ref{thm:main_kirchhoff_cptness} in Section \ref{compactness} and the $\Gamma$-liminf inequality of Theorem \ref{thm:main_kirchhoff_gamma_conv} in Section \ref{gamma_liminf}. Our reasoning follows in spirit the analogous one in \cite[Section 3]{KFZ:2023} and relies on the approximation of a sequence of deformations with equibounded energy by mappings which enjoy good Sobolev bounds on a large portion (and asymptotically all) of the bulk domain of the thin plate, while still retaining a good control on their total jump set in the full domain.  This will allow us to use tools from the theory of $SBV$ functions, in particular \textit{Ambrosio's compactness and lower semicontinuity  theorems} (cf.~\cite[Theorems 4.7 and 4.8]{Ambrosio-Fusco-Pallara:2000}), together with the proof strategy from \cite{friesecke2002theorem, SantiliSchmidt2022}. 

In this and the following sections, we will use  the continuum  subscript $h>0$ instead of the sequential subscript notation $(h_j)_{j\in \N}$ purely for notational convenience. To state the main results of this section, we need to introduce some further notation. Recall the definition of $\Omega_h$ in \eqref{eq: Omegah}. As the estimates provided by the rigidity result stated in Theorem \ref{prop:rigidity}   below are only local,  we need to introduce a slightly smaller reference domain. To this end,  for every  $\rho\in (0,1)$, we define 
\begin{equation}\label{Omega_h_local}
\Omega_{h, \rho}:=\big\{x'\in S\colon \mathrm{dist}(x',\partial S)> \rho \big\}\times\big(-\tfrac{(1-\rho)}{2}h,\tfrac{(1-\rho)}{2}h\big)\,.
\end{equation}
Eventually, in Sections \ref{compactness}--\ref{gamma_liminf} we will send $\rho \to 0^+$, after the convergence $h \to 0^+$. In what follows, for every $i\in \mathbb{Z}^2$, we introduce cubes and cuboids of \OOO sidelengths proportional to $h>0$, namely \EEE
\begin{equation}\label{eq:Q_h}
Q_{h}(i) :=(hi,0)+(-\tfrac{h}{2},\tfrac{h}{2})^3\,\ \ \ \ \text{and }  \ \ \ \hat Q_{h}(i) :=(hi,0)+(-\tfrac{3h}{2},\tfrac{3h}{2})^2\times(-\tfrac{h}{2},\tfrac{h}{2})\,,
\end{equation}
which will serve  to cover $\Omega_{h}$, \OOO up to the negligible skeleton of the covering\EEE. In particular, we set  
\begin{equation*}
\mathcal{Q}_{h} := \lbrace Q_{h}(i)\colon\hat Q_{h}(i)\subset \Omega_h \rbrace\,. 
\end{equation*} 
Similarly, for every $\rho\in (0,1)$ we set 
\begin{equation}\label{eq:Q_h_rho}
Q_{h,\rho}(i) :=(hi,0)+(1-\rho)(-\tfrac{h}{2}, \tfrac{h}{2})^3\,, \quad \hat Q_{h,\rho}(i) :=(hi,0)+ (1-\rho)\big((-\tfrac{3h}{2},\tfrac{3h}{2})^2\times(-\tfrac{h}{2},\tfrac{h}{2})\big)\,.
\end{equation}
Given a measurable set $K\subset \R^3$ and $\gamma >0$, we introduce the localized version of the total surface energy as 
\begin{equation}\label{F_surf_energy}
\mathcal{G}^\gamma_{\rm{surf}}(E;K):=\mathcal{H}^2(\partial E\cap K)+ \gamma \int_{\partial E\cap K}|\bm A|^2\,\mathrm{d}\mathcal{H}^2\,, 
\end{equation}  
using a general parameter $\gamma$ in place of $\kappa_h$ for later purposes. The total rescaled energy is then given by 
\begin{align}\label{eq: G energy}
\mathcal{G}^h(v,E):= \frac{1}{h^3} \int_{\Omega_h\setminus \overline{E}} W(\nabla v)\,\mathrm{d}x+ \frac{1}{h} \mathcal{G}^{\kappa_h}_{\rm{surf}}(E;\Omega_h)=h^{-3}\mathcal{F}^h(v,E)\,,
\end{align}
for $(v,E)\in \mathcal{A}_h$, \OOO cf.~\EEE \eqref{initial_admissible_configurations}, where $\mathcal{F}^h$ is as in \eqref{initial_energy}. Note again that one factor $h$ in the rescaling of the elastic energy corresponds to the volume of $\Omega_h$, while the extra factor $h^2$ corresponds to the average elastic energy per unit volume.

\begin{proposition}[Sobolev modification of deformations and their gradients]\label{Lipschitz_replacement}
Let  $0<\rho \le \rho_0$ for some universal $\rho_0>0$. Then, there exist constants $C:=C(S, M)>0$ and $h_0=h_0(\rho)>0$ such that  for every sequence $(v_h,E_h)_{h>0}$ with $(v_h,E_h) \in \mathcal{A}_h$  and  
\begin{equation}\label{Lipschitz_replacement_energy_bound}
\sup_{h>0} \mathcal{G}^h(v_h,E_h)<+\infty\,,
\end{equation}
and for every $0<h\leq h_0$ there exist fields $r_h\in SBV^2(\Omega_{h,\rho};\R^3)$ and  $R_h\in SBV^2(\Omega_{h,\rho};\R^{3\times 3})$   satisfying  the following properties:
\begin{align}\label{almost_the same_volume_and_control_of_energy}
\hspace{-1em}\begin{split}
\rm{(i)}\quad& \|r_h\|_{L^\infty(\Omega_{h,\rho})}+\|R_h\|_{L^\infty(\Omega_{h,\rho})}\leq  C\,,\\[5pt]
\rm{(ii)}\quad&  \H^2(J_{r_h})+\H^2(J_{R_h})\leq Ch\,,\\[3pt]
\rm{(iii)}\quad& \int_{\Omega_{h,\rho}}|\nabla r_h|^2\,\mathrm{d}x + \int_{\Omega_{h,\rho}}|\nabla R_h|^2\,\mathrm{d}x \leq Ch\,,\\[3pt]
\rm{(iv)}\quad& h^{-1}\L^3(\Omega_{h,\rho}\cap \{|v_h(x)-r_h|>\theta_h\})\to 0\,,\ h^{-1}\L^3(\Omega_{h,\rho}\cap\{|\nabla v_h-R_h|>\theta_h\})\EEE\to 0\,,
\end{split}
\end{align}	
where $(\theta_h)_{h>0}\subset(0,+\infty)$ is a sequence with 
\begin{equation}\label{eq:choice_theta_h}
\theta_h\to0 \ \ \text{and }\ \theta_hh^{-1}\to\infty\,.
\end{equation}
Moreover, there exists $w_h \in SBV^2(\Omega_{h,\rho};\R^3)$ such that 
\begin{align}\label{w_h_almost_the same}
\hspace{-1em}\begin{split}
\rm{(i)}\quad& h^{-1}\L^3(\{w_h\neq v_h\}\cap \Omega_{h,\rho})\to 0 \ \text{as } h\to 0\,,\\[3pt]
\rm{(ii)}\quad& \int_{\Omega_{h,\rho}}|\nabla w_h-R_h|^2\,\mathrm{d}x\leq Ch^3\,,\\[3pt]		
\rm{(iii)}\quad& \|w_h\|_{L^\infty(\Omega_{h,\rho})}+h^{-1}\H^2(J_{w_h})+h^{-1}\int_{\Omega_{h,\rho}}|\nabla w_h|^2\,\mathrm{d}x\leq  C\,,\\[5pt]
\rm{(iv)}\quad& J_{w_h}  \subset \Omega_{h,\rho}\cap \bigcup_{Q_h(i)\in \OOO \Q\EEE_{v_h}}\partial Q_{h}(i)\,,\ \text{for some } \OOO \Q\EEE_{v_h}\subset \OOO \Q\EEE_h \ \text{with } \#\OOO \Q\EEE_{v_h}\leq Ch^{-1}\,. 
\end{split}
\end{align}	
\end{proposition}

The maps $r_h$, $R_h$ can be thought of as regularized piecewise affine/constant approximations of $v_h$, $\nabla v_h$, respectively, see \eqref{almost_the same_volume_and_control_of_energy}(iv). These approximations enjoy good $SBV^2$-bounds provided by \eqref{almost_the same_volume_and_control_of_energy}(i)--(iii), which will be crucial later to employ compactness and lower semicontinuity results in $SBV^2$  on their appropriate $T_{1/h}$-rescalings. The role of the maps $w_h$ in the second part of the proposition is analogous, with the extra advantage that $w_h$ is obtained by  changing the map $v_h$ on an asymptotically vanishing portion of the volume, see \eqref{w_h_almost_the same}(i), while having a more precise information on the geometry of the jump set, the latter being in a sense \textit{cubic}, see \eqref{w_h_almost_the same}(iv).

We note that $r_h$, $R_h$, and $w_h$ depend on $\rho$ which we do not include in the notation for simplicity. 
In the rest of this section, we focus on proving Proposition \ref{Lipschitz_replacement}, so that starting from Section \ref{compactness} we can give the proofs of our main compactness and $\Gamma$-convergence results. In the proofs, we will send the parameters $h, \rho $ to zero in this order. Thus, for the sake of keeping the notation simple, generic constants which are independent of $h, \rho$  will be denoted by $C$, and we will use a subscript notation in order to highlight the dependence of a particular constant on a specific parameter.

\subsection{Rigidity results}\label{se: rigi}

This subsection is devoted to recalling some rigidity results which are the basis for our proofs, as was also the case for the derivation of effective theories for elastic rods with voids in our previous work \cite{KFZ:2023}.

\emph{Geometric rigidity in variable domains:}  We first recall the result \cite[Theorem 2.1]{KFZ:2021}, see also \cite[Section 3.1]{KFZ:2023}. As mentioned in the introduction, the behavior of deformations $v$ on connected components of $\Omega_h \setminus \overline{E}$ might fail to be rigid,  cf.~\cite[Example~2.6]{KFZ:2021}. The main result in \cite{KFZ:2021} consists in showing that rigidity estimates can be obtained outside of a slightly thickened version of the voids. We omit the proofs of the next two results, since they are identical to the ones of \cite[Proposition~3.3 and Theorem 3.1]{KFZ:2023}.

\begin{proposition}[Global thickening of sets]\label{prop:setmodification} 
Let $h,\rho>0$ and $\gamma \in (0,1)$. There exists a universal constant $C_0>0$ and $\eta_0 = \eta_{0}(\rho) \in (0,1)$, such that for every $\eta\in (0,\eta_0]$ the following holds:

Given   $E\in \mathcal{A}_{\rm reg}(\Omega_h)$, we can find an open set $E_{h,\eta,\gamma}$ such that $E  \subset  E_{h,\eta,\gamma}  \subset \Omega_h$, $\partial E_{h,\eta,\gamma}\cap \Omega_h$ is a union of finitely many $C^2$-regular  submanifolds, and   
\begin{align}\label{eq: partition-new}
\begin{split}
{\rm (i)} & \ \    \mathcal{L}^3(E_{h,\eta,\gamma}\setminus E) \le h\eta \gamma^{1/{ 2}}  \mathcal{G}_{\rm surf}^{\gamma h^2}(E;\Omega_h),  \quad \quad \quad  \dist_\mathcal{H}(E, E_{h,\eta,\gamma}) \le h\eta \gamma^{1/{2}}\,,  
\\
{\rm (ii)} & \ \      \H^2(\partial E_{h,\eta,\gamma} \cap \Omega_h )   \leq (1+C_0\eta) \,  \mathcal{G}_{\rm surf}^{\gamma h^2}(E;\Omega_h)\,.
\end{split}
\end{align}
\end{proposition}

On the complement of $\Omega_{h,\rho}\setminus \overline{E_{h,\eta,\gamma}}$ one can obtain good quantitative piecewise rigidity estimates, as \MMM provided by \EEE the following theorem. For its statement, we  recall \eqref{eq:Q_h}--\eqref{eq:Q_h_rho}, and mention that this version is specific to our purposes compared to the more general one of \cite[Theorem 3.1]{KFZ:2023}, from which it actually follows by a direct application.

\begin{theorem}[Geometric rigidity in variable domains]\label{prop:rigidity}  
Let  $h,\rho>0$ and $\gamma \in (0,1)$. There exist a universal constant $C_0>0$,   $\eta_0 = \eta_0(\rho)>0$, and for each $\eta\in (0,\eta_0]$ there exists  $C_\eta>0$ such that the following holds:\\
\noindent For every  $E \in \mathcal{A}_{\rm reg}(\Omega_h)$, denoting by  $E_{h,\eta,\gamma}$ the set of Proposition \ref{prop:setmodification}, for every $Q_h(i)\in \mathcal{Q}_h$,  for the connected components $(U_j)_j$  of $\hat Q_{h,\rho}(i)  \setminus \overline{{E_{h,\eta,\gamma}}}$ and for every $y  \in W^{1,2}(\Omega_h \setminus \overline{E};\R^3)$ there exist  corresponding rotations $(R_j)_j \subset SO(3)$ and vectors $(b_j)_j\subset \R^3$ such that 
\begin{align}\label{eq: main rigitity}
{\rm (i)} & \ \  \sum\nolimits_j \int_{U_j}\big|{\rm sym}\big((R_j)^T \nabla y-\mathrm{Id}\big)\big|^2\,\mathrm{d}x 
\leq C_0 \big(1 +  C_\eta \gamma^{-15/2} h^{-3} \varepsilon   \big)\int_{\hat Q_h(i)\setminus  \overline{E}} \mathrm{dist}^2(\nabla y,SO( 3))\,\mathrm{d}x\,,\notag\\
{\rm (ii)} & \ \   \sum\nolimits_{j}\int_{U_j}    \big|(R_j)^T \nabla y-\mathrm{Id}\big|^2\,\mathrm{d}x \leq C_\eta  \gamma^{-3} \int_{\hat Q_h(i)\setminus  \overline{E}} \mathrm{dist}^2(\nabla y,SO( 3))\,\mathrm{d}x\,,
\notag\\
{\rm (iii)} & \ \   \sum\nolimits_{j}\int_{U_j}  \frac{1}{h^2} \big|  y- (R_jx+b_j) \big|^2 \,\mathrm{d}x \leq C_\eta  \gamma^{-5} \int_{\hat Q_h(i)\setminus  \overline{E}} \mathrm{dist}^2(\nabla y,SO( 3))\,\mathrm{d}x\,,
\end{align}
 where for brevity $\eps := \int_{\hat Q_h(i)\setminus \overline{E}} \dist^2(\nabla y, SO( 3)) \, {\rm d}x$. 
\end{theorem}

In the subsequent proofs we will choose the parameters $\eta$ and  $\gamma$ depending on the regime of the elastic energy $\eps$ such that $C_\eta \gamma^{-15/2}h^{-3}\varepsilon \le 1$ and $C_\eta  \gamma^{-5}  \le \eps^{-\theta}$ for some $\theta >0$ small. With these choices, we obtain a sharp control on symmetrized gradients with repect to $\eps$, \OOO see \EEE \eqref{eq: main rigitity}$\rm{(i)}$, while the estimate in \eqref{eq: main rigitity}$\rm{(ii)}$ and the Poincar\'e-type estimate \eqref{eq: main rigitity}$\rm{(iii)}$ yield a suboptimal control in the exponent, of the order $\eps^{1-\theta}$.  

\medskip 
 
\emph{Korn and Poincar\'e inequalities \OOO in $SBV^2$\EEE:} As in \cite{KFZ:2023}, the issue of the suboptimal exponent in the gradient estimate can be remedied in case the surface area of the void set is small. This relies on sophisticated Korn and Poincar\'e inequalities in the space  $GSBD^2$ \cite{DalMaso:13}. \MMM Since  \EEE we will here need the results only for $SBV^2$-functions, we formulate the corresponding \OOO statements \EEE of \cite[Theorem~1.1, Theorem~1.2]{Cagnetti-Chambolle-Scardia} in a simplified setting. \MMM In the sequel, we call \EEE a mapping $a\colon \R^{3} \to \R^{3}$  an \emph{infinitesimal rigid motion} if $a$ is affine with ${\rm sym}(\nabla a)= 0$.

\begin{theorem}[Korn-Poincar{\'e} inequality for functions with small jump set]\label{th: kornSBDsmall}
Let $U \subset \R^{d}$ be a bounded Lipschitz domain. Then, there exists a constant $c = c(U,d)>0$ such that for all  $u \in SBV^2(U;\R^{d})$  there  exists  a set of finite perimeter $\omega \subset U$ with 
\begin{align}\label{eq:R2main}
\mathcal{H}^{d-1}(\partial^* \omega) \le c\mathcal{H}^{d-1}(J_u)\,, \ \ \ \ \mathcal{L}^{d}(\omega) \le c(\mathcal{H}^{d-1}(J_u))^{d/d-1}\,,
\end{align}
and an infinitesimal rigid motion $a$ such that
\begin{align*}
 (\mathrm{diam}(U))^{-1}\Vert u - a \Vert_{L^{2}(U \setminus \omega)} + \Vert \nabla u - \nabla a \Vert_{L^{2}(U \setminus \omega)}\le c \Vert {\rm sym}(\nabla u) \Vert_{L^2(U)}\,.
\end{align*}
Moreover, there exists $v \in W^{1,2}(U;\R^{d})$ such that $v\equiv  u$ on $U \setminus \omega$ and
\begin{align*}
\Vert {\rm sym}(\nabla v) \Vert_{L^2(U)} \le c \Vert {\rm sym}(\nabla  u)\Vert_{L^2(U)}\,.
\end{align*}
Furthermore, if $u \in L^\infty(U;\R^{d})$, one has  $\|v\|_{L^\infty(U)}\leq  \| u \|_{L^\infty(U)}$.  
\end{theorem}

\begin{remark}\label{eq:easy_rem}
We refer the reader to the discussion below \cite[Theorem 3.2]{KFZ:2023} on the derivation of the above theorem from the results in \cite{Cagnetti-Chambolle-Scardia}. Note that the result is indeed only relevant if $\mathcal{H}^{d-1}(J_u)$ is small\OOO, \EEE since otherwise $\omega= U$ is possible and the statement is empty. Moreover, it is easily seen that the constant $c=c(U,d)$ of Theorem \ref{th: kornSBDsmall} is invariant under dilations of the domain $U$.
\end{remark}

\EEE An easy consequence of the above theorem is a version of the Poincar{\'e} inequality in $SBV$ in arbitrary codimension, namely the following. 

\begin{corollary}
\label{replace-cor}
Let $U \subset \R^{d}$ be a bounded Lipschitz domain and $m \in \N$. Then, there exists a constant $c = c(U,d,m)>0$ such that for all  $u \in SBV^2(U;\R^{m})$  there  exists  a set of finite perimeter $\omega \subset U$ with 
\begin{align}\label{eq:omega}
\mathcal{H}^{d-1}(\partial^* \omega) \le c\mathcal{H}^{d-1}(J_u)\,, \ \ \ \ \mathcal{L}^{d}(\omega) \le c(\mathcal{H}^{d-1}(J_u))^{d/(d-1)}\,,
\end{align}
and a constant vector $b \in \R^m$ such that
\begin{align}\label{eq:omega2}
\big(\mathrm{diam}(U)\big)^{-1}\Vert u - b \Vert_{L^{2}(U \setminus \omega)}  \le c  \Vert \nabla u \Vert_{L^2(U)}\,.
\end{align}
\end{corollary} 

\begin{proof}
 The statement  is a simple consequence of Theorem \ref{th: kornSBDsmall}.
Without restriction we can assume that   $m \ge d$, as otherwise we add $(d-m)$-components to $u$ which are identically zero.   \MMM We \EEE denote by $\Sigma$ the collection of \OOO all strictly increasing multi-indices of length $d$ from $\lbrace 1,\ldots,m\rbrace$\EEE, so that $\#\Sigma=\binom{m}{d}$. For every $\sigma \in \Sigma$ and $t \in \R^m$, we denote by $\pi_\sigma t \in \R^{\OOO m\EEE}$ the \OOO orthogonal \EEE  projection of $t$ onto the components indicated by $\sigma$. \OOO 
In \EEE a similar fashion, we define  $\pi_\sigma u \colon U \to \R^{\OOO m\EEE}$. We apply Theorem \ref{th: kornSBDsmall} on $\pi_\sigma u$ \OOO(which is essentially $\R^d$-valued) \EEE to obtain $\omega_\sigma \subset U$ satisfying \eqref{eq:omega} (for $\omega_\sigma$ in place of $\omega$) and $\OOO b_\sigma \EEE\in \R^{\OOO m\EEE}$ 
such that
\begin{align*}
\big(\mathrm{diam}(U)\big)^{-1}\Vert \pi_\sigma u - \OOO b_\sigma \EEE \Vert_{L^{2}(U \setminus \omega_\sigma)}  \le \OOO c\EEE \Vert \nabla u \Vert_{L^2(U)}\,.
\end{align*}
We define $\omega := \bigcup_{\sigma \in \Sigma} \omega_\sigma$ and observe that \eqref{eq:omega} holds. 
\OOO Defining \EEE
\[b := \frac{1}{\binom{m}{d-1}}\sum_{\sigma \in \Sigma} b_\sigma\in \R^{\OOO m \EEE}\,,\] it can be easily verified that \eqref{eq:omega2} holds, noticing that also 
$u= \frac{1}{\binom{m}{d-1}}\sum_{\sigma \in \Sigma} \pi_\sigma u$. 
\end{proof}

As in Theorem \ref{th: kornSBDsmall}, we emphasize that the application of the previous corollary is only meaningful if $\mathcal{H}^{d-1}(J_u)$ is smaller than a sufficiently small constant depending on $U$.

\EEE
\medskip

\emph{Difference of affine maps:} We close this subsection with the statement of the following elementary lemma, cf.~\cite[Lemma 3.1]{KFZ:2023}, whose proof can be found therein. By $B_r(x) \subset \R^3$ we denote the open ball  centered at $x \in \R^3$ with radius $ r>0$. 
\begin{lemma}[Estimate on  affine maps]\label{lemma: rigid motions}
Let $\delta >0$. Then there exists a constant $C>0$ only depending on  $\delta$ such for every $G \in \R^{3\times  3}$, $b \in \R^{3}$, $x\in \R^3$, and $E \subset B_r(x)$ for some $r >0$ with $\mathcal{L}^{ 3}(E) \ge \delta r^{ 3}$\BBB,\EEE we have
$$\Vert G \cdot  + b \Vert_{L^\infty(B_r(x))} \le C\OOO r^{-3/2} \EEE \Vert G \cdot  + b \Vert_{L^2(E)}\,, \quad |G| \le C \OOO r^{-5/2}\EEE \Vert G \cdot  + b \Vert_{L^2(E)}\,.   $$
\end{lemma}

\subsection{Local rigidity estimates and Sobolev replacement on good cubes}\label{sec: locest}

In this subsection we first introduce some extra necessary notation and definitions for the rest of the section. We start by introducing the thickened void sets and then partition our reference domain $\Omega_{h, \rho}$,  see \EEE \eqref{Omega_h_local} into cubes, where the partitioning is with respect to the surface area of the boundary of the thickened void and the size of the local elastic energy \OOO in each cuboid of the partition\EEE.

Let us start with a sequence $(v_h,E_h)_{h>0}$ of admissible deformations and void sets in the  thin plate $\Omega_h$. Recalling \eqref{eq: G energy}, we suppose that 
\begin{equation}\label{uniform_energy_bound}
\sup_{h>0}  \mathcal{G}^{h}(v_h,E_h)< +\infty\,.
\end{equation}
We fix 
\begin{equation}\label{eq:choice_of_rho}
0<\rho \le \rho_0:= 1 -(127/128)^{1/3}
\end{equation} 
as in Proposition \ref{Lipschitz_replacement} and recall the choice of \OOO$(\kappa_h)_{h>0}$ \EEE in \eqref{rate_1_gamma_h}.  
Let $\eta_0 = \eta_{0}(\rho) \in (0,1)$ be the \MMM minimum of the  constants in Proposition \ref{prop:setmodification} and Theorem~\ref{prop:rigidity}. \EEE  In view of \eqref{rate_1_gamma_h}, we can choose a sequence $(\eta_h)_{h >0} \subset (0,\eta_0)$ converging slowly enough to zero, so that the constant $C_{\eta_h}$ in \eqref{eq: main rigitity}, obtained by applying Theorem~\ref{prop:rigidity} for $\rho$, $\eta = \eta_h \in (0,\eta_0(\rho))$, and $\gamma = {\kappa}_h/h^2$, satisfies  
\begin{equation}\label{parameters_for_uniform_bounds}
\limsup_{h\to 0 } \,   C_{\eta_h}\Big(\frac{h^2}{{\kappa}_h}\Big)^{5} h^{2/5} \leq 1\,.
\end{equation}
\MMM Applying \EEE Proposition \ref{prop:setmodification} with   the above choice of $\rho$, $\eta$ and $\gamma$\EEE, for all $h>0$ we find open sets $E^*_h$ with $E_{h} \subset E^*_h \subset \Omega_h$ such that $\partial E^*_h\cap \Omega_{h}$ is a union of finitely many $C^2$-regular submanifolds and 
\begin{align}\label{eq: void-new}
\begin{split}
{\rm{(i)}} & \ \   h^{-2} \mathcal{L}^3(E^*_h\setminus E_{h}) \to 0,  \quad \quad  h^{-1}\dist_\mathcal{H}(E^*_h,E_h) \to 0 \ \ \text{ as $h \to 0$}\,,
\\[2pt]
{\rm{(ii)}} & \ \    \liminf_{h \to 0 }    h^{-1}  \H^2(\partial E^*_h\cap \Omega_h)\leq    \liminf_{h \to 0 } \, h^{-1}\mathcal{G}^{\kappa_h}_{\rm surf}(E_h;\Omega_h) \,.
\end{split}
\end{align}
Here, we made use of \eqref{eq: partition-new}, $ \eta_h \to 0$, \eqref{rate_1_gamma_h}, and the fact that $h^{-1} \mathcal{G}_{\rm surf}^{\gamma h^2}(E_h;\Omega_h) = h^{-1}\mathcal{G}_{\rm surf}^{ {\kappa}_h}(E_h;\Omega_h) $ is uniformly bounded by  \eqref{eq: G energy} and \eqref{uniform_energy_bound}. In the estimates \eqref{eq: main rigitity}, the behavior of the deformation inside $E_h^*$  cannot be controlled. Thus, in accordance to \eqref{initial_admissible_configurations}, for definiteness only we can without restriction assume that the deformation is the identity also inside $E_h^*$, i.e., we define  $v_h^*\colon \Omega_h \to \R^3$ by
\begin{align}\label{eq: *mod}
v_h^*(x) := \begin{cases}  v_h(x) &  \text{ if } x \in \Omega_h \setminus E_h^*\,, \\
\OOO x & \text{ if } x \in E_h^*\,.
\end{cases}
\end{align}
In view of \eqref{initial_admissible_configurations} and \eqref{eq: void-new}(i), we get 
\begin{align}\label{eq: *mod2}
h^{-2} \mathcal{L}^3( \lbrace  v_h^* \neq v_h \rbrace ) \le h^{-2}  \mathcal{L}^3(E^*_h\setminus E_{h}) \to 0 \quad \text{ as $h \to 0$}\,.
\end{align}
In view of Remark \ref{eq:easy_rem}, set $c_\mathrm{KP}:=c(\hat Q_1(0))$ and also introduce the (small) parameter
\begin{align}\label{eq: T2}
\alpha \OOO:=\EEE \Big(\frac{128}{9} \max\{\OOO 2\EEE c_{\mathrm{\OOO isop\EEE}}, c_\mathrm{KP}\}\Big)^{-2/3}\,,
\end{align}
where $c_{\rm{\OOO isop\EEE}}$ denotes the relative isoperimetric constant of the cuboid $\hat Q_1(0)$ in $\R^3$, the latter being also scaling invariant, cf.~\cite[Equation  (3.43)]{Ambrosio-Fusco-Pallara:2000} for a version stated on balls instead of cuboids. 
For every $Q_h(i)\in \Q_h$, we also introduce the localized elastic energy
\begin{equation}\label{eq: localized_elastic}
\eps_{i,h}:=\int_{\hat Q_h(i)\setminus \overline{E_h}} \mathrm{dist}^2(\nabla v_h, SO(3))\,\mathrm{d}x\,.
\end{equation}
We now divide the family of cubes $\mathcal{Q}_h$ into two subfamilies: first,  we consider  the family  of indices associated to \textit{good cubes}, defined by
\begin{equation}\label{good_cuboids}
 I^h_{\rm g} :=\Big\{i\in \Z^2:Q_h(i)\in \Q_h, \ 
\H^2(\partial E^*_h\cap  \hat Q_{h,\rho}(i))\leq \alpha h^2\,, \ \eps_{i,h}\leq h^4
\Big\}\,.
\end{equation} 
For each cube in this subfamily, Theorem \ref{th: kornSBDsmall}\ is applicable without introducing a too large exceptional set, cf.~\eqref{eq:R2main}. The complementary family of indices will correspond to the \emph{bad cubes}, namely we set 
\begin{align*}
 I^h_{\rm b} := \lbrace i\in \Z^2\setminus I^h_{\OOO \rm{g}\EEE}\colon Q_h(i)\in \Q_h\rbrace\,. 
\end{align*}
We also remark that, for each $i\in \Z^2$,  
\begin{align}\label{ 5number}
\#\big\{i'\in \Z^2\colon \hat Q_{h,\rho}(i) \cap \hat Q_{h,\rho}(i')\neq \emptyset\big\}\leq 25\,.
\end{align}
 By \eqref{good_cuboids}--\eqref{ 5number}, \eqref{eq: void-new}(ii), \eqref{eq: nonlinear energy}(iv), \eqref{eq: G energy}, and \eqref{uniform_energy_bound}, for $h>0$ small enough, we obtain
\begin{align*}
\#I^h_{\rm b} &\leq \alpha^{-1}h^{-2} \sum_{i \in I^h_{\rm b}} \H^2(\partial E^*_h\cap  \hat Q_{h,\rho}(i))+h^{-4}\sum_{i\in I_{\rm{b}}^h}\eps_{i,h} \\
&\leq C\alpha^{-1}h^{-2} \mathcal{H}^2(\partial E^*_h\cap \Omega_h)+Ch^{-4}\int_{\Omega_h\setminus \overline{E}}W(\nabla v_h)\,\mathrm{d}x\leq Ch^{-1}\mathcal{G}^h(v_h,E_h)\leq Ch^{-1}\,,
\end{align*}
i.e., we deduce that
\begin{equation}\label{cardinality_of_bad_cuboids}
\# I^h_{\rm b}  \leq Ch^{-1}
\end{equation} \EEE
for an absolute constant $C=C(\alpha) >0$, independent of $h$.  Moreover, by using again \eqref{eq: nonlinear energy}(iv), \OOO \eqref{ 5number}, \eqref{eq: G energy}, and \eqref{uniform_energy_bound}\EEE, we also obtain the estimate 
\begin{equation}\label{localized_elastic_energy_on_good_cuboids}
\sum_{i\in I^h_{\rm{g}}\cup I^h_{\rm{b}}}\eps_{i,h} \leq C \int_{\Omega_h\setminus \overline{E_h}}\mathrm{dist}^2(\nabla v_h,SO(3))\,\mathrm{d}x\leq  Ch^3\,.
\end{equation}

\begin{proposition}[Local rigidity and Sobolev approximation]\label{summary_estimates_proposition}
Let $0 < \rho \le \rho_0$. There exist an absolute constant $C>0$ independent of $h$ and $h_0=h_0(\rho)>0$ such that for all $0<h\leq h_0$ and for every $i\in  I^h_{\rm g}$ there exists a set of finite perimeter $D_{i,h}\subset  \hat Q_{h,\rho}(i)$, satisfying  
\begin{equation}\label{big_volume_of_dominant_set}  
 \L^3\big(\hat Q_{h,\rho}(i)\setminus D_{i,h} \big)  \leq C h\H^{2}\big(\partial E_h^*\cap \hat Q_{h}(i)\big)\,, \  \L^3(\hat Q_{h}(i) \setminus D_{i,h})\leq \frac{1}{32}\L^3(\hat Q_{h}(i))\,, 
\end{equation}	 
\begin{equation}\label{big_surface_of_dominant_set}  
\ \H^2(\partial^*D_{i,h}\cap \hat Q_{h,\rho}(i))\leq C\H^{2}\big(\partial E_h^*\cap \hat Q_{h,\rho}(i)\big)\,,
\end{equation}
and a corresponding rigid motion $r_{i,h}(x):=R_{i,h}x+ b_{i,h}$, where $R_{i,h}\in SO(3)$ and $b_{i,h}\in \R^3$ with $|b_{i,h}| \le CM$ (recall \eqref{initial_admissible_configurations}), such that 
\begin{equation}\label{L2_gradient_estimates}
h^{-2}\int_{D_{i,h}}\big|v^*_h(x)-r_{i,h}(x)\big|^2\,\mathrm{d}x+\int_{D_{i,h}}\big|\nabla v_h^*(x)- R_{i,h}\big|^2\, \mathrm{d}x\leq
 {C}\eps_{i,h}\,,
\end{equation}
with $v_h^*$ defined as in \eqref{eq: *mod}.

Furthermore, 
\OOO there exists \EEE 
a Sobolev map   $z_{i,h}\in W^{1,2}(\hat Q_{h,\rho}(i);\R^3)$  such that 
\begin{align}\label{properties_of_Sobolev_replacement}
\begin{split}
\rm{(i)}&\quad z_{i,h}\equiv v^*_h \quad \text{on }\ D_{i,h}\,,\\
\rm{(ii)}&\quad h^{-2}\int_{\hat Q_{h,\rho}(i)}\big|z_{i,h}(x)-r_{i,h}(x)\big|^2\,\mathrm{d}x+\int_{\hat Q_{h,\rho}(i)}\big|\nabla z_{i,h}(x)- R_{i,h}\big|^2\, \mathrm{d}x\leq
{C}\eps_{i,h}\,, \\
\rm{(iii)} & \quad \Vert z_{i,h} \Vert_{L^\infty(\hat Q_{h,\rho}(i))} \le  CM\,.
\end{split}
\end{align}
\end{proposition}

Since $\L^3(\hat Q_{h}(i)\setminus D_{i,h})$ is small, we will be hereafter referring to $D_{i,h}$ as the \emph{dominant component}, and in a similar fashion $r_{i,h}$ will be referred to as the \emph{dominant rigid motion} which approximates $v_h^*$ in $\hat Q_{h, \rho}(i)$. Of course, it is still possible that $D_{i,h} \subset E_h^*$, which should be interpreted as the void having large volume in $\hat Q_{h,\rho}(i)$.

The reader might have already noticed that the estimate  \eqref{L2_gradient_estimates} for the full gradient and also for the deformations $v_h^*$ comes with the optimal exponent in the local elastic energy \OOO $\eps_{i,h}$\EEE, which at first glimpse might be in contrast to \eqref{eq: main rigitity}, as $\gamma^{-1}= (\kappa_h/h^2)^{-1}$ is diverging with $h\to 0^{\OOO+\EEE}$, cf.~\eqref{rate_1_gamma_h}. As shown in the proof, such an improvement is possible by applying the Korn-Poincar\'e inequality of Theorem \ref{th: kornSBDsmall}, in the case of void sets with small surface area. 
 
The proof of Proposition \ref{summary_estimates_proposition} is basically a repetition of the analogous one in \cite[Proposition~3.5]{KFZ:2023}. For the sake of  completeness, we include it in  Appendix \ref{sec: aux estimates}. The main difference is the additional requirement $\eps_{i,h}\leq h^4$ in the definition of $I^h_{\OOO\rm{g}\EEE}$ in \eqref{good_cuboids}. We refer to  Remark \ref{rem:choice_h_4} for a comment in this direction. 

By a standard argument, which is again repeated in Appendix \ref{sec: aux estimates}, an immediate consequence of the previous proposition is an optimal estimate for the difference of two dominant rigid motions on neighboring good cubes.  

\begin{corollary}[Difference of rigid motions]\label{difference_rigid-motions}
Let $i,i'\in I^h_{\rm g}$ be such that $|i-i'|_\infty\leq 1$. The rigid motions $r_{i,h}$,  $r_{i',h}$ of  Proposition \ref{summary_estimates_proposition} satisfy 
\begin{equation}\label{general_differences_rotations2}
h^{-2}\|r_{i,h}-  r_{i',h}\|_{L^{\infty}(\hat Q_{h}(i)  \cup \hat Q_{h}(i') )   }^2+\big|R_{i,h}-R_{i',h}\big|^2\leq  Ch^{-3}(\eps_{i,h}+\eps_{i',h})\,.
\end{equation}
\end{corollary}

\subsection{Construction of almost Sobolev replacements and proof of Proposition~\ref{Lipschitz_replacement}}\label{sec: global_constructions_proofs}

In this subsection we construct the fields $(r_h)_{h>0}$, $(R_h)_{h>0}$, and $(w_h)_{h>0}$ of Proposition~\ref{Lipschitz_replacement}, and afterwards give the proof of the proposition.

First of all, recalling \eqref{Omega_h_local} and   \eqref{good_cuboids},  \MMM  we introduce index sets related to \textit{interior good cubes} by \EEE
\begin{equation}\label{eq:interior_indices} 
I^h_{\rm{int}}:=\big\{i\in I^h_{\rm{g}}\colon i'\in I^h_{\rm{g}} \ \text{ for every }   i' \in \Z^2   \text{ with }  |i'-i|_{\infty}\leq 1\big\}\,,
\end{equation} 
 and \MMM {\it exterior good cubes} by \EEE
\begin{equation*} 
I^{h,\rho}_{\rm{ext}}:=\big\{i\in I^h_{\rm{g}}\setminus I^h_{\rm{int}}\colon Q_h(i)\cap \Omega_{h,\rho}\neq \emptyset\big\}\,.
\end{equation*} 
\EEE
Note that, \OOO for $h>0$ small enough, \EEE if $i\in \OOO I^{h,\rho}_{\rm{ext}}\EEE$, then there exists at least one $i'\in \Z^2$ with $|i'-i|_\infty\leq 1$ such that $i'\in I^h_{\rm{b}}$. \MMM By \EEE    \eqref{cardinality_of_bad_cuboids}  this  yields 
\begin{equation}\label{eq:card_of_non_inter_indices}
\#\OOO I^{h,\rho}_{\rm{ext}}\EEE\leq Ch^{-1}\,,	
\end{equation}
for a universal constant $C>0$. 

For the construction of the sequences $(r_h)_{h>0}$, $(R_h)_{h>0}$, and $(w_h)_{h>0}$, we consider a partition of unity subordinate \MMM to \EEE $I^h_{\rm{int}}$, cf.~\eqref{eq:interior_indices}. For every $h>0$, we introduce $(\psi^i_h)_{i\in I^h_{\rm{int}}}\subset C^\infty_c(\R^3)$ such that
\begin{equation}\label{eq:partition_of_unity_2}
\sum_{i\in I^h_{\rm{int}}} \psi^i_h=1\,,
\end{equation}
and, for every $i\in I^h_{\rm{int}}$\,,
\begin{equation}\label{eq:partition_of_unity_1}
0\leq \psi^i_h\leq 1\,, \quad \psi^i_h\equiv 1 \ \text{on } \tilde Q_{\OOO 9h/10\EEE }(i)\,, \quad \psi^i_h\equiv 0 \ \text{on } \R^2\setminus \tilde Q_{\OOO 11h/10\EEE }(i)\,, \quad \|\nabla\psi^i_h\|_{L^\infty(\R^2)}\leq \frac{C}{\OOO h\EEE}\,, 
\end{equation}
where for \MMM $s>0$ \EEE we have used the notation 
\[\tilde Q_{sh}(i):=(hi,0)+\left(\EEE-\frac{sh}{2},\frac{sh}{2}\right)^2\times \left(-\frac{h}{2},\frac{h}{2}\right)\,.\]
Then, we define $r_h \in  SBV^2(\Omega_{h,\rho} ;\R^3)$, $R_h\in SBV^2(\Omega_{h,\rho};\R^{3\times 3})$ and $w_{h}\in  SBV^2(\Omega_{h,\rho} ;\R^3)$ as
\begin{align}\label{eq:r_h_construcc}
\begin{split}
\ \ \ r_h(x) := \begin{cases}  x& \text{ if } x\in Q_{h}(i)\cap\Omega_{h,\rho} \text{ for some } i\notin I^h_{\rm{int}}\,,\\
\sum_{i\in I^h_{\rm{int}}}\psi_h^i(x)r_{i,h}(x) & \text{ otherwise}\,,
\end{cases}
\end{split}
\end{align}
\begin{align}\label{eq:R_h_construcc}
\begin{split}
\  R_h(x) := \begin{cases}  \mathrm{Id}& \ \text{if } x\in Q_{h}(i)\cap\Omega_{h,\rho}\ \text{ for some } i\notin I^h_{\rm{int}}\,,\\
\sum_{i\in I^h_{\rm{int}}}\psi_h^i(x)R_{i,h} & \ \text{otherwise}\,,
\end{cases}
\end{split}
\end{align}
\begin{align}\label{eq:w_h_construcc}
\begin{split}
w_h(x) := \begin{cases}  x& \text{ if } x\in Q_{h}(i)\cap\Omega_{h,\rho} \text{ for some } i\notin I^h_{\rm{int}}\,,\\
\sum_{i\in I^h_{\rm{int}}}\psi_h^i(x)z_{i,h}(x) & \text{ otherwise}\,,
\end{cases}
\end{split}
\end{align}
where the fields $r_{i,h}$, $R_{i,h}$, and $z_{i,h}$ are given in Proposition \ref{summary_estimates_proposition}. The construction and the definition of $\Omega_{h,\rho}$ in \eqref{Omega_h_local} implies \OOO that \EEE $r_h \in  SBV^2(\Omega_{h,\rho};\R^3)$,  $R_h \in  SBV^2(\Omega_{h,\rho};\R^{3\times 3})$, $w_h \in~SBV^2(\Omega_{h,\rho};\R^{3})$, and the jump sets satisfy, for $\OOO h>0\EEE$ sufficiently small, 
\begin{align}\label{eq: whjum}
J_{r_h} \cup  J_{R_h} \cup J_{w_h} \subset \Omega_{h,\rho}\cap \bigcup_{i \in I^h_{\rm{b}}\cup\OOO I^{h,\rho}_{\rm{ext}}\EEE} \partial Q_{h}(i)\,. 
\end{align}
We are now ready to give the proof of  Proposition \ref{Lipschitz_replacement}, following the strategy of the proof of the corresponding result in the setting of rods, namely \cite[Proposition 3.1]{KFZ:2023}.

\begin{proof}[Proof of Proposition \ref{Lipschitz_replacement}]
First of all, the bounds in \eqref{almost_the same_volume_and_control_of_energy}(i) follow from the definitions \eqref{eq:r_h_construcc}--\eqref{eq:R_h_construcc}, the bound $|b_{i,h}|\le CM$ for $i \in I^h_{\rm g}$, and the fact that $SO(3) \subset \R^{3 \times 3}$ is compact. Next, \eqref{almost_the same_volume_and_control_of_energy}(ii) follows directly from \eqref{eq: whjum}, \eqref{cardinality_of_bad_cuboids}, and \eqref{eq:card_of_non_inter_indices}.

We proceed to show \eqref{almost_the same_volume_and_control_of_energy}(iii), which we verify first for the fields $(R_h)_{h>0}$. For this, we first obtain  a control on $\|\nabla R_h\|_{L^\infty(Q_{h}(i)\cap \Omega_{h,\rho})}$ for every $i\in I_{\rm{int}}^{h}$. Indeed, in view of \eqref{eq:R_h_construcc}, \eqref{eq:partition_of_unity_2}--\eqref{eq:partition_of_unity_1}, and \eqref{general_differences_rotations2}, for $i\in I_{\rm{int}}^{h}$ fixed, $x\in Q_{h}(i)\cap \Omega_{h,\rho}$, and $k\in \{1,2,3\}$, we can estimate 
\begin{align*}
|\partial_kR_h(x)|&=\Big|\sum_{j\in I^h_{\rm{int}}}\partial_k\psi_h^{j}(x)R_{j,h}\Big|=\Big|\sum_{j\in I^h_{\rm{int}}}\partial_k\psi_h^{j}(x)(R_{j,h}-R_{i,h})\Big|\\
&\leq\sum_{j\in \mathcal{N}(i)}\|\partial_k\psi_h^j\|_{L^\infty}|R_{j,h}-R_{i,h}|\leq \frac{C}{h}h^{-3/2}\sum_{j\in \mathcal{N}(i)}\eps^{1/2}_{j,h} \,,
\end{align*}
where  
\begin{equation}\label{eq:neighbouring}
\mathcal{N}(i):=\{j\in \Z^2\colon |j-i|_{\infty}\leq 1\}\,.
\end{equation} 
Therefore,
\begin{align}\label{eq:nabla_R_h_L_infty}
\begin{split}
\|\nabla R_h\|^2_{L^\infty(Q_h(i)\cap \Omega_{h,\rho})}\leq Ch^{-5}\sum_{j\in \mathcal{N}(i)}\eps_{j,h} \,.
\end{split}
\end{align}
Using \eqref{eq:R_h_construcc}, \eqref{eq:nabla_R_h_L_infty}, and \eqref{localized_elastic_energy_on_good_cuboids}, we can thus estimate
\begin{equation}\label{eq:nabla_R_h_L_2_grad}
\int_{\Omega_{h,\rho}}|\nabla R_h|^2\,\mathrm{d}x\leq \sum_{i\in I^h_{\rm{int}}}\int_{Q_h(i)\cap \Omega_{h,\rho}}|\nabla R_h|^2\,\mathrm{d}x\leq Ch^{-5}h^3\sum_{i\in I^h_{\rm{g}}}\eps_{i,h}\leq Ch\,.
\end{equation}
Analogously, for $i\in I_{\rm{int}}^{h}$, $x\in Q_{h}(i)\cap \Omega_{h,\rho}$, and $k\in \{1,2,3\}$, using \eqref{eq:r_h_construcc}, \eqref{eq:partition_of_unity_2}--\eqref{eq:partition_of_unity_1}, and \eqref{general_differences_rotations2}, we estimate 
\begin{align*}
|\partial_kr_h(x)|&=\Big|\sum_{j\in I^h_{\rm{int}}}(\partial_k\psi_h^{j}(x)r_{j,h}(x)+\psi_h^{j}(x)R_{j,h}e_k)\Big| \le \Big|\sum_{j\in I^h_{\rm{int}}}\partial_k\psi_h^{j}(x)(r_{j,h}(x)-r_{i,h}(x))\Big|+ 1 \\
&\le \sum_{j\in \mathcal{N}(i)}\|\partial_k\psi_h^j\|_{L^\infty}\|r_{j,h}-r_{i,h}\|_{L^\infty(Q_{h}(i))} +1\leq Ch^{-3/2}\sum_{j\in \mathcal{N}(i)}\eps^{1/2}_{j,h} +1\,.
\end{align*}
Thus,
\begin{equation}\label{eq:nabla_r_h_L_infty}
\|\nabla r_h\|^2_{L^\infty(Q_h(i)\cap \Omega_{h,\rho})}\leq Ch^{-3}\sum_{j\in \mathcal{N}(i)}\eps_{j,h} +1\,.
\end{equation}
In a similar fashion as before, using  \eqref{eq:r_h_construcc}, \eqref{cardinality_of_bad_cuboids}, \eqref{eq:card_of_non_inter_indices}, \EEE \eqref{eq:nabla_r_h_L_infty}, and  \eqref{localized_elastic_energy_on_good_cuboids}, we can estimate
\begin{align}\label{eq:nabla_r_h_L_2_grad}
\begin{split}
\int_{\Omega_{h,\rho}}|\nabla r_h|^2\,\mathrm{d}x&  = 3\sum_{i\in I^h_{\rm{b}}\cup  I^{h,\rho}_{\rm{ext}}}\L^3(Q_{h}(i)\cap \Omega_{h,\rho})+ \EEE \sum_{i\in I^h_{\rm{int}}}\int_{Q_h(i)\cap \Omega_{h,\rho}}|\nabla r_h|^2\,\mathrm{d}x \\
&\leq   C\Big[\big(\#I^h_{\rm{b}}+  \#I^{h,\rho}_{\rm{ext}}\big) h^{3}+\sum_{i\in I^h_{\rm{g}}}\eps_{i,h}+h\Big]\leq C(h^2+h^3+h)\leq Ch\,. \EEE
\end{split}
\end{align}
Collecting \eqref{eq:nabla_R_h_L_2_grad} and \eqref{eq:nabla_r_h_L_2_grad}, we conclude the proof of \eqref{almost_the same_volume_and_control_of_energy}(iii).

We now proceed to verify \eqref{almost_the same_volume_and_control_of_energy}(iv). Using \OOO
\eqref{cardinality_of_bad_cuboids},
\eqref{eq:card_of_non_inter_indices}, \EEE
\eqref{big_volume_of_dominant_set}, \eqref{ 5number}, \eqref{uniform_energy_bound}, \eqref{eq: void-new}, \eqref{F_surf_energy}, and \eqref{eq: G energy} \EEE we first estimate for $h>0$ sufficiently small, 
\begin{align}\label{eq:measure_estimate_v_h_1}
\begin{split}	
\L^3\Big(\Omega_{h,\rho}\setminus\bigcup_{i\in I^h_{\rm{int}}}D_{i,h}\big)&\leq \sum_{\OOO i\in I^h_{\rm{b}}\cup I^{h,\rho}_{\rm{ext}}}\EEE\L^3(\hat Q_{h,\rho}(i))+
\sum_{i\in I^h_{\rm{int}}}\L^3(\hat Q_{h,\rho}(i)\setminus D_{i,h})\\
&\leq C\big(\#I^h_{\rm{b}}+\#\OOO I^{h,\rho}_{\rm{ext}}\EEE\big)h^{3}
+Ch\sum_{i\in I^h_{\rm{int}}}\H^2(\partial E_h^*\cap \hat Q_{h}(i))\\
&\leq Ch^2+ Ch\H^2(\partial E_h^*\cap \Omega_h)\leq Ch^2\,.\EEE
\end{split}	
\end{align}
Choose a sequence  $(\theta_h)_{h>0}\subset(0,+\infty)$ as in \eqref{eq:choice_theta_h}. Then, using \OOO Chebyshev's inequality, \EEE \eqref{eq:r_h_construcc}, \eqref{eq:partition_of_unity_2}--\eqref{eq:partition_of_unity_1}, \eqref{L2_gradient_estimates}, \eqref{general_differences_rotations2}, and  \eqref{localized_elastic_energy_on_good_cuboids}, we find
\begin{align}\label{eq:measure_estimate_v_h_2}
\L^3\OOO\Big(\EEE\{|v^*_h(x)-r_h|>\theta_h\}\cap\bigcup_{i\in I^h_{\rm{int}}}D_{i,h}\OOO\Big)\EEE&\leq \theta_h^{-2}\sum_{i\in I^h_{\rm{int}}}\int_{D_{i,h}}\Big|\sum_{j\in \mathcal{N}(i)}\psi_h^j(v_h^*(x)-r_{j,h})\Big|^2\,\mathrm{d}x\nonumber\\
&\leq C\theta_h^{-2}\sum_{i\in I^h_{\rm{int}}}\sum_{j\in \mathcal{N}(i)}\int_{D_{i,h}}\big|v_h^*(x)-r_{i,h}\big|^2\,\mathrm{d}x\nonumber\\
& \ \ \ +C\theta_h^{-2}\sum_{i\in I^h_{\rm{int}}}\sum_{j\in \mathcal{N}(i)}\int_{D_{i,h}}\big|r_{i,h}-r_{j,h}\big|^2\,\mathrm{d}x\nonumber\\
&\leq C\theta_h^{-2}h^{2}\sum_{i\in I^h_{\rm{g}}}\eps_{i,h}\leq  Ch^{5}\theta_h^{-2}\,. 
\end{align}
Combining now \eqref{eq: *mod2}, \eqref{eq:measure_estimate_v_h_1}, and  \eqref{eq:measure_estimate_v_h_2}, we obtain
\begin{align}\label{eq:measure_estimate_v_h_3}
\L^3\big(\Omega_{h,\rho}\cap \{|v_h(x)-r_h|>\theta_h\}\big)&\leq \L^3\big(\{v_h^*\neq v_h\}\big)+\L^3\big(\Omega_{h,\rho}\setminus\bigcup_{i\in I^h_{\rm{int}}}D_{i,h}\big)\nonumber\\
&\quad +\L^3\big(\{|v^*_h(x)-r_h|>\theta_h\}\cap\bigcup_{i\in I^h_{\rm{int}}}D_{i,h}\big)\nonumber\\
&\leq  Ch^2+Ch^5\theta_h^{-2}\,. 
\end{align}
Using the same estimates as for \eqref{eq:measure_estimate_v_h_2}, we also get 
\begin{equation*}
\L^3\OOO\Big(\EEE\{|\nabla v^*_h(x)-R_h|>\theta_h\}\cap\bigcup_{i\in I^h_{\rm{int}}}D_{i,h}\OOO\Big)\EEE\leq Ch^{3}\theta_h^{-2}\,, 
\end{equation*}
so that, in the same way as for \eqref{eq:measure_estimate_v_h_3}, we   obtain
\begin{equation}\label{eq:measure_estimate_v_h_5}
\L^3\big(\Omega_{h,\rho}\cap \{|\nabla v_h(x)-R_h|>\theta_h\}\big)\leq Ch^2+Ch^3\theta_h^{-2}\,. 
\end{equation}
Now, \eqref{almost_the same_volume_and_control_of_energy}(iv) follows from \eqref{eq:measure_estimate_v_h_3}--\eqref{eq:measure_estimate_v_h_5} and \eqref{eq:choice_theta_h}. 
Similarly, we observe that 
\begin{equation*}
\OOO\{w_h\neq v_h\}\cap \Omega_{h,\rho}\EEE\subset \{v_h\neq v_h^*\}\cup \bigcup_{i\in I^h_{\rm{b}}\cup \MMM I^{h,\rho}_{\rm{ext}} \EEE  }\hat Q_{h,\rho}(i)\cup \bigcup_{i\in I^h_{\rm{int}}}\big(\EEE\hat Q_{h,\rho}(i)\setminus D_{i,h}\big)\,,
\end{equation*}
so that \eqref{w_h_almost_the same}(i) follows from \eqref{eq: *mod2} and  \eqref{eq:measure_estimate_v_h_1}. 
We proceed with \eqref{w_h_almost_the same}(ii). By \OOO \eqref{eq:R_h_construcc}, \EEE \eqref{eq:w_h_construcc}, \OOO \eqref{eq:partition_of_unity_2}--\eqref{eq:partition_of_unity_1}, and \EEE \eqref{eq:neighbouring}, we can   estimate  
\begin{align}\label{eq:gradient_bound_z_h_1}	  
\int_{\Omega_{h,\rho}}|\nabla w_h-R_h|^2\,\mathrm{d}x&\le \sum_{i\in I^h_{\rm{int}}}\int_{\hat Q_{h,\rho}(i)}\Big|\sum_{j\in \mathcal{N}(i)}z_{j,h}  \otimes   \nabla \psi_h^j +\sum_{ \MMM j\in \mathcal{N}(i)}\psi_h^j(\nabla z_{j,h}-R_{j,h})\Big|^2\nonumber\\
&= \sum_{i\in I^h_{\rm{int}}}\int_{\hat Q_{h,\rho}(i)}\Big|\sum_{j\in \mathcal{N}(i)} (z_{j,h}-r_{i,h})\otimes \nabla \psi_h^j +\sum_{j\in \mathcal{N}(i)}\psi_h^j(\nabla z_{j,h}-R_{j,h})\Big|^2\nonumber\\
&\leq Ch^{-2}\sum_{i\in I^h_{\rm{int}}}\sum_{j\in \mathcal{N}(i)}\int_{\hat Q_{h,\rho}(i)}|z_{j,h}-r_{i,h}|^2\,\mathrm{d}x+C\sum_{i\in I^h_{\rm{g}}}\int_{\hat Q_{h,\rho}(i)}|\nabla z_{i,h}-R_{i,h}|^2\,\mathrm{d}x\nonumber\\
&\leq Ch^{-2}\sum_{i\in I^h_{\rm{int}}}\sum_{j\in \mathcal{N}(i)}h^3\|r_{i,h}-r_{j,h}\|^2_{L^\infty(\hat Q_h(i)\cup\hat Q_h(j))}\nonumber \\
&\qquad +C\sum_{i\in I^h_{\rm{g}}}\int_{\hat Q_{h,\rho}(i)}\left(h^{-2}|z_{i,h}-r_{i,h}|^2+|\nabla z_{i,h}-R_{i,h}|^2\right)\,\mathrm{d}x\,.
\end{align}
Using \eqref{properties_of_Sobolev_replacement}(ii), \eqref{ 5number}, \eqref{general_differences_rotations2}, and \eqref{localized_elastic_energy_on_good_cuboids}, we therefore estimate
\begin{equation*}
\int_{\Omega_{h,\rho}}|\nabla w_h-R_h|^2\,\mathrm{d}x\leq  \OOO C\EEE\sum_{i\in I^h_{\rm{g}}}\eps_{i,h}\leq Ch^3\,,
\end{equation*}
which proves \eqref{w_h_almost_the same}(ii). Then,   the gradient bound in \eqref{w_h_almost_the same}(iii) follows from \eqref{w_h_almost_the same}(ii) and \eqref{almost_the same_volume_and_control_of_energy}(i). 
The $L^\infty$-bound in \eqref{w_h_almost_the same}(iii) follows from the definition \eqref{eq:w_h_construcc} and the uniform control in  \eqref{properties_of_Sobolev_replacement}(iii).The surface area bound therein is a consequence of \eqref{eq: whjum}, \eqref{cardinality_of_bad_cuboids}, and \eqref{eq:card_of_non_inter_indices}. Finally, \eqref{w_h_almost_the same}(iv) follows directly from \eqref{eq: whjum} for $\Q_{v_h}:=\{Q_h(i)\in \Q_h\colon i\in I^h_{\rm{b}} \cup \OOO \OOO I^{h,\rho}_{\rm{ext}}\EEE\}$, where we recall once again \eqref{cardinality_of_bad_cuboids} and \eqref{eq:card_of_non_inter_indices}.
The proof is now complete.
\end{proof}

\section{Proof of Theorem \ref{thm:main_kirchhoff_cptness}}\label{compactness}

We again use  the continuum  subscript $h>0$ instead of the sequential subscript notation $(h_j)_{j\in \N}$ for convenience.

\begin{proof}[Proof of Theorem \ref{thm:main_kirchhoff_cptness}] We split the proof into two steps.

\begin{step}{(1): Compactness for the voids}
By the energy bound \eqref{uniform_rescaled_energy_bound} and  \eqref{eq: newenergy} we have that
\begin{equation}\label{energy_bound_in_the_order_1_domain}
h^{-2}\int_{\Omega\setminus \overline{V_h}} W(\nabla_hy_h(x))\,\mathrm{d}x+\int_{\partial V_h\cap \Omega}\big|\big(\nu^1_{V_h}(z), \nu^2_{V_h}(z), h^{-1}\nu^3_{V_h}(z)\big)\big|\, \mathrm{d}\H^2(z)\leq C\,,
\end{equation}
where $\nu_{V_h}(z):=(\nu^1_{V_h}(z),\nu^2_{V_h}(z), \nu^3_{V_h}(z))$ denotes the outward pointing unit normal to $\partial V_h\cap \Omega$ at the point $z$. Note that \eqref{energy_bound_in_the_order_1_domain} implies  
$$\sup_{h>0}\big(\L^3(V_h)+\H^2(\partial V_h\cap \Omega)\big)\leq C\,.$$
Hence, by the standard compactness result for sets of finite perimeter, cf.~\cite[Theorem 3.39]{Ambrosio-Fusco-Pallara:2000}, there exists $\tilde V\in \mathcal{P}(\Omega)$ such that,  up to a non-relabeled subsequence, we have
 \begin{equation}\label{V_j_converge_to_V}
\chi_{V_h}\to \chi_{\tilde V}\ \ \text{in\ } L^1(\Omega)\,.
\end{equation}
Invoking also \textit{Reshetnyak's lower semicontinuity theorem}, cf.~\cite[Theorem 2.38]{Ambrosio-Fusco-Pallara:2000}, applied to the lower semicontinuous, positively 1-homogeneous, convex function $\phi:\mathbb{R}^2\to [0,+\infty)$ with $\phi(\nu):=|\nu^3|$, and using again  \eqref{energy_bound_in_the_order_1_domain}, we obtain
\begin{equation*}
\int_{\partial^* \tilde  V\cap \Omega} |\nu_{\tilde V}^3| \,\mathrm{d}\H^2\leq \liminf_{h\to 0}\int_{\partial V_h\cap \Omega}|\nu_{V_h}^3|\,\mathrm{d}\H^2\leq C\liminf_{h\to 0}h=0\,, 
\end{equation*}
where $\nu_{\tilde V}$ denotes the  measure-theoretic outer unit normal to $\partial^* \tilde V$. This implies   that 

\begin{equation}\label{eq:planarity_of_tilde_V}
\nu^3_{\tilde V}(x)=0 \ \text{ for }\H^2\text{-a.e. } x\in \partial^*\tilde V\cap\Omega\,.
\end{equation}
In order to prove that the set $\tilde V$ is \OOO cylindrical over a  two-dimensional set\EEE, we proceed as follows. Let
\begin{equation}\label{eq:proj_V}
V:=\big\{(x_1,x_2)\in S\colon \mathcal{H}^1\big(((x_1,x_2)\times\R)\cap \tilde V\big)>0\big\}\,.	
\end{equation} 
Our aim is to prove that 
\begin{equation}\label{eq:tilde_V_eq_proj_V}
\mathcal{L}^3\Big(\tilde V\triangle \big(V\times(-\tfrac{1}{2}, \tfrac{1}{2})\big)\Big)=0\,.
\end{equation}
In view of Fubini's theorem,  for the verification of \eqref{eq:tilde_V_eq_proj_V}  it is enough to show that 
\begin{equation}\label{eq:slices_are_0}
\mathcal{H}^1\Big(\big[\tilde V\triangle \big(V\times(-\tfrac{1}{2}, \tfrac{1}{2})\big)\big]\cap\big((x_1,x_2)\times \R\big)\Big)=0 \ \text{for } \mathcal{L}^2\text{-a.e. } (x_1,x_2)\in S\,.
\end{equation}
Trivially, for every $(x_1,x_2)\notin V$ we have by \eqref{eq:proj_V}
\begin{equation}\label{eq:zero_proj_case_1}
\mathcal{H}^1\big(\big[\tilde V\triangle \big(V\times(-\tfrac{1}{2}, \tfrac{1}{2})\big)\big]\cap\big((x_1,x_2)\times \R\big)\big)\leq \mathcal{H}^1\big(\tilde V\cap\big((x_1,x_2)\times \R\big)\big)=0\,.
\end{equation} 
On the other hand, if $(x_1,x_2)\in V$, we can use \eqref{eq:planarity_of_tilde_V} and the coarea formula, cf.~\cite[Formula 4.36]{maggi2012sets}, to obtain
\begin{equation*}
0=\int_{\partial^*\tilde V\cap \Omega}|\nu^3_{\tilde V}|\,\mathrm{d}\mathcal{H}^2=\int_S\mathcal{H}^0\big((\partial^*\tilde V\cap\Omega) \cap \big((x_1,x_2)\times \R\big)\big)\,\mathrm{d}x_1\mathrm{d}x_2\,.
\end{equation*}
In particular, for $\mathcal{L}^2$-a.e.\ $(x_1,x_2)\in V$ we find that $\mathcal{H}^0((\partial^*\tilde V  \cap\Omega) \cap ((x_1,x_2)\times \R))=0$, which further  implies   that 
\begin{equation}\label{eq:zero_proj_case_2}
\mathcal{H}^1\big(\big[\tilde V\triangle \big(V\times(-\tfrac{1}{2}, \tfrac{1}{2})\big)\big]\cap\big((x_1,x_2)\times \R\big)\big) =0. 
\end{equation}
Now, \eqref{eq:zero_proj_case_1}--\eqref{eq:zero_proj_case_2} imply \eqref{eq:slices_are_0}. As discussed above, this in turn gives \OOO\eqref{eq:tilde_V_eq_proj_V}, which in particular implies that $V\in \mathcal{P}(S)$. \EEE Then, \eqref{V_j_converge_to_V} yields \eqref{compactness_properties}(i). 
\end{step}

\medskip
\begin{step}{(2): Compactness for the deformations}
Let $(v_h,E_h)_{h>0}$ be the sequence related to the sequence $(y_h,V_h)_{h>0}$ via \eqref{order_1_domain}\OOO-\EEE\eqref{from_v_to_y}, and let us fix $\rho >0$ sufficiently small. In order to show \eqref{compactness_properties}(ii),(iii) for the sequence $(y_h)_{h>0}$, we first consider the fields $(\tilde r_h)_{h>0}\subset SBV^2(\Omega_{1,\rho};\R^3)$ and $(\tilde R_h)_{h>0}\subset SBV^2(\Omega_{1,\rho};\R^{3\times 3})$, defined via 
\begin{equation}\label{eq:r_h_R_h_order_1}
\tilde r_h(x):= r_h(T_h x) \ \OOO \text{ and } \ \EEE
\tilde R_h(x):= R_h(T_h x)\,,
\end{equation}
where we recall \eqref{eq:r_h_construcc}, \eqref{eq:R_h_construcc}, and \eqref{anisotropic_dilation}.  Let us mention once again   that   $\tilde r_h$, $\tilde R_h$ \MMM may \EEE depend also on $\rho$, which we do not include in the notation for simplicity. \EEE By \eqref{almost_the same_volume_and_control_of_energy}(i),(iii) we can estimate
\begin{align}\label{eq:r_h_SBV_bound1.0}
\|\tilde r_h\|_{L^\infty(\Omega_{1,\rho})}+\|\nabla \tilde r_h\|_{L^2(\Omega_{1,\rho})} & \leq \|\tilde r_h\|_{L^\infty(\Omega_{1,\rho})}+\|\nabla_h \tilde r_h\|_{L^2(\Omega_{1,\rho})} \notag \\ & \le \|r_h\|_{L^\infty(\Omega_{h,\rho})}+h^{-\frac{1}{2}}\|\nabla r_h\|_{L^2(\Omega_{h,\rho})}  \leq C\,.
\end{align}
 By \eqref{almost_the same_volume_and_control_of_energy}(ii) and a simple change of variables (analogously to \eqref{energy_bound_in_the_order_1_domain}), we also get
\begin{align}\label{eq:r_h_SBV_bound2.0}
\H^2(J_{\tilde r_h})&\leq  \int_{J_{\tilde r_h}}|(\nu^1_{J_{\tilde r_h}}, \nu^2_{J_{\tilde r_h}}, h^{-1}\nu^3_{J_{\tilde r_h}})| \, \mathrm{d}\H^2 \OOO=\EEE h^{-1}\H^2(J_{r_h}) \leq C\,.
\end{align} 
\EEE Exactly in the same fashion, we also have 
\begin{equation}\label{eq:r_h_SBV_bound3.0}
\|\tilde R_h\|_{L^\infty(\Omega_{1,\rho})}+\| \MMM \nabla_h \EEE \tilde R_h\|_{L^2(\Omega_{1,\rho})}+\H^2(J_{\tilde R_h})\leq C\,.
\end{equation}
Thus, we can apply \textit{Ambrosio's SBV compactness theorem}, cf.~\cite[Theorems 4.8]{Ambrosio-Fusco-Pallara:2000}, to obtain fields $r_\rho\in SBV^2(\Omega_{1,\rho};\R^3)$ and $R_\rho\in SBV^2(\Omega_{1,\rho};\R^{3\times 3})$ such that, up to  subsequences (not  relabeled), 
\begin{align}\label{eq:convergence_of_r_h_R_h}
\begin{split}	
\rm{(i)}&\quad\tilde r_h\to r_\rho\ \mathrm{strongly\ in\ } L^2(\Omega_{1,\rho};\R^3)\,,\quad\qquad  \nabla\tilde r_h\rightharpoonup \nabla r_\rho\ \ \mathrm{weakly\ in\ } L^2(\Omega_{1,\rho};\R^{3\times 3})	\\
\rm{(ii)}&\quad\tilde R_h\to R_\rho\ \mathrm{strongly\ in\ } L^2(\Omega_{1,\rho};\R^{3\times 3})	\,,\quad \ \nabla \tilde R_h\rightharpoonup \nabla R_\rho\ \mathrm{weakly\ in\ } L^2(\Omega_{1,\rho};\R^{3\times3\times 3})\,.
\end{split}
\end{align}	
We now verify that 
\begin{align}\label{eq:structure_of_limiting_fields}
\begin{split}
\rm{(i)}&\quad\partial_3 r_\rho= 0\ \ \OOO \L^3\text{-a.e. in }\EEE \Omega_{1,\rho}\,, \ \ R_\rho\in SO(3) \ \text{with } \partial_3 R_\rho= 0\ \ \OOO \L^3\text{-a.e. in }\EEE \Omega_{1,\rho}\,.\\
\rm{(ii)}&\quad \nabla'r_\rho=R'_\rho\,,\ \text{where } R'_\rho:=\big(\OOO R_\rho e_1|R_\rho e_2\EEE\big)\in SBV^2(\Omega_{1,\rho};\R^{3\times2})\,
\end{split}
\end{align}
\MMM where we recall the notation $\nabla' :=( \partial_1 , \partial_2)$. \EEE
Indeed, by the lower semicontinuity of the $L^2$-norm under weak convergence,  \eqref{eq:r_h_SBV_bound1.0}, and \eqref{eq:r_h_SBV_bound3.0} we obtain  
\begin{align*}
\|\partial_3 r_\rho\|_{L^2(\Omega_{1,\rho})}+\|\partial_3 R_\rho\|_{L^2(\Omega_{1,\rho})}&\leq \liminf_{h\to 0}\|\partial_3 \tilde r_h\|_{L^2(\Omega_{1,\rho})}+\liminf_{h\to 0}\|\partial_3 \tilde R_h\|_{L^2(\Omega_{1,\rho})}\nonumber \\
&\leq \liminf_{h\to 0} \big( h\|\nabla_h\tilde r_h\|_{L^2(\Omega_{1,\rho})}\big)+ \liminf_{h\to 0} \big(h\|\nabla_h\tilde R_h\|_{L^2(\Omega_{1,\rho})}\big)\nonumber \\
&\leq C\liminf_{h\to 0} h=0\,.
\end{align*}
Moreover, by \OOO \eqref{eq:r_h_R_h_order_1}, \EEE\eqref{eq:R_h_construcc}, \eqref{eq:partition_of_unity_2}--\eqref{eq:partition_of_unity_1}, \OOO \eqref{eq:neighbouring}, \EEE \eqref{general_differences_rotations2}, and \eqref{localized_elastic_energy_on_good_cuboids}, 
\begin{align}\label{eq:dist_h_to_SO(3)}
\int_{\Omega_{1,\rho}}\mathrm{dist}^2(\tilde R_h\OOO ,\EEE SO(3))\,\mathrm{d}x&= h^{-1}\int_{\Omega_{h,\rho}}\mathrm{dist}^2(R_h\OOO,\EEE SO(3))\,\mathrm{d}x\leq h^{-1}\sum_{i\in I^h_{\rm{int}}}\int_{Q_{h}(i)\cap \Omega_{h,\rho}}|R_h-R_{i,h}|^2\,\mathrm{d}x \nonumber \\
&\leq   h^{-1}\sum_{i\in I^h_{\rm{int}}}\int_{Q_{h}(i)\cap \Omega_{h,\rho}}\Big|\sum_{j\in\mathcal{N}(i)}\psi_h^j(x)(R_{j,h}-R_{i,h})\Big|^2\,\mathrm{d}x \nonumber \\
&\leq Ch^{-1}h^3\sum_{i\in I^h_{\rm{int}}}\ \sum_{j\in\mathcal{N}(i)}|R_{j,h}-R_{i,h}|^2\leq Ch^{-1}\sum_{i\in I^h_{\rm{int}}}\ \sum_{j\in \mathcal{N}(i)}(\eps_{i,h}+\eps_{j,h})\nonumber\\
&\leq Ch^{-1}\sum_{i\in I^h_{\rm{g}}}\eps_{i,h}\leq Ch^2\,.
\end{align}
By passing to the limit as $h\to 0$ in \eqref{eq:dist_h_to_SO(3)}, and using \eqref{eq:convergence_of_r_h_R_h}(ii), we obtain 
\begin{equation*}
\int_{\Omega_{1,\rho}}\mathrm{dist}^2(R_\rho\OOO,\EEE SO(3))\,\mathrm{d}x=\lim_{h\to 0}\int_{\Omega_{1,\rho}}\mathrm{dist}^2(\tilde R_h\OOO,\EEE SO(3))\,\mathrm{d}x=0\Rightarrow\mathrm{dist}(R_\rho\OOO,\EEE SO(3))=0\ \text{ a.e. in } \Omega_{1,\rho}\,.
\end{equation*}
This \MMM concludes the proof of \EEE \eqref{eq:structure_of_limiting_fields}(i). We now get that 
\begin{equation}\label{eq:y_h_and_grad_conv}
y_h\to 	r_\rho \ \text{in measure on }  \Omega_{1,\rho}\,, \quad \nabla_hy_h\to 	R_\rho \ \text{in measure on }  \Omega_{1,\rho}\,,
\end{equation}
\OOO which \EEE follows easily from \eqref{eq:convergence_of_r_h_R_h}(i),(ii), \eqref{eq:r_h_R_h_order_1}, \eqref{from_v_to_y}, and \eqref{almost_the same_volume_and_control_of_energy}(iv) via a change of variables.  In particular, using \eqref{admissible_configurations_h_level}, \eqref{eq: nonlinear energy}(iv), \eqref{eq: newenergy}, and \eqref{uniform_rescaled_energy_bound}, we obtain
\[\sup_{h>0}\left(\|\nabla y_h\|_{L^2(\Omega_{1,\rho})}+ \H^2(J_{y_h}) +\|y_h\|_{L^\infty(\Omega_{1,\rho})}\right)\leq  C\,,\]
so that \eqref{eq:y_h_and_grad_conv} and \textit{Ambrosio's closure theorem} \cite[Theorems 4.7]{Ambrosio-Fusco-Pallara:2000} lead to \eqref{eq:structure_of_limiting_fields}(ii).
\EEE

Recalling \eqref{Omega_h_local}, the convergence in \eqref{eq:y_h_and_grad_conv}, together with a monotonicity argument as $\rho\to 0$, allows us to define fields
$r\in SBV^2(\Omega;\R^3)$ and $R\in SBV^2(\Omega;\R^{3\times 3})$, such that 
\begin{align}\label{eq:structure_of__global_limiting_fields}
\begin{split}
\rm{(i)}&\quad\partial_3 r= 0\ \ \OOO \L^3\text{-a.e. in }\EEE \Omega\,, \ \ R\in SO(3) \ \text{with } \partial_3 R= 0\ \ \OOO \L^3\text{-a.e. in }\EEE \Omega\,,\\
\rm{(ii)}&\quad \nabla'r=R'\,,\ \text{where } R':=\big(\OOO Re_1|Re_2\EEE\big)\in SBV^2(\Omega;\R^{3\times2})\,,
\end{split}
\end{align} 
so that in particular $\OOO Re_3\EEE =\partial_1r\wedge \partial_2r$, and 
\begin{equation}\label{eq:y_h_and_grad_global_conv}
y_h\to 	r \ \text{in measure on }  \Omega\,, \quad \nabla_hy_h\to 	R \ \text{in measure on }  \Omega\,.
\end{equation}
We next prove that
\begin{equation}\label{eq:L2_equiintegrability}
\nabla_hy_h \to R \text{ strongly in }  L^2(\Omega;\mathbb{R}^{3\times3})\,.
\end{equation}
Indeed, defining $L_h \defas \lbrace \OOO x\in \Omega: \EEE  |\nabla_{h}y_{h}\OOO(x)\EEE|> 2\sqrt{3}\rbrace$, by \eqref{energy_bound_in_the_order_1_domain} \OOO and \eqref{eq: nonlinear energy}(iv) \EEE we get 
\begin{equation}\label{eq:for_equiint}
3\mathcal{L}^3(L_h) \le       \frac{1}{4}\int_{L_h}|\nabla_{h}y_{h}|^2\,\mathrm{d}x \leq \int_\Omega \mathrm{dist}^2(\nabla_{h}y_{h},SO(3))\,\mathrm{d}x \leq C{h^2}\,.
\end{equation}
This shows that $\chi_{\Omega \setminus L_h} \to 1$ boundedly in measure on $\Omega$ and thus, in view of \eqref{eq:y_h_and_grad_global_conv}, we get $\OOO\chi_{\Omega \setminus L_h}\EEE\nabla_hy_h  \to R$ strongly in $L^2(\Omega;\mathbb{R}^{3\times3})$. By \eqref{eq:for_equiint} we further get $\Vert \OOO \chi_{L_h} \EEE \nabla_hy_h\Vert_{L^2(\Omega)} \le Ch$, which concludes the proof of \eqref{eq:L2_equiintegrability}. 

Hence, setting $\tilde y:=r$, collecting \eqref{eq:structure_of__global_limiting_fields}--\eqref{eq:L2_equiintegrability} and recalling \MMM the uniform bound on deformations assumed in \EEE \eqref{initial_admissible_configurations}, we derive \eqref{compactness_properties}(ii)-(iii). This concludes the proof of compactness.
\end{step}
\end{proof}

\begin{corollary}\label{compactness_for_w_h}
In the setting of Proposition \ref{Lipschitz_replacement}, let $(\tilde w_h)_{h>0}\subset SBV^2(\Omega_{1,\rho};\R^3)$ be defined by
\begin{equation}\label{eq:w_h_rescaling}
\tilde w_h(x):=w_h(T_hx)\,,
\end{equation}
where $w_h \in  SBV^2(\Omega_{h,\rho};\R^3)$ is  as in \eqref{w_h_almost_the same} and $T_h$ as in \eqref{anisotropic_dilation}.   Then, for   $(y,V)\in \mathcal{A}$ given in Theorem \ref{thm:main_kirchhoff_cptness}, we also have, up to subsequences (not relabeled),
\begin{align}\label{w_h_comp}
\begin{split}
\rm{(i)}\quad& \tilde w_{h}\longrightarrow \tilde y  \text{ in  } L^1(\Omega_{1,\rho};\R^3)  \,,\\
\rm{(ii)}\quad& \nabla_h \tilde w_{h}\to \big(\nabla'\tilde y,\OOO \partial_1 \tilde y\wedge \partial_2\tilde y\EEE\big)  \ \text{ strongly in } L^2(\Omega_{1,\rho};\R^{3\times 3})\,.
\end{split}
\end{align}
\end{corollary}

\begin{proof}
Using \eqref{eq:r_h_SBV_bound1.0}--\eqref{eq:r_h_SBV_bound2.0} with $\tilde w_h$ in place of $\tilde r_h$ (cf.~\eqref{w_h_almost_the same}(iii)) and once again \textit{Ambrosio's SBV compactness theorem}, the corollary follows from \eqref{w_h_almost_the same}(i),(ii) and \eqref{almost_the same_volume_and_control_of_energy}(iv), after a simple change of variables,  together with \eqref{eq:r_h_R_h_order_1} and \eqref{eq:convergence_of_r_h_R_h}(ii).  
 \end{proof}

\section{Proof of Theorem \ref{thm:main_kirchhoff_gamma_conv}{\rm(i)}}\label{gamma_liminf}

In this section we give the proof of the lower bound  of Theorem \ref{thm:main_kirchhoff_gamma_conv}, for which we prove separately the lower bound for the bulk and the surface part of the energy.   Recalling Definition \ref{eq:tau_convergence}, we consider $(y_{h}, V_{h})_{h>0}$ and $(y, V)\in \mathcal{A}$ such that $(y_{h}, V_{h}) \overset{\tau}{\longrightarrow}(y,V)$, i.e., \eqref{compactness_properties}(i)--(iii) hold true. We start with  the lower bound of the elastic energy.

\begin{proposition}\label{elastic_lower_bound}
Suppose that  $(y_{h}, V_{h})\overset{\tau}{\longrightarrow}(y,V)$ for some  $(y,V)\in \mathcal{A}$, \OOO cf.~\eqref{limiting_admissible_pairs}. \EEE Then, 
\begin{equation}\label{lower_bound}
\liminf_{h\to 0} \Big(h^{-2}\int_{\Omega\setminus \overline{V_h} }W(\nabla_hy_h)\,\mathrm{d}x\Big)\geq 
\frac{1}{24}\int_{S\setminus V}\Q_2(\mathrm{II}_y(x'))\,\mathrm{d}x'\,,
\end{equation} 
where we also recall \MMM the definition of $\mathrm{II}_y$ in  \eqref{def:second_FF}. \EEE
\end{proposition}

\begin{proof}
Since it is not restrictive to assume that the sequence of total energies $(\E^{h}(y_{h}, V_{h}))_{h>0}$ is bounded, i.e., that \eqref{Lipschitz_replacement_energy_bound} holds, we can apply Proposition \ref{Lipschitz_replacement} for $\rho >0$ small and the sequence $(v_h, E_h)_{h>0}$ related to $(y_h,V_h)_{h>0}$ via 
\eqref{order_1_domain}--\eqref{from_v_to_y}. 

\MMM Recalling the definition of $I^h_{\rm{int}}$ in \eqref{eq:interior_indices}, \EEE we introduce the sequence of piecewise constant rotation fields $(\bar R_h)_{h>0}\subset SBV^2(\Omega_{1,\rho};\R^{3\times 3})$,  defined by  
\begin{align*}
\begin{split}
\bar R_h(x) := \begin{cases}  \mathrm{Id}& \text{ if } x\in T_{1/h}(Q_{h}(i))\cap\Omega_{1,\rho}\ \text{ for some } i\notin I^h_{\rm{int}}\,,\\
R_{i,h} & \text{ if } x\in T_{1/h}(Q_{h}(i))\cap\Omega_{1,\rho}\ \text{ for some } i\in I^h_{\rm{int}}\,.
\end{cases}
\end{split}
\end{align*}
 With a similar argument as in \eqref{eq:gradient_bound_z_h_1}, \OOO recalling \eqref{eq:r_h_R_h_order_1}, \EEE\eqref{eq:R_h_construcc}, \eqref{eq:partition_of_unity_2}--\eqref{eq:partition_of_unity_1}, \eqref{eq:neighbouring}, \eqref{general_differences_rotations2},  and \eqref{localized_elastic_energy_on_good_cuboids}, we estimate  
\begin{align}\label{eq:rescaled_fields_equibounded_1}
h^{-2}\int_{\Omega_{1,\rho}}|\OOO\tilde R\EEE_h-\bar R_h|^2\,\mathrm{d}x&=h^{-3}\int_{\Omega_{h,\rho}}|R_h-\bar R_h|^2\OOO=\EEE h^{-3}\sum_{i\in I^h_{\rm{int}}}\int_{\Omega_{h,\rho}\cap Q_{h}(i)}\Big|\Big(\sum_{j\in \mathcal{N}(i)}\psi_h^jR_{j,h}\Big)-R_{i,h}\Big|^2 \nonumber\\
&\leq h^{-3}\sum_{i\in I^h_{\rm{int}}}\int_{\OOO Q_{h}(i)\EEE}\Big|\sum_{j\in \mathcal{N}(i)}\psi_h^j(R_{j,h}-R_{i,h})\Big|^2\leq C\sum_{i\in I^h_{\rm{int}}}\sum_{j\in \mathcal{N}(i)}|R_{j,h}-R_{i,h}|^2\nonumber\\
&\leq Ch^{-3}\sum_{i\in I^h_{\rm{int}}}\sum_{j\in \mathcal{N}(i)}(\eps_{j,h}+\eps_{i,h})\leq Ch^{-3}\sum_{i\in I^h_{\rm{g}}}\eps_{i,h}\leq C\,.
\end{align}
Thus, combining \eqref{eq:rescaled_fields_equibounded_1} with \eqref{w_h_almost_the same}(ii) and using a change of variables, we deduce \EEE
\begin{equation}\label{eq:rescaled_fields_equibounded}
\sup_{h>0}\|h^{-1}(\bar R_h^T\nabla_h\tilde w_h-\mathrm{Id})\|_{L^2(\Omega_{1,\rho})}\leq C\,,
\end{equation}
where $\tilde w_h$ is defined as in \eqref{eq:w_h_rescaling}. 
In particular,  we deduce that there exists $G\in L^2(\Omega_{1,\rho};\R^{3\times 3})$ such that, up to a  subsequence (not relabeled),
\begin{equation}\label{eq:rescaled_fields_convergence}
G_h:=\frac{\bar R_h^T\nabla_h\tilde w_h-\mathrm{Id}}{h}\rightharpoonup G\ \ \text{weakly in } L^2(\Omega_{1,\rho};  \MMM \R^{3\times 3})\,. \EEE
\end{equation}
 
We can then proceed with a classical linearization argument, cf.~\cite{friesecke2002theorem, SantiliSchmidt2022}, which we nevertheless detail in Appendix \ref{sec: standard_liminf} for the reader's convenience. This leads to

\begin{equation}\label{aux_lower_bound}
\liminf_{h\to 0} \Big(h^{-2}\int_{\Omega\setminus \overline{V_h} }W(\nabla_hy_h)\,\mathrm{d}x\Big)\geq 
\frac{1}{2}\int_{\Omega_{1,\rho}}\Q_3(G)\,\mathrm{d}x\geq 
\frac{1}{2}\int_{\Omega_{1,\rho}}\Q_2(G')\,\mathrm{d}x\,,
\end{equation} 
where \EEE we recall \eqref{def:Q_2}, and where for a matrix $F \in \R^{3 \times 3}$ we   use the notation $F'\in \R^{2\times 2}$ for its upper-leftmost ($2\times2$)-block.

Hence, as in the purely elastic setting, we are confronted with identifying the upper-left block $G'$ of the weak limit $G$ in \eqref{eq:rescaled_fields_convergence}. Although the sequence $(\tilde w_h)_{h>0}$ is not Sobolev in the entire domain, due to its construction, cf.~\eqref{w_h_almost_the same} \OOO and \eqref{eq:w_h_rescaling}\EEE, it does not exhibit jumps in the transversal $x_3$-direction, so that the strategy originally devised in \cite[Proof of Theorem 6.1(i)]{friesecke2002theorem} is applicable. We next give the details of the proof.

Let $\tilde G_h:=(\OOO G_he_1|G_he_2\EEE)\in \R^{3\times 2}$ denote the matrix of the first two columns of $G_h$, and similarly \MMM $\tilde{G}$ \EEE for the weak $L^2$-limit $G$. In the slightly smaller domain $\Omega_{1,2\rho}$, we consider the finite difference quotients in $x_3$-direction, $(H_h)_{h>0}\subset L^2(\Omega_{1,2\rho};\R^{3\times 2})$, defined by 
\begin{equation}\label{eq:H_diff_quotient}
H_h(x',x_3):=\frac{\tilde G_h(x',x_3+z)-\tilde G_h(x',x_3)}{z}=(\bar R_h)^T\frac{\frac{1}{h}\nabla'\tilde w_h(x',x_3+z)-\frac{1}{h}\nabla'\tilde w_h(x',x_3)}{z}\,, 
\end{equation}
where $|z|\leq \rho$ ($z\neq 0$). In view of \eqref{eq:rescaled_fields_convergence}, we have 
\begin{equation}\label{eq:weak_conv_H_diff_quotient}
H_h\rightharpoonup H:=\frac{\tilde G(x',x_3+z)-\tilde G(x',x_3)}{z} \ \text{ weakly in } L^2(\Omega_{1,2\rho};\R^{3\times 2})\,. 
\end{equation}
By \eqref{eq:rescaled_fields_equibounded} and \eqref{w_h_comp}(ii), $(\bar R_h)_{h>0}$ converges boundedly in measure to $(\nabla'\tilde y|b_{\tilde y}) \in SBV^2(\Omega_{1,2\rho};\R^{3\times 3})$, where 
\begin{equation}\label{eq:tilde b}
b_{\tilde y}:=\OOO \partial_1 \tilde y\wedge \partial_2\tilde y\EEE\,,
\end{equation}
so that by \eqref{eq:H_diff_quotient}--\eqref{eq:tilde b} we obtain 
\begin{equation}\label{eq:weak_conv_diff_quotient}
\frac{\frac{1}{h}\nabla'\tilde w_h(x',x_3+z)-\frac{1}{h}\nabla'\tilde w_h(x',x_3)}{z}	\rightharpoonup (\nabla'\tilde y| b_{\tilde y})H\ \text{ weakly in } L^2(\Omega_{1,2\rho};\R^{3\times 2})\,.
\end{equation}
In order to identify the weak limit $H$, we can then argue as in the end of the proof of the lower bound for the elastic energy in \cite[Section 5]{SantiliSchmidt2022}. Setting
\begin{equation}\label{eq:difference_quotient}
f_h(x',x_3)   :=   \frac{\tilde w_h(x',x_3+z)-\tilde w_h(x',x_3)}{hz}\in SBV^2(\Omega_{1,2\rho};\R^3),
\end{equation}
we observe that, by \eqref{w_h_almost_the same}(iv) and a slicing argument\OOO,\EEE 
\begin{equation*}
f_h(x',x_3)=\int_{0}^1\frac{1}{h}\partial_3\tilde w_h(x',x_3+tz)\,\mathrm{d}t\,.
\end{equation*}
By Corollary \ref{compactness_for_w_h}\OOO, see \eqref{w_h_comp} and \eqref{eq:tilde b}, \EEE we have that $\frac{1}{h}\partial_3\tilde w_h\OOO \rightarrow \EEE b_{\tilde{y}}$ \OOO strongly in \EEE $L^2(\Omega_{1,\rho};\R^{\OOO3\EEE})$, and therefore 
\begin{equation*}
f_h\to  \int_{0}^1  b_{\tilde y} (\cdot,\cdot +tz)\,\mathrm{d}t =       b_{\tilde y} \ \text{ strongly in } L^2(\Omega_{1,2\rho};\R^3)\,, 
\end{equation*}
where the last equality follows from the fact that $b_{\tilde y}$ is independent of $x_3$. \OOO By \eqref{eq:weak_conv_diff_quotient}, \eqref{w_h_comp}(ii), and \EEE \eqref{w_h_almost_the same}(iii), we have
\begin{equation*}
\sup_{h>0} \| \nabla f_h \|_{L^2(\Omega_{1,2\rho})} <+\infty \quad \text{and} \quad \sup_{h>0}\mathcal{H}^2(J_{f_h})<+\infty\,,
\end{equation*}
so that, by the basic closure theorem in $SBV$, cf.~\cite[Theorem 4.7]{Ambrosio-Fusco-Pallara:2000}, we deduce that  
\begin{equation}\label{eq:gradient_f_h_conv}
\nabla'f_h\rightharpoonup \nabla'b_{\tilde y} \ \text{ weakly in } L^2(\Omega_{1,2\rho};\R^{3\times 3}) \,.
\end{equation}
Combining \eqref{eq:weak_conv_diff_quotient}, \eqref{eq:gradient_f_h_conv}, the fact that $(y,V)\in \mathcal{A}$  (see \eqref{limiting_admissible_pairs}), and recalling the \MMM identification \eqref{eq:y_ident}, \EEE  and \eqref{def:second_FF}, we obtain 
\begin{equation*}
H=(\nabla' \tilde y| \MMM b_{\tilde{y}})^T\nabla'b_{\tilde{y}} \EEE  \in SBV^2(\Omega_{1,2\rho};\R^{3\times 2})\,.
\end{equation*}
Thus, \eqref{eq:weak_conv_H_diff_quotient} implies 
\begin{equation}\label{eq:identification_of_G}
\tilde G(x', x_3)=\tilde G(x', 0) +x_3H (x')\   \text{ for } (x',x_3) \in\Omega_{1,2\rho}\,.
\end{equation}
Using \OOO the bilinearity of $\Q_2$ in \eqref{def:Q_2}\EEE, \eqref{eq:identification_of_G}, that $\int_{-1/2+2\rho}^{1/2-2\rho} x_3\,\mathrm{d}x_3=0$, \OOO and the definition of the second fundamental form \eqref{def:second_FF}, \EEE we obtain
\begin{align}\label{eq:almost_final_liminf}
\begin{split}	
\OOO \frac{1}{2}\int_{\Omega_{1,\rho}}\Q_2(G')\,\mathrm{d}x
&\OOO =\EEE \frac{1}{2}\int_{\Omega_{1,2\rho}} \Q_2(G'(x',0))\,\mathrm{d}x+\frac{1}{2}\int_{\Omega_{1,2\rho}} x_3^{2}\Q_2(\mathrm{II}_y\OOO(x')\EEE)\,\mathrm{d}x\,,\\
&\geq \frac{1}{2}\int_{\Omega_{1,2\rho}} x_3^{2}\Q_2(\mathrm{II}_y\OOO(x')\EEE)\,\mathrm{d}x\,.
\end{split}
\end{align}
Then, \eqref{lower_bound} follows from \OOO\eqref{aux_lower_bound} and \eqref{eq:almost_final_liminf}\EEE, after letting $\rho\to 0$.
\end{proof}

\bigskip

We now proceed with the lower bound for the surface part of the energy, namely  
\begin{equation*}
\mathcal{E}^h_{\mathrm{surf}}(V_h):= \mathcal{E}^h(y_h,V_h)-h^{-2}\int_{\Omega\setminus \overline{V_h}}W(\nabla_hy_h) \, \mathrm{d}x=h^{-1}\mathcal{G}^{\kappa_h}_{\mathrm{surf}}(E_h; \Omega_h)\,,
\end{equation*}
where we refer to \eqref{rescaled_energy} and  \eqref{F_surf_energy}.  Our approach deviates significantly from the proof of lower bounds in relaxation results for energies defined on pairs \OOO of deformations and sets\EEE, cf.~\cite{BraChaSol07}, \cite{CrismaleFriedrich}, \cite{SantiliSchmidt2022}, the main reason being that our piecewise nonlinear geometric rigidity result allows for a control \OOO only \EEE in a large part of $\Omega \setminus \overline{V_h}$\EEE.

While the justification for the \OOO middle \EEE term on the right-hand side of \eqref{limiting_two_dimensional_energy} is standard, for the \OOO last \EEE one therein the argument is based on a fine blow-up analysis around jump points of $J_{(y,\nabla'y)}\setminus \partial^*V$. In particular, a delicate contradiction argument is employed to obtain the desired density lower bound of the surface energy. \OOO For technical reasons, the latter \EEE is augmented with a vanishing contribution of the elastic energy, see \eqref{eq:blow-up_measures} and \eqref{eq:density_estimate} below. A suitable (two-dimensional in nature) blow-up of the deformations, their derivatives, as well as the void sets (cf.~\eqref{eq:blow_up_void_function}), together with a De-Giorgi type argument, will allow us to identify, up to translations, an appropriate three-dimensional thin rod on which the sequence $(v_h,E_h)_{h>0}$ enjoys uniform energy bounds, related to the bending energy for thin rods with voids, see \cite[Equations (2.3) and (2.8)]{KFZ:2023}. Finally, our compactness result from \cite[Theorem 2.1]{KFZ:2023} and the particular structure of the limiting pair will allow us to conclude the contradictory argument. 
 \EEE

\begin{proposition}\label{surface_lower_bound}
Suppose that  $(y_{h}, V_{h})\overset{\tau}{\longrightarrow}(y, V)$ for some  $(y, V)\in \mathcal{A}$.  Then,
\begin{equation}\label{surface_part_lower_bound}
\liminf_{h\to 0} \E^h_{\mathrm{surf}}(V_h)\geq \H^1(\partial^*V\cap S)+2\H^1\big(J_{(y, \nabla'y)}\setminus \partial^*V\big)\,.
\end{equation} 
\end{proposition}

\begin{proof} Let $(E_h)_{h>0}$ be the void sets associated to  $(V_h)_{h>0}$ according to \eqref{order_1_domain}. Let $(\tilde y,\tilde V)$ be the pair associated to $(y,V)$ as in \eqref{eq:y_ident} and \eqref{eq:V_ident} respectively. 

\begin{step}{(1): Blow-up argument} 	
In order to prove \eqref{surface_part_lower_bound}, we perform a blow-up argument.  \MMM Let \EEE $h>0$ and $\eta>0$ ($\eta$  will be eventually sent to $0$ after we send $h\to 0^+$). \MMM We \EEE introduce the family of Radon measures $\mu_{\eta,h}:\mathfrak{M}(S)\to \R_+$, defined by 
\begin{equation}\label{eq:blow-up_measures}
\mu_{\eta, h}(K):=\eta h^{-3}\int_{\left(K\times\left(-\frac{h}{2},\frac{h}{2}\right)\right)\setminus \overline{E}_h}W(\nabla v_h)\,\mathrm{d}x+h^{-1}\mathcal{G}^{\kappa_h}_{\mathrm{surf}}\left(E_h; K\times\left(-\frac{h}{2},\frac{h}{2}\right)\right)\,,
\end{equation}
where we recall \eqref{from_v_to_y} \MMM and \EEE \eqref{F_surf_energy}. By the assumption  that the sequence $(\mathcal{E}^h(y_h,V_h))_{h>0}$ is bounded (cf.~\OOO also \EEE the proof of Lemma \ref{elastic_lower_bound}), and after passing to a subsequence (not relabeled), we may suppose that $(\mu_{\eta,h})_{h>0}$ converges weakly* to some Radon measure $\mu_\eta$.  Let also 
\begin{equation*}
\lambda:=\H^1\llcorner_{(\partial^* V\cup J_{( y,\OOO\nabla'\EEE y)})\cap S}\,,
\end{equation*}
and $\mathrm{d}\mu_\eta/\mathrm{d}\lambda$ be the corresponding Radon-Nikodym derivative. In view of the lower semicontinuity of the mass under weak*-convergence, the equi-boundedness of the total energy, and the arbitrariness of $\eta>0$, the estimate in \eqref{surface_part_lower_bound}  will follow by proving  that for every Lebesgue point of $\mu_{\eta}$ with respect to $\lambda$ there holds
\begin{align}\label{eq:density_estimate}
\frac{\mathrm{d}\mu_\eta}{\mathrm{d}\lambda}(x_0)\geq \begin{cases}1\,, \ \text{ if } x_0\in \partial^* V\cap S\,,\\
2\,, \ \text{ if } x_0\in J_{( y,\OOO\nabla'\EEE  y)}   \setminus \partial^*  V\,. \EEE	
\end{cases}
\end{align}
Now, fix $x_0 \in (\partial^* V\cup J_{(y,\OOO\nabla'\EEE y)}) \cap S$ such that a generalized unit normal  at the point $x_0$ exists, which  we denote by $\nu(x_0)$. Since this property holds for $\mathcal{H}^{\OOO 1\EEE}$-a.e.\ point, it suffices to prove \eqref{eq:density_estimate} in this case.  Without loss of generality we assume that $x_0=0$ and $\nu \MMM (x_0) \EEE =e_1$. For $r<1$ we let $Q_r := (-\frac{r}{2} ,\frac{r}{2})^2$. Noting that $\lambda(Q_r)=r+o(r)$ as $r\to 0$,  in order to prove \eqref{eq:density_estimate}, it suffices to show that 
\begin{align}\label{eq:h_density_estimate}
\liminf_{r\to 0}\liminf_{h\to 0}	\frac{\mu_{\eta,h}(Q_r)}{r}\geq \begin{cases}1\,, \ \text{ if } x_0\in \partial^* V\cap S\,,\\
2\,, \ \text{ if } x_0\in  J_{( y,\OOO\nabla'y\EEE)}   \setminus \partial^*  V\,. \EEE	
\end{cases}
\end{align}
\end{step} 

\begin{step}{(2): Boundary of voids}	 	Regarding the first case in \eqref{eq:h_density_estimate}, for $0\in \partial^* V\cap S$, with a change of variables, \OOO cf.~\eqref{eq: newenergy}\EEE, we can estimate from below
\begin{align}\label{eq:density_1}
\begin{split}
\mu_{\eta,h}(Q_r)&\geq h^{-1}\H^2\left(\partial E_h\cap \left( Q_r\times \left(-\frac{h}{2},\frac{h}{2}\right)\right)\right)\\&=\int_{\partial V_h\cap \left(Q_r \times\left(-\frac{1}{2},\frac{1}{2}\right)\right)}\big|\big(\nu^1_{V_h}, \nu^2_{V_h}, h^{-1}\nu^3_{V_h}\big)\big|\, \mathrm{d}\H^2\geq \mathcal{H}^2\left(\partial V_h\cap \left(Q_r\times\left(-\frac{1}{2},\frac{1}{2}\right)\right)\right)\,.
\end{split}
\end{align}
Since $V_h \to \tilde V$ in $L^1(\Omega)$, the $L^1$-lower semicontinuity of the perimeter, \eqref{eq:density_1}, \eqref{eq:V_ident},  and the fact that $0\in \partial^* V\cap S$ imply  
\begin{align*}
\begin{split}
\hspace{-0.25em}\liminf_{r\to 0}\liminf_{h\to 0}	\frac{\mu_{\eta,h}(Q_r)}{r}&\geq 
\liminf_{r\to 0}\liminf_{h\to 0} \frac{ \mathcal{H}^2\left(\partial V_h\cap \left(Q_r\times\left(-\frac{1}{2},\frac{1}{2}\right)\right)\right)}{r}
\\&\geq \liminf_{r\to 0} \frac{ \mathcal{H}^2\big(\partial^{\OOO *\EEE} \tilde V\cap \left(Q_r\times\left(-\frac{1}{2},\frac{1}{2}\right)\right)\big)}{r}=\liminf_{r\to 0} \frac{ \mathcal{H}^1\left(\partial^{\OOO *\EEE}V\cap Q_r\right)}{r}\OOO=\EEE 1\,.
\end{split}
\end{align*}
\end{step}

\begin{step}{(3): Jump points} Regarding the second case in \eqref{eq:h_density_estimate}, let $0\in J_{( y,\OOO\nabla'\EEE  y)}   \setminus \partial^*  V$, in particular we have $0\in V^0$, where  by $V^0$ we denote  the set of points with two-dimensional density $0$ with respect to $ V$.  We now define   the auxiliary fields \EEE $Y_h \colon \Omega_h \to \mathbb{R}^3\times \mathbb{R}^{3\times 3}$  and $ Y  \colon S \to \mathbb{R}^3\times \mathbb{R}^{3\times 3}$ as \[{Y}_h:=  (v_h, \nabla v_h) \ \text{  and } Y:=\big({y},(\nabla' {y}, \partial_1y\wedge  \partial_2y)\big)\,,\] respectively, and observe that 
$$J_{{Y}} = J_{( y,\nabla'  y)}\,.$$ 
The assumption   $0\in V^0\cap J_{( y,\OOO\nabla'\EEE y)}   = V^0\cap J_{ Y}$  together with \eqref{compactness_properties}(ii),(iii) and a scaling argument implies  
\begin{align}\label{eq:blow_up_void_function}
\begin{split}
\mathrm{(i)}&\quad \lim_{r\to 0}  \lim_{h\to 0} \frac{\L^3\left(E_h\cap \left( Q_r\times\left(-\frac{h}{2},\frac{h}{2}\right)\right)\right)}{r^2 h}=0\,,\\[3pt]
\mathrm{(ii)}&\quad \lim_{r\to 0}  \lim_{h\to 0}  \fint_{Q_r  \times \left(-\frac{h}{2},\frac{h}{2}\right)}|Y_h- Y^{\pm}|\,\mathrm{d}x=0\,,
\end{split}
\end{align}
where 
\begin{equation}\label{eq:jump_function}
Y^{\pm}:=\begin{cases}  Y^{+}(0)\,, \quad \text{if } x_1>0\\
 Y^{-}(0)\,, \quad \text{if } x_1\leq 0\,,
\end{cases}
\end{equation}
with $ Y^+(0)=({y}^+,{R}^+)$ and $Y^{-}(0)=({y}^-,{R}^-)$ being the one-sided traces of $ Y$ at $0\in J_{Y}$. Suppose by contradiction that the desired assertion was false, i.e., there exists $ 0<\delta<1$ such that
\begin{equation*}
\liminf_{r \to 0} \liminf_{h\to 0} \frac{\mu_{\eta,h}(Q_r)}{r} < 2-\delta\,.
\end{equation*}
We will proceed to show a contradiction, namely we will show that
\begin{align}\label{eq:jump-equalities}
{\rm (a)}\quad {y}^+= {y}^- \quad  \quad \text{and} \quad \quad {\rm (b)}\quad   {R}^+= {R}^-\,.
\end{align} 
This implies that $ Y^+(0)= Y^{-}(0)$ and thus $0 \notin J_Y$: a contradiction.  Up to passing to subsequences in $r$ and $h$, for all $0<r\leq r_0(\delta)$ and $0<h\leq h_0( \OOO r\EEE )$  small enough,
\begin{align}\label{ineq:cont_ass_uniform}
\mu_{\eta,h}(Q_r)\leq(2-\delta)r\,.
\end{align}
\end{step}

\begin{step}{(4): Preparations for the proof of \eqref{eq:jump-equalities}} The following arguments will be performed for fixed $r>0$, which will be chosen sufficiently small along the proof. For $j \in \mathbb{Z}$ we define the (pairwise disjoint) stripes 
\begin{align*}
S_h(j) := hje_2+\Big(-\frac{r}{2},\frac{r}{2}\Big)\times\Big(-\frac{h}{2},\frac{h}{2}\Big)\,,
\end{align*}
and note that $S_h(j) \subseteq Q_r$ for all $|j|\leq \frac{r}{2h}-\frac{1}{2}$. We set $N_h:=\lfloor \frac{r}{2h}-\frac{1}{2}\rfloor$ and define
\begin{align} \label{def:goodstripes_surface}
\mathcal{S}_h^\mathrm{good}:=\left\{j \in \mathbb{Z} \colon |j|\leq N_h\,,\ \  \mu_{\eta,h} (S_h(j) ) <\left(2 -\frac{\delta}{2}\right)  h \right\}\,,
\end{align}
and $\mathcal{S}_h^\mathrm{bad}:=\{j \in \mathbb{Z} \colon |j|\leq N_h\} \setminus \mathcal{S}_h^\mathrm{good}$. In view of \eqref{ineq:cont_ass_uniform} and \eqref{def:goodstripes_surface}, we can estimate \[\#\mathcal{S}_h^{\rm{bad}} \leq \frac{4-2\delta}{(4-\delta)} \frac{r}{h}\,.\] 
Therefore, for $h\in (0,h_0( \OOO r \EEE )]$ small enough, we obtain
\begin{align}\label{eq:card_bad_stripes}
\#\mathcal{S}_h^{\rm{good}}=2N_h+1-\#\mathcal{S}_h^{\rm{bad}}\geq 2\left(\frac{r}{2h}-\frac{1}{2}\right) - 1   - \frac{4-2\delta}{(4-\delta)}\frac{r}{h} \geq  \frac{\delta}{4}\frac{r}{h}\,.
\end{align}
We note that \eqref{eq:card_bad_stripes} and \eqref{eq:blow_up_void_function}(i),(ii) imply that there exists $j_h\in \mathcal{S}^{\rm{good}}_h$ such that 
\begin{align}\label{eq:estimate_on_good_stripes}
\L^3\left(E_h\cap\left( S_h(j_h)\times\left(-\frac{h}{2},\frac{h}{2}\right)\right)\right)+\int_{S_h(j_h)\times\left(-\frac{h}{2},\frac{h}{2}\right)}|Y_h- Y^{\pm}|\leq\frac{4}{\delta}\sigma_r r h^2\,,
\end{align}
for all $0<r\leq r_0$ and $0<h\leq h_0$, for a modulus of continuity $\sigma_r\to 0^{\OOO +\EEE}$ as $r\to 0^{\OOO +\EEE}$.  Since $j_h\in \mathcal{S}^{\rm{good}}_h$, by the definition in \eqref{def:goodstripes_surface} and \eqref{eq:blow-up_measures} we immediately have
\begin{equation}\label{eq:energy_estimate}
\eta h^{-4}\int_{\left(S_h(j_h)\times\left(-\frac{h}{2},\frac{h}{2}\right)\right)\setminus \overline{E}_h}W(\nabla v_h)\,\mathrm{d}x+h^{-2}\mathcal{G}^{\kappa_h}_{\mathrm{surf}}\left(E_h;  S_h(j_h)\times\left(-\frac{h}{2},\frac{h}{2}\right)\right)<2-\frac{\delta}{2}\,.
\end{equation}
Introducing the notation
$$ S_{r,h,h} =   \Big( - \frac{r}{2}, \frac{r}{2} \Big) \times \Big( - \frac{h}{2}, \frac{h}{2} \Big)^2\,,$$ 
up to translation of  $E_h$ and the domain of $v_h$ by $-hj_{\OOO h\EEE} e_2$  (not relabeled), \eqref{eq:energy_estimate} is equivalent to 
\begin{equation}\label{eq:rescaled_energy_estimate}
h^{-4}\int_{S_{r,h,h}\setminus \overline{E}_h}\eta W(\nabla v_h)\,\mathrm{d}x+h^{-2}\mathcal{G}^{\kappa_h}_{\mathrm{surf}}(E_h;S_{r,h,h})< 2-\frac{\delta}{2} \,.	
\end{equation}
For $r>0$ fixed, we now apply \cite[Theorem~2.1]{KFZ:2023} to obtain 
$$((y^{\rm rod}|d_2|d_3),I) \in SBV^2_{\mathrm{isom}} \left(-\frac{r}{2},\frac{r}{2}\right) \times \mathcal{P}\left(-\frac{r}{2},\frac{r}{2}\right)$$  (cf.~the definition in \cite[(2.13)]{KFZ:2023}) such that, up to a subsequence in $h$ (not relabeled),  
\begin{align}\label{coooonv1}
&\chi_{V^{\rm rod}_{h}}\longrightarrow \chi_{V^{\rm rod}}\ \text{ in } L^1(S_{r,1,1})\,,\notag \\
\quad& y^{\rm rod}_{h} \longrightarrow \bar{y}^{\rm rod}  \text{ in  } L^1(S_{r,1,1};\R^3)  \,,
\end{align}
where $y^{\rm rod}_{h}(x) = v_h(x_1,hx_2,hx_3)$ and  $\bar{y}^{\rm rod}(x) = y^{\rm rod}(x_1)$  for $x \in S_{r,1,1}$, as well as 
\[V^{\rm rod}_{h} =\lbrace x \in S_{r,1,1} \colon (x_1,h x_2,h x_3) \in E_{h} \rbrace\,,\]
and $V^{\rm rod} = I \times (-\frac{1}{2},\frac{1}{2})^2$ (see also \cite[\OOO Equations \EEE (2.6) and (2.14)]{KFZ:2023} for the notations). By definition of  $SBV^2_{\mathrm{isom}}\left(-
\frac{r}{2},\frac{r}{2}\right)$, we particularly have that 
\begin{equation}\label{eq:R_rod_1}
R^{\rm rod}(x_1) \defas  (\partial_1 y^{\rm rod}|d_2|d_3)(x_1)\in SO(3) \text{ for } \L^1\OOO\text{-a.e. } \EEE x_1 \in \left(-\frac{r}{2},\frac{r}{2}\right)\,.
\end{equation} 
\MMM Moreover, by \cite[Theorem~2.1]{KFZ:2023} we get \EEE
\begin{align}\label{coooonv2}
\chi_{S_{r,1,1}\setminus V^{\rm rod}_{h}}  \big( \partial_1 y^{\rm rod}_{h}, \tfrac{1}{h} \partial_2 y^{\rm rod}_{h},\tfrac{1}{h} \partial_3 y^{\rm rod}_{h} \big) \rightharpoonup  \chi_{S_{r,1,1}\setminus V^{\rm rod}} \bar{R}^{\rm rod}  \ \text{ weakly in } L^2(S_{r,1,1};\R^{3\times 3})\,, 
\end{align}
where $\bar{R}^{\rm rod}  (x) = R^{\rm rod}(x_1)$ for $x \in S_{r,1,1}$  (see also  \cite[\OOO Equations \EEE (2.9) and (2.15)]{KFZ:2023} for the notations). 
	  
By the lower semicontinuity result in \cite[Lemmata 5.2 and 5.3]{KFZ:2023} \MMM and \eqref{eq:rescaled_energy_estimate}, \EEE we then get 
\begin{align}\label{ineq:limitsurface}
c_*\eta \int_{(-\frac{r}{2},\frac{r}{2})} |(R^{\rm rod})^T \partial_1 R^{\rm rod}|^2 \, {\rm d}x_1 + \mathcal{H}^0\big(\partial I \cap (-\tfrac{r}{2},\tfrac{r}{2})\big) + 2\mathcal{H}^0\big((J_{y^{\rm rod}}\cup J_{R^{\rm rod}}) \setminus \partial I\big) <2-\frac{\delta}{2}
\end{align}
for	some $c_*>0$. Here, we observe that the quadratic form on the right-hand side of \cite[Equation~(5.3)]{KFZ:2023} is coercive, which can be seen by comparison to its form in the isotropic case addressed in \cite[Remark~3.5]{Mora}. 
	
Now, \eqref{ineq:limitsurface} implies 
\begin{equation}\label{eq:no_jump_y_rod}
(J_{y^{\rm rod}}\cup J_{R^{\rm rod}}) \setminus \partial I =\emptyset \ \text{ and }\ \mathcal{H}^0\left(\partial I\cap\OOO\left(-\tfrac{r}{2},\tfrac{r}{2}\right)\EEE\right)\leq 1\,.
\end{equation}
From \eqref{eq:estimate_on_good_stripes} and a \MMM change of variables \EEE we get $\OOO \L^1(I)\EEE \le \frac{r}{3}$, provided that  \EEE  $r>0$ is small enough such that $\sigma_r <\frac{\delta}{4}\frac{1}{3}$. This \OOO and \eqref{eq:no_jump_y_rod} imply \EEE that $I \subset \left(-\frac{r}{2},-\frac{r}{6}\right)$ or $I\subset\left(\frac{r}{6},\tfrac{r}{2}\right)$. Furthermore,  for $ Y^{\rm rod} \defas  (y^{\rm rod},R^{\rm rod})$, by \eqref{eq:estimate_on_good_stripes}, and  \eqref{coooonv1}--\eqref{coooonv2},  the lower semicontinuity of \OOO the $L^1$-norm \EEE under weak convergence, and a change of variables, we have 
\begin{align}\label{ineq:limitestimatey}
\int_{\left(-\frac{r}{6},\frac{r}{6}\right)} |Y^{\rm rod}- {Y}^\pm|\,\mathrm{d}x_1 \leq \frac{4}{\delta} r \sigma_r\,,
\end{align}
where we recall \eqref{eq:jump_function}. 
\end{step}
	
\begin{step}{(5): Proof of \eqref{eq:jump-equalities}}	
First, we show \eqref{eq:jump-equalities}(a), i.e.,   ${y}^+={y}^-$. Assume by contradiction that  $|{y}^+-{y}^-|>0$, and choose $r>0$ such that  
\begin{equation}\label{eq:choice_r_1}
\OOO r <  |{y}^+-{y}^-| \ \text{ and } \ \sigma_r<\frac{\delta}{48}|{y}^+- {y}^-|\,.\EEE
\end{equation} 
Let $x^-\in (-\frac{r}{6},0)$, $x^+\in (0,\frac{r}{6})$ be such that 

$$|y^{\rm rod}(x^+)- {y}^+|\leq \frac{6}{r} \int_{(0,\frac{r}{6})} |Y^{\rm rod}-{Y}^+|\,\mathrm{d}x_1\quad \text{and}\quad  |y^{\rm rod}(x^-)- {y}^-|\leq \frac{6}{r} \int_{(-\frac{r}{6},0)} |Y^{\rm rod}-{Y}^-|\,\mathrm{d}x_1\,.
$$
Then, by \eqref{ineq:limitestimatey} \OOO and \eqref{eq:choice_r_1} \EEE  we obtain 
\begin{align}\label{eq:cont_a_1}
|y^{\rm rod}(x^+)-y^{\rm rod}(x^-)| &\geq |{y}^+- {y}^-| - |y^{\rm rod}(x^+)- y^+|- |y^{\rm rod}(x^-)- {y}^-|\notag \\&\geq  |{y}^+-{y}^-| - \frac{6}{r} \int_{(-\frac{r}{6},\frac{r}{6})}|Y^{\rm rod}- {Y}^\pm|\,\mathrm{d}x_1 \geq \frac{1}{2}| {y}^+- {y}^-| \OOO >\EEE r/2\,, 
\end{align}
where the last step follows \OOO from the choice of $r$ in \eqref{eq:choice_r_1}. \EEE On the other hand, as $J_{y^{\rm rod}} \cap (-\frac{r}{6},\frac{r}{6})  = \emptyset$, \OOO cf.~\eqref{eq:no_jump_y_rod}, \EEE by the Fundamental Theorem of Calculus and \OOO \eqref{eq:R_rod_1}, \EEE  we obtain
\begin{align}\label{eq:cont_a_2}
\begin{split}	
|y^{\rm rod}(x^+)-y^{\rm rod}(x^-)| \leq \int_{(x^-,x^+)} | \partial_1y^{\rm rod}|\,\mathrm{d}x_1 \OOO =\EEE |x^+-x^-|\leq  r/3\,.
\end{split}
\end{align}
\OOO Now, \eqref{eq:cont_a_1} and \eqref{eq:cont_a_2} \EEE contradict each other, which shows that ${y}^+={y}^-$.

Now, \OOO in a similar manner, \EEE we show \eqref{eq:jump-equalities}(b), i.e., ${R}^+= {R}^-$.	Assume by contradiction that  $|{R}^+-{R}^-|>0$ and choose $r>0$ \OOO  small enough so that \EEE
\begin{equation}\label{eq:choice_r_small_2}
r^{1/4} <  |R^+-R^-|\ \text{ and } \sigma_r< \frac{\delta}{48} |\OOO R^+- R^-\EEE|\,.
\end{equation}
Let $\OOO x\EEE^- \in (-\frac{r}{6},0)$, $\OOO x\EEE^+ \in (0,\frac{r}{6})$ be such that 
\begin{align*}
| {R}^{\rm rod}(\OOO x\EEE^+)-{R}^+|\leq \frac{6}{r} \int_{(0,\frac{r}{6})} |Y^{\rm rod}- {Y}^+|\,\mathrm{d}x_1\quad \text{and}\quad  |{R}(\OOO x\EEE^-)-{R}^-|\leq \frac{6}{r} \int_{(-\frac{r}{6},0)} |Y^{\rm rod}-{Y}^-|\,\mathrm{d}x_1\,.
\end{align*}
This, along with \eqref{ineq:limitestimatey} \OOO and the choice of $r$ in \eqref{eq:choice_r_small_2}, \EEE shows 
\begin{align}\label{eq:cont_b_1}
\begin{split}	
|R^{\rm rod}(\OOO x\EEE^+)-R^{\rm rod}(\OOO x\EEE^-)| &\geq |{R}^+-{R}^-| - |R^{\rm rod}(\OOO x\EEE^+)-{R}^+|- |R^{\rm rod}(\OOO x\EEE^-)-{R}^-|\\&\geq  |{R}^+-{R}^-| - \frac{6}{r} \int_{(-\frac{r}{6},\frac{r}{6})}|Y^{\rm rod}-{Y}^\pm|\,\mathrm{d}x_1 \geq \frac{1}{2}|{R}^+-{R}^-| \OOO >\EEE \frac{1}{2}r^{1/4}\,. 
\end{split}
\end{align}
On the other hand, using \eqref{ineq:limitsurface} we  get that $\Vert \partial_1 R^{\rm rod} \Vert_{L^2((-\frac{r}{6},\frac{r}{6}))} \le \bar{C}$ for a constant \OOO$\bar C>0$ \EEE depending on $\eta$, but independent \OOO of \EEE $r$. \OOO Using this $L^2$-bound, together with the Fundamental Theorem of Calculus along with the fact that  $J_{R^{\rm rod}} \cap (-\frac{r}{6},\frac{r}{6})  = \emptyset$, \OOO cf.~again \eqref{eq:no_jump_y_rod}, \EEE  shows 
\begin{align}\label{eq:cont_b_2}
|{R}^{\rm rod}(\OOO x\EEE^+)-{R}^{\rm rod}(\OOO x\EEE^-)| \leq \bar{C} r^{1/2}\,.
\end{align}	
For $r>0$ sufficiently \OOO small (depending on $\eta>0$), \eqref{eq:cont_b_1} and \eqref{eq:cont_b_2} \EEE contradict each other, which shows $R^+ = R^-$. This concludes the proof. 
\end{step}
\end{proof}
\EEE

\section{Proof of Theorem \ref{thm:main_kirchhoff_gamma_conv}{\rm(ii)}}\label{sec: gamma_limsup}

In this last section we construct recovery sequences for admissible limits $(y,V)\in\mathcal{A}$, see \eqref{limiting_admissible_pairs}, for which we proceed in several steps. 

\medskip

\begin{step}{(1): Preparations}	The first step is devoted to the smoothening of the void set $V$ and covering most of the jump set $ J_{(y, \nabla' y)}$ by a suitable smooth void set. We fix an arbitrary error parameter $\eta \in (0,1)$, which we will send to zero only at the end of the proof by means of a diagonal argument. We  choose a smooth set $Z_\eta \subset \R^2$ such that 
\begin{align}\label{prep1}
\L^2((V \triangle Z_\eta) \cap S) \le \eta, \quad \quad   \mathcal{H}^1(\partial Z_\eta \cap S) \le \mathcal{H}^1(\partial^* V \cap S) +  \eta, \quad \quad  \mathcal{H}^1(\partial^* V \setminus \overline{Z_\eta}) \le \eta.   
\end{align}
Indeed, we first apply \cite[\OOO Theorem 3.1, Remark 3.2(i)\EEE]{set-reference} to find \OOO a relatively open set $Z'_\eta \in \mathcal{P}(S)$ such that $\partial Z'_\eta\cap S$ is a 1-dimensional $C^1$-submanifold, \EEE with 
\[\L^2(V \triangle Z'_\eta) \le \frac{\eta}{2}\,\quad \text{and}\quad  \mathcal{H}^1(\partial^* V \triangle \partial Z'_\eta) \le \frac{\eta}{2}\,.\] Then, for \eqref{prep1}, it suffices to choose a smooth set $Z_\eta \supset Z_\eta'$ with \[\L^2(Z_\eta \setminus Z_\eta') \le \frac{\eta}{2} \, \quad \text{and} \quad\mathcal{H}^1(\partial Z_\eta \cap S) \le \mathcal{H}^1(\partial Z_\eta' \cap S)  +  \frac{\eta}{2}\,.\]
Let
\begin{equation*}
J_\eta := (J_{(y, \nabla' y)} \cup \partial^* V) \setminus \overline{Z_\eta}\,.
\end{equation*}
By a standard Besicovitch covering argument (see, e.g.,~\cite[Equations (2.3), (2.6)]{Francfort-Larsen:2003} for details), we  can find a finite number of pairwise disjoint closed rectangles $(R_i)_{i=1}^N$, so that for every $i=1,\dots,N$, $R_i\subset\subset S \setminus Z_\eta$, with length $l_i$ and height $\eta l_i$, and
\begin{equation}\label{prep2}
\mathcal{H}^1\left(J_\eta \setminus \bigcup_{i=1}^N  R_i \right) \le \eta, \quad \quad    \sum_{i=1}^N l_i \le (1+ \eta)\mathcal{H}^1(J_\eta)\,.
\end{equation}
We can also pick pairwise disjoint smooth sets $T_i \subset \subset S \setminus Z_\eta$, $i=1,\dots,N$, so that  
\begin{equation}\label{eq:T_i_prop}
T_i \supset R_i\,,\quad \L^2(T_i) \le (1+\eta)\L^2(R_i)\,, \quad  \text{and} \quad  \mathcal{H}^1(\partial T_i) \le (1+\eta)\mathcal{H}^1(\partial R_i)\,.
\end{equation} 
We define $V_\eta := Z_\eta \cup \bigcup_{i=1}^N T_i$ and $y_\eta  \in SBV^{2,2}_{\mathrm{isom}}(S;\R^3)$ by
\begin{align}\label{prep3}
y_\eta\CCC(x')\EEE \defas \begin{cases}
y\CCC(x')\EEE & \CCC{\rm{ for }}\  x'\in\EEE S \setminus V_\eta\,,\\
\CCC x'\EEE & \CCC{\rm{ for }}\  x'\in\EEE V_\eta\,.
\end{cases} 
\end{align}
We also denote by $\tilde{y}_\eta \colon \Omega\to \R^3$ the corresponding deformation indicated by the identification \eqref{eq:y_ident}.  By the fact that the jump set of  $(y_\eta,\nabla'y_\eta)$ \EEE is contained in $J_\eta \cup \partial V_\eta$, \eqref{prep2} and \eqref{eq:T_i_prop} yield
\begin{align}\label{Jeta}
\mathcal{H}^1\big(J_{(y_\eta, \nabla' y_\eta)} \cap (S \setminus\overline{V_\eta})\big) \le \eta\,.
\end{align}
Moreover, \eqref{prep1}--\eqref{eq:T_i_prop} also imply that  
\begin{align}\label{prep4}
\mathcal{H}^1(\partial V_\eta \cap S) &\le \mathcal{H}^1(\partial Z_\eta \cap S)   + \sum_{i=1}^N \mathcal{H}^1(\partial T_i) \le \mathcal{H}^1(\partial^* V \cap S)  + \eta + \sum_{i=1}^N (1+\eta) (2+2\eta)l_i \notag \\
& \le \mathcal{H}^1(\partial^* V \cap S)  + 2\mathcal{H}^1(J_\eta)+ C\eta  \le  \mathcal{H}^1(\partial^* V \cap S)  + 2\mathcal{H}(J_\eta \setminus \partial^* V)+ C\eta \notag \\ & \le \mathcal{H}^1(\partial^* V \cap S)  + 2\mathcal{H}^1\big(J_{(y, \nabla' y)} \setminus \partial^*V \big)+ C\eta\,,
\end{align}
where $C>0$ depends only on $\mathcal{H}^1(J_\eta)$ and thus only on $(y,V)$.   
Using again \eqref{prep1}--\eqref{eq:T_i_prop},  we also find
\begin{align}\label{prep5}
	\begin{split}	
		\L^2\big((V \triangle V_\eta) \cap S\big)&\leq  \L^2((V \triangle Z_\eta)\cap S)+\sum_{i=1}^N\L^2(T_i) \EEE
		\\ 
		&\le \eta +  (1+\eta)\eta \sum_{i=1}^N l_i^2\le  \eta+ C\eta\mathcal{H}^1(J_\eta)\le C\eta\,,
	\end{split}
\end{align}
where we \MMM  employed \EEE that $l_i \le (1+\eta)\mathcal{H}^1(J_\eta) \le C$ for every $i=1,\dots,N$.  

Baring this construction in mind, our goal \OOO now \EEE is to construct sets $(\tilde{V}_h)_{h>0}$  and functions $(y_h)_{h>0}$ as follows (for the sake of not overburdening the notation in the following, \EEE subscripts $h$ will indicate that the objects depend on both $h$ and $\eta$): we need to find smooth sets  $(W^{\rm void}_{h})_{h>0} \subset \mathcal{A}_{\rm{reg}}(\R^2)\EEE$ with 
\begin{align}\label{eq: toshow1}
\lim_{h \to 0}\L^2(W^{\rm void}_h) = 0\,, \quad \quad \mathcal{H}^1(\partial W^{\rm void}_h) \le C\eta\,,
\end{align}
such that also the set $V_h \defas  {\rm int}(\overline{V_\eta \cup W^{\rm void}_h})$ is smooth, and $\tilde{V}_h \defas V_h \times (-\frac{1}{2},\frac{1}{2}) \subset \Omega$ satisfies
\begin{align}\label{eq: toshow1.5} 
{\rm (i)} \ \ &  \chi_{\tilde{V}_h} \to \chi_{\tilde{V}_\eta} \ \ \text{ in } L^1(\Omega)\, \OOO \quad \text{ as }h\to 0\EEE\notag \\
{\rm (ii)} \ \ & \kappa_h  \int_{\partial \tilde{V}_h\cap \Omega}|\bm{A}|^2\,\mathrm{d}\mathcal{H}^2 \to 0 \quad \text{as $h\to 0$}\,,
\end{align} 
where $\tilde{V}_\eta \defas V_\eta \times (-\frac{1}{2},\frac{1}{2})$.   Moreover, we need to find a sequence $(y_h)_{h>0}$ with $y_h \in W^{1,2}(\Omega \setminus \tilde{V}_h;\R^3)$  such that 
\begin{align}\label{eq: toshow2}
\begin{split}	
{\rm (i)} \ \ &  y_h \to \CCC\tilde y_\eta\EEE   \ \ \text{strongly in }  L^1(\Omega;\R^3) \OOO \ \text{ as }h\to 0\EEE\,,\\
{\rm (ii)} \ \ &  \nabla_{h}y_{h}\longrightarrow  \big(\nabla'  \CCC\tilde y_\eta\EEE,  \partial_1 \CCC\tilde y_\eta\EEE \wedge \partial_2  \CCC\tilde y_\eta\EEE\big)  \ \text{ strongly in } L^2(\Omega;\R^{3\times 3}) \OOO \ \text{ as }h\to 0\EEE\,,\\
{\rm (iii)} \ \ &  \limsup_{h \to 0} \left(h^{-2}\int_{\Omega\setminus \tilde{V}_h} W(\nabla_hy_h(x))\,\mathrm{d}x\right) \le \frac{1}{24}\int_{S \setminus V_\eta}\Q_2(\mathrm{II}_{y_\eta}(x'))\,\mathrm{d}x' + \eta\,,\\
{\rm (iv)} \ \ &  \Vert y_h \Vert_{L^\infty(\Omega)} \le M\,. 
\end{split}
\end{align}
Then,  recalling \eqref{initial_energy}, \eqref{rescaled_energy}, \eqref{eq: newenergy}, by \eqref{eq: toshow2}(iii), \eqref{eq: toshow1},\eqref{eq: toshow1.5},  and the fact that $\partial V_h\setminus \partial V_\eta\subset \partial W^{\rm{void}}_h$,   we find  
\[\limsup_{h\to 0} \E^{h}(y_{h},\tilde{V}_{h}) \le   \Big(\frac{1}{24}\int_{S \setminus {V}_\eta}\Q_2(\mathrm{II}_{y_\eta}(x'))\,\mathrm{d}x'+\H^1(\partial V_\eta\cap S)\Big)+C\eta\,\]
where we used that \eqref{eq: toshow1.5}(ii) is equivalent to 
\[h^{-1}\kappa_h  \int_{\partial (T_h(\tilde{V}_h))\cap \Omega_h}|\bm{A}|^2\,\mathrm{d}\mathcal{H}^2 \to 0 \quad \text{as $h\to 0$}\OOO\,,\EEE\]
since $\partial \tilde V_h\cap \Omega$ is cylindrical over $\partial V_h\cap S$.   Recalling \eqref{limiting_two_dimensional_energy}, \EEE we also observe that 
\[\limsup_{\eta \to 0}\Big(\frac{1}{24}\int_{S \setminus  {V}_\eta}\Q_2(\mathrm{II}_{y_\eta}(x'))\,\mathrm{d}x'+\H^1(\partial V_\eta\cap S)\Big)\EEE \le \mathcal{E}^{0}(y,V)\quad \text{as }\eta \to 0\,, \] 
by \eqref{prep3} and \eqref{prep4},\eqref{prep5}. As $\L^2(\lbrace y_\eta \neq y \rbrace) \le C\eta$ by \eqref{prep3} and \eqref{prep5}, using a diagonal argument in the theory of $\Gamma$-convergence, we obtain the desired recovery sequence. Here, we also use \eqref{eq: toshow1.5}(i) and \EEE \eqref{eq: toshow2}(i),(ii) to see that the convergence $\tau$ defined in Definition \ref{eq:tau_convergence} is satisfied, and we use \eqref{eq: toshow2}(iv) to guarantee that $y_h$ (extended by $ T_h(\mathrm{id})$ inside $\tilde{V}_h$) is admissible, cf.~\eqref{admissible_configurations_h_level}.

Summarizing, in the sequel it suffices to construct the sets $(W^{\rm void}_h)_{h>0}$ and the functions $(y_h)_{h>0}$ such that \eqref{eq: toshow1}--\eqref{eq: toshow2} hold.
\end{step}

\begin{step}{(2): Meshes and auxiliary regularization of the jump set}
In view of \eqref{Jeta}, the jump set 
$J_{(y_\eta, \nabla' y_\eta)}$  inside $S \setminus\overline{V_\eta}$ has small  $\H^1$-measure. To be in the position to repeat the arguments from the elastic case \cite[Section 6]{friesecke2002theorem}, it would be necessary to replace $y_\eta$ by a Sobolev function on $S \setminus\overline{V_\eta}$. A first idea could be to apply Corollary \ref{replace-cor} (or the corresponding Sobolev replacement in Theorem~\ref{th: kornSBDsmall}) on $S \setminus\overline{V_\eta}$. Yet, this approximation \MMM is \EEE  not compatible with the nonlinear rotationally invariant elastic energy, and would provide 
  inadequate   estimates.  Better approximations can be obtained by applying Corollary \ref{replace-cor} on meshes of scales smaller or equal to $h$. To this end, for $n \in \N_0$ and some $\Lambda\ge 1$, we partition $\R^2$ up to a set of negligible measure into the squares 
\begin{align}\label{mesh}
\mathcal{Q}_h^n \defas \big\{ Q_h^n(p) := p + \Lambda 2^{-n}h (-\tfrac{1}{2},\tfrac{1}{2})^2, \ p \in  \Lambda 2^{-n}h \Z^2\big\}\,.
\end{align}
The parameter $\Lambda$ will be chosen eventually in \eqref{Lambda-ref} (depending only on $\eta$) and  plays a role in an extension procedure, which  will become clear in  Steps 7--8 below. 
We write  
\begin{equation}\label{def:U_eta}
U_\eta :=  S \setminus \overline{V_\eta}  
\end{equation}
for notational convenience. Our strategy consists in defining a Whitney-type covering related to $U_\eta$ such that the jump set is covered by squares with small \OOO area\EEE, see Proposition \ref{prop:whit}  below for the precise statement and particularly \eqref{Wit2-5}.  For technical reasons in this construction, it is convenient to regularize the jump set of $y_\eta$. For notational convenience, we write
\begin{equation}\label{def:F_eta_b_eta}
(F_\eta,b_\eta) := (\nabla' y_\eta, \partial_1 y_\eta \wedge \partial_2 y_\eta)\,.  
\end{equation}  
By the density result   \cite[Theorem 3.1]{Cortesani-Toader:1999} we can find functions $(z_h)_{h>0} \subset SBV^2(U_\eta;\R^3)$  and $(F_h,b_h)_{h>0} \subset SBV^2(U_\eta;\R^{3\times 3})$ such that $J_{z_h}$ and $J_{(F_h,b_h)}$ consist of a finite number of segments, and 
\begin{align}\label{h3}
\begin{split}	
{\rm (i)} & \ \ \Vert z_h -y_\eta \Vert_{L^1(U_\eta)} + \Vert \nabla'  z_h-\nabla' y_\eta  \Vert_{L^2(U_\eta)}  \le h^2\,,\\
{\rm (ii)} & \ \ \Vert (F_h,b_h)-(F_\eta,b_\eta)\Vert_{L^2(U_\eta)} + \Vert   \nabla' (F_h,b_h)-\nabla'(F_\eta,b_\eta)  \Vert_{L^2(U_\eta)}  \le h^2\,,\\
{\rm (iii)} & \ \  \mathcal{H}^1(\Gamma_h) \le 2\eta, \quad \text{for $\Gamma_h \defas J_{z_h}\cup J_{(F_h,b_h)}$}\,,\\
{\rm (iv)} & \ \ \Vert z_h \Vert_{L^\infty(U_\eta)} \le  \Vert y_\eta \Vert_{L^\infty(U_\eta)}\,, \quad \Vert b_h \Vert_{L^\infty(U_\eta)} \le \Vert b_\eta \Vert_{L^\infty(U_\eta)}\,,
\end{split}
\end{align}
where for \OOO\eqref{h3}\EEE(iii) we used \eqref{Jeta} \OOO and \eqref{def:U_eta}\EEE. Note that the jump sets of $z_h$ and $(F_h,b_h)$ depend on $h$ and therefore this regularization does not appear to be helpful \OOO yet,  \EEE as it does not allow for uniform estimates. The only reason for this approximation is that it guarantees that the Whitney-type covering in Proposition~\ref{prop:whit} terminates at some finite scale $\mathcal{Q}_h^{K_h} $ for $K_h \in \N$ depending on $h$, see the discussion below \eqref{theclaim}. 
\end{step}

\begin{step}{(3):  Construction of  \MMM $W^{\rm void}_h$\EEE }
 Recalling \eqref{mesh}, we denote generic squares in $\mathcal{Q}_h^n$, $n \in \N_0$, by $q$.  By $\ell(q)$ we indicate the sidelength of the square, i.e., $\ell(q) := \Lambda 2^{-n}h$ for some $n \in \N_0$. Moreover, by  $q'$ and $q''$ we denote squares with the same center as $q$ and
\begin{align}\label{ell}
	\ell(q') = \frac{3}{2} \ell (q), \quad \quad \ell(q'') =   21   \ell (q)\,.
\end{align}
(The value $21$ is chosen for definiteness only and could also be any odd number sufficiently large.)
Since we consider squares of size $\sim h$, the jump set $\Gamma_h$, see \eqref{h3}(iii), is not necessarily small compared to $\ell(q)$ in each square $q$, which might  prevent the application of Corollary~\ref{replace-cor}. To this end, we define 
\begin{align}\label{5u}
\mathcal{Q}^U_h := \lbrace q \in  \mathcal{Q}^{0}_h\colon  q \subset U_\eta\rbrace\,,
\end{align} 
and given some universal $\theta \in (0,\frac{1}{16})$ small to be specified later   (see below \OOO\eqref{eq:isoper_omega_1_2}\EEE), we introduce the collection of \emph{bad squares} defined by  
\begin{equation}\label{eq: badd}
\mathcal{Q}_h^{\rm bad} :=  \big\{ q \in  \mathcal{Q}^U_h\colon    \mathcal{H}^1(\Gamma_h\cap q') \ge  \theta \Lambda h \rbrace\,.
\end{equation} 
We define the sets
\begin{equation}\label{Ubadest0}
U_h := \bigcup_{q \in \mathcal{Q}_h^{U}} \overline{q}\,, \quad \quad U_h^{\rm bad} := \bigcup_{q \in \mathcal{Q}_h^{\rm bad}} \overline{q''}\,. 
\end{equation}
From the definition of $\mathcal{Q}_h^{\rm bad}$  in \eqref{eq: badd} \EEE and \eqref{h3}(iii) as well as \eqref{ell}, we get 
\begin{align}\label{Ubadest}
\begin{split}
\L^2(U_h^{\rm bad})  & \le C(\Lambda h)^2\#\Q_h^{\rm{bad}}\EEE\le C\Lambda h \mathcal{H}^1(\Gamma_\eta) \le C\Lambda h\,, \\
 \mathcal{H}^1(\partial U_h^{\rm bad}) &  \le C(\Lambda h)\#\Q_h^{\rm{bad}}\EEE\le C\mathcal{H}^1(\Gamma_\eta) \le C\eta  
 \end{split}
\end{align}
for a constant $C>0$ depending only on $\theta$, where we used that each $x \in \R^2$ is only contained in a  (universally)  bounded number of squares $q'$, for $q \in \mathcal{Q}_h^0$. 

Then, \OOO recalling the notation in \eqref{neigh-def}, \EEE  we can choose \OOO a smooth \EEE $W^{\rm void}_h \supset \OOO  (U_h^{\rm bad})_{\Lambda h} \EEE$, satisfying   \eqref{eq: toshow1}. More precisely, since $U_h^{\rm bad}$ consists of at most $C\eta (\Lambda h)^{-1}$-many squares of sidelength $\Lambda h$, this can be done in such a way that 
\[ \|\bm{A}\|_{L^\infty(\partial W^{\rm void}_h)} \le C(\Lambda h)^{-1}\,,\EEE\]
cf.\ also \cite[Lemma 3.5]{KFZ:2022} for a similar construction. \EEE Therefore, by a careful choice of the sets $W^{\rm void}_h$  so that also $V_h \defas  {\rm int}(\overline{V_\eta \cup W^{\rm void}_h})$ is smooth,  in view of \eqref{eq: toshow1},  the set \MMM $\tilde{V}_h = V_h \times (-\frac{1}{2},\frac{1}{2})$ \EEE satisfies 
$$\int_{\partial \tilde{V}_h\cap \Omega}|\bm{A}|^2\,\mathrm{d}\mathcal{H}^2 \le \int_{(\partial V_\eta \cap S)\times (-\frac{1}{2},\frac{1}{2})}|\bm{A}|^2\,\mathrm{d}\mathcal{H}^2 +  C\mathcal{H}^1(\partial W^{\rm void}_h) h^{-2} \le C_\eta + C\eta h^{-2},$$
where $C_\eta>0$ depends on $V_\eta$ and thus on  $\eta$. By \eqref{rate_1_gamma_h} 
this shows \eqref{eq: toshow1.5}(ii). Clearly, \eqref{eq: toshow1} implies \eqref{eq: toshow1.5}(i).

For convenience, we define  $U_h^{\rm good} := U_h \setminus U_h^{\rm bad}$. For later purposes, we extend the set $U_h^{\rm good} $ by adding two extra layers around it. More precisely, we define
\begin{equation}\label{omega-ext0}
U_h^{\rm ext} :=   U_h^{\rm good}  \cup  \bigcup_{q \in \mathcal{Q}_h^{\rm ext}} \overline{q}\,,
\end{equation}
where 
\begin{equation}\label{cubes-ext0}
\Q_h^{\rm{ext}}:=\Big\lbrace q \in \mathcal{Q}_h^0 \colon q \notin U_h^{\rm good}\,,\ {\rm dist}_\infty(q, U_h^{\rm good}) \in \lbrace 0 , \Lambda h \rbrace  \Big\rbrace\,.
\end{equation} 
\MMM Using \EEE the definition of $U_h^{\rm bad}$ in \eqref{Ubadest0}, and the fact that \OOO $W^{\rm void}_h \supset (U_h^{\rm bad})_{\Lambda h}$, \EEE it is elementary to check that   
\begin{equation}\label{extensnion}
U_h^{\rm ext}  \cup  W^{\rm void}_h \supset   (U_\eta)_{(2-\sqrt{2})\Lambda h}\,.
\end{equation}
\end{step}

\begin{step}{(4): Construction of a Whitney-type covering}
We now construct a covering of $U_h^{\rm good}$. The main point is that Corollary \ref{replace-cor}  (for $d=2$) \EEE is \OOO then \EEE applicable in all squares and that the entire jump set   $\Gamma_h$, cf. \eqref{h3}(iii), \EEE can be covered by a set with small \OOO area\EEE, see \eqref{Wit2}--\eqref{Wit2-5}. In the covering we also ensure that the squares at the boundary of $U_h^{\rm good}$ are in $\mathcal{Q}_h^0$, which later will allow us to easily extend the covering to the set $U_h^{\rm  ext}$ defined in \eqref{omega-ext0}. For the next statement we recall the notations $q,q',q''$ introduced before \eqref{ell}, and refer to Figure~\ref{fig:cube_covering} for an illustration of the covering the next Proposition describes.

\begin{proposition}\label{prop:whit}
There exists a covering of Whitney-type $\mathcal{W}_h := ( q_i )_{i \in \OOO \I\EEE} \subset \bigcup^{K_h}_{n=0} \mathcal{Q}_h^n $ for some $K_h \in \N$ such that the squares $(q_i )_{i\in \OOO \I\EEE}$ are pairwise disjoint, and satisfy
\begin{align}\label{Wit1}
\begin{split}	
{\rm (i)} & \ \ \bigcup_{i \in \OOO \I\EEE} \overline{q_i} = \overline{ U_h^{\rm good}}\,,\\
{\rm (ii)} & \ \ q_i' \cap q_j' \neq \emptyset \quad \implies \quad   \frac{1}{2}\ell(q_i) \le \ell(q_j) \le 2 \ell (q_i)\,,\\[4pt]
{\rm (iii)} & \ \  \# \lbrace  j \in \OOO \I\EEE  \colon \,   q_i' \cap q_j' \neq \emptyset \rbrace \le 12  \quad \text{for all } i \in \OOO \I\EEE\,.
\end{split}
\end{align}
 Moreover,   defining
\begin{align}\label{the four sets}
\begin{split}	
\mathcal{W}^{\rm bdy}_h  & \defas  \lbrace q_i \colon   \partial q_i \cap \partial U_h^{\rm good} \neq \emptyset \rbrace\,,\\[3pt]
\mathcal{W}^{\rm jump}_h  & \defas  \lbrace q_i \notin   \mathcal{W}^{\rm bdy}_h \colon    \theta^2 \ell(q_i) \le   \mathcal{H}^1(\Gamma_h\cap q_i')  \rbrace\,,\\[3pt]
\mathcal{W}^{\rm empt}_h  & \defas  \lbrace q_i  \notin \mathcal{W}^{\rm bdy}_h  \colon   \mathcal{H}^1(\Gamma_h\cap q_i') = 0  \rbrace\,,\\[3pt]
\mathcal{W}^{\rm neigh}_h  & \defas  \mathcal{W}_h \setminus \big(   \mathcal{W}^{\rm bdy}_h \cup  \mathcal{W}^{\rm jump}_h \cup \mathcal{W}^{\rm empt}_h\big)\,,
\end{split}		
\end{align}
it holds that 
\begin{align}\label{Wit2}
\begin{split}	
{\rm (i)} & \ \  q_i \in \mathcal{Q}_h^0 \quad \text{for all } q_i \in \mathcal{W}^{\rm bdy}_h\,,\\[3pt]
{\rm (ii)} & \ \  \mathcal{H}^1(\Gamma_h\cap q_i') \le \theta \ell(q_i) \quad \text{for all } q_i \in \mathcal{W}^{\rm bdy}_h \cup \mathcal{W}^{\rm jump}_h \cup   \mathcal{W}^{\rm neigh}_h\,,\\[3pt]
{\rm (iii)} & \ \    q_i''  \subset  W_h^{\rm cov}\quad \text{for all } q_i \in   \mathcal{W}^{\rm neigh}_h,
\end{split}
\end{align}
where 
\begin{equation}\label{eq:W_h_cov}
W_h^{\rm cov} := \bigcup_{q \in   \mathcal{W}^{\rm bdy}_h \cup \mathcal{W}^{\rm jump}_h} \overline{q''}\,.
\end{equation} 
In particular, the set $ W_h^{\rm cov}$ satisfies
\begin{equation}\label{Wit2-5}
\L^2( W_h^{\rm cov}) \le C h\,,
\end{equation}
for $C>0$ only depending on $\Lambda$, $\theta$, \MMM and $\eta$. \EEE
\end{proposition}

\begin{figure}[htp]
\includegraphics[width=1\linewidth]{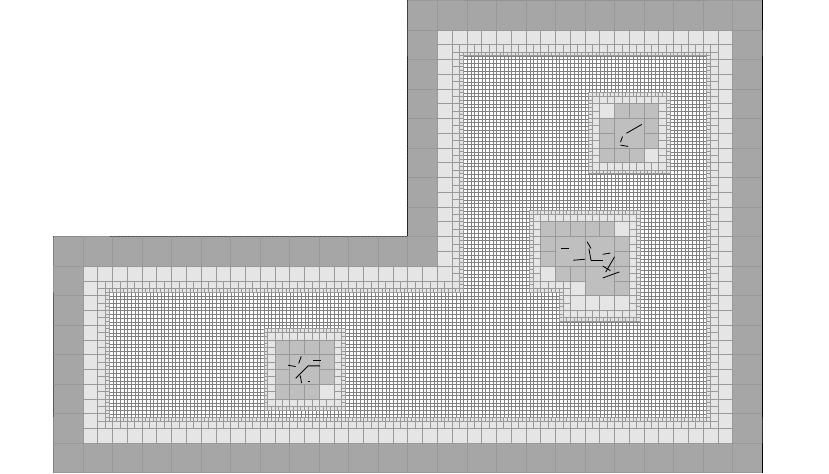}
\caption{An illustration of the dyadic construction giving the covering $\mathcal{W}_h$.    The  {jump set} $\Gamma_h$ is depicted by the black segments.\EEE}
\label{fig:cube_covering}
\end{figure}

\begin{proof}[Proof of Proposition \ref{prop:whit}]
First, all $q \in \mathcal{Q}_h^{0}$ with $q \subset U_h^{\rm good}$ and $\partial q \cap \partial U_h^{\rm good} \neq \emptyset$ are collected in $ \mathcal{W}^{\rm bdy}_h$. Note that  \eqref{Wit2}(i) holds and by \OOO the definition \EEE of $\mathcal{Q}_h^{\rm bad}$ in \eqref{eq: badd}, property \eqref{Wit2}(ii) is satisfied for each $ q\in  \mathcal{W}^{\rm bdy}_h$. 
	
\smallskip

\emph{Step 1: Induction.} The rest of the covering is constructed inductively. We first construct the collections   $ \mathcal{W}^{\rm jump}_h$ and  $\mathcal{W}^{\rm neigh}_h$.  All $q \in \mathcal{Q}_h^{0} $, $q \subset U_h^{\rm good}$, with   
\begin{equation}\label{def:Jump_0}
\mathcal{H}^1(\Gamma_h\cap q') \ge \theta^2 \ell(q) = \theta^2 \Lambda h\,,
\end{equation}
or $q \in \mathcal{W}^{\rm bdy}_h$ are collected in  $\mathcal{Y}_0^{\rm jump}$, and we define  $Y_0 :=  \bigcup_{q \in \mathcal{Y}_0^{\rm jump}} \overline{q} $. (Note that for later purposes it is convenient to also add the squares at the boundary to this set.) For definiteness, we also define $\mathcal{Y}_0^{\rm neigh} = \emptyset$.
	
Suppose that for some $k \in \N$ the collections  $\mathcal{Y}_j^{\rm jump},  \mathcal{Y}_j^{\rm neigh} \subset  \mathcal{Q}_h^{j}$, $0\le j \le k$, and 
\begin{equation}\label{def:Y_k}
Y_k :=   \bigcup_{j=0}^k\,\, \bigcup_{q \in \mathcal{Y}_{j}^{\rm jump} \cup \mathcal{Y}_{j}^{\rm neigh}} \overline{q}
\EEE
\end{equation}
have already been constructed such that 
\begin{align}\label{eq: propteries in each step}
\begin{split}	
{\rm (i)} & \ \  q_1' \cap q_2' \neq \emptyset \ \text{ for } q_1,q_2  \in \bigcup_{j=0}^k\mathcal{Y}_j^{\rm jump} \cup  \mathcal{Y}_j^{\rm neigh} \quad \implies \quad   \frac{1}{2}\ell(q_2) \le \ell(q_1) \le 2 \ell (q_2)\,,\\[3pt]
{\rm (ii)} & \ \   \theta^2 \ell(q) \le   \mathcal{H}^1(\Gamma_h\cap q')    \quad \text{for all $q \in \bigcup_{j=0}^k \mathcal{Y}_j^{\rm jump} \setminus \mathcal{W}^{\rm bdy}_h$}\,,\\[3pt]
{\rm (iii)} & \ \    \mathcal{H}^1(\Gamma_h\cap q')  \le  \theta \ell(q)  \quad \text{for all $q \in  \bigcup_{j=0}^k  \mathcal{Y}_j^{\rm jump} \cup \mathcal{Y}_j^{\rm neigh}$ and $q \in \mathcal{Y}_k^{\rm rest}$}\,,\\[3pt]
{\rm (iv)} & \ \  \text{for each $q \in \mathcal{Y}_k^{\mathrm{neigh}}\,\  \exists j\in \{0, \dots, k-1\}$, $\tilde q \in \mathcal{Y}^{\rm{jump}}_{j}$, so that $\OOO{\rm dist}_{\infty}\EEE(q,\tilde{q}) \le \OOO\Lambda h\EEE\sum^{k-1}_{l=j} 2^{-l}$}
\,,\\[3pt]
{\rm (v)} & \ \  \text{for all } q \in  \bigcup_{j=0}^k(\mathcal{Y}_j^{\rm jump} \cup  \mathcal{Y}_j^{\rm neigh})\EEE  \setminus \mathcal{W}^{\rm bdy}_h\colon \partial q \cap \partial Y_k \neq \emptyset \implies q \in \mathcal{Q}_h^k\,,
\end{split}
\end{align}
where we have set 
\begin{equation*}
\mathcal{Y}_k^{\rm rest} \defas \lbrace q \in \mathcal{Q}_h^k\colon q \subset U_h^{\rm good} \setminus Y_k\rbrace\,. \EEE
\end{equation*}
Clearly, by construction all the above properties are satisfied for $k=0$, see \eqref{eq: badd}, \eqref{def:Jump_0}, \MMM and \eqref{def:Y_k}. \EEE 
	
Before we proceed with the induction step, let us briefly explain the relevance of these properties. First, \eqref{eq: propteries in each step}(i) will be needed  for  \eqref{Wit1}(ii),(iii),  and \eqref{eq: propteries in each step}(ii),(iii)  are essential for \eqref{Wit2}(ii) and for the definition of $\mathcal{W}^{\rm jump}_h$, respectively. Next, \eqref{eq: propteries in each step}(iv) will lead to  \eqref{Wit2}(iii), and finally \eqref{eq: propteries in each step}(v) is needed to obtain  \eqref{eq: propteries in each step}(i) in the next iteration step.
	
We now come to the step $k+1$.  We define $\mathcal{Y}_{k+1}^{\rm jump}$ and  $\mathcal{Y}_{k+1}^{\rm neigh}$  as follows: recalling \eqref{def:Y_k}, \EEE we let 
\begin{equation}\label{Y_k+1_neigh}
 \mathcal{Y}_{k+1}^{\rm neigh}:=\{q \in \mathcal{Q}_h^{k+1}\colon q \subset U_h^{\rm good} \setminus Y_{k}\,,\quad \partial q \cap \partial Y_k \neq \emptyset\}\,.\EEE
\end{equation} 
Then, we let  
\begin{equation}\label{Y_k+1_jump}
 \mathcal{Y}_{k+1}^{\rm jump}:=\{q \in \mathcal{Q}_h^{k+1} \setminus \mathcal{Y}_{k+1}^{\rm neigh}\colon q \subset U_h^{\rm good} \setminus  Y_{k}\,, \quad \mathcal{H}^1(\Gamma_h\cap q')  \ge \theta^2 \ell(q)\}\,. \EEE
\end{equation}
We now confirm the properties \eqref{eq: propteries in each step} for the step $k+1$. First,  \eqref{eq: propteries in each step}(i),(v) for step $k$ guarantee \eqref{eq: propteries in each step}(i) for step $k+1$. Moreover, the construction   in \eqref{Y_k+1_neigh} and \eqref{Y_k+1_jump} \EEE  directly shows \eqref{eq: propteries in each step}(ii),(v) in step $k+1$.
	
Next, we address \eqref{eq: propteries in each step}(iii). To this end, fix $q \in \mathcal{Y}_{k+1}^{\rm jump} \cup  \mathcal{Y}_{k+1}^{\rm neigh} \cup \mathcal{Y}_{k+1}^{\rm rest}$, and choose the unique square $q_* \in \mathcal{Q}_h^k$ with $q \subset q_*$. Note that $\mathcal{H}^1(\Gamma_h\cap q_*') < \theta^2 \ell({q}_*)$, as otherwise $q_*$ would have been added to $\mathcal{Y}_{k}^{\rm jump}$ in the previous iteration step. This shows  
$$\mathcal{H}^1(\Gamma_h\cap q') \le  \mathcal{H}^1(\Gamma_h\cap q_*') < \theta^2 \ell({q}_*) = 2 \theta^2 \ell(q) \le \theta \ell(q)\,,$$
where we used that $\OOO0<\EEE\theta \le 1/2$, recalling the choice before \eqref{eq: badd}.

It remains to show \eqref{eq: propteries in each step}(iv) in step $k+1$. 
For $q\in \mathcal{Y}_{k+1}^{\rm neigh}$, in view of the definitions \eqref{Y_k+1_neigh} and \eqref{def:Y_k}, property \eqref{eq: propteries in each step}(v) in step $k$ yields that  there exists $ \hat{q}\in \mathcal{Y}_{k}^{\rm jump} \cup \mathcal{Y}_{k}^{\rm neigh}$ such that ${\rm dist}_{\OOO\infty\EEE}(q,\hat{q}) = 0$. If $\hat{q}\in \mathcal{Y}_{k}^{\rm jump}$, the statement follows for $\tilde q = \hat{q}$. Otherwise, if $\hat{q}\in \mathcal{Y}_{k}^{\rm neigh}$, by \eqref{eq: propteries in each step}(iv) in step $k$, we find $j\in \{0, \dots, k-1\}$  and $\tilde q \in \mathcal{Y}^{\rm{jump}}_{j}$ such that  ${\rm dist}_{\OOO\infty\EEE}(\hat{q},\tilde{q}) \le \OOO\Lambda h\EEE\sum^{k-1}_{l=j} 2^{-l}$. Since ${\rm dist}_{\OOO\infty\EEE}(q,\tilde{q}) \le {\rm dist}_{\OOO\infty\EEE}(\hat{q},\tilde{q}) + \OOO\Lambda h \EEE2^{-k}$, the statement also follows in this case. \EEE
	
\medskip
	
\emph{Step 2: Definition of the covering.} We now show that we can terminate the iteration at some step $K_h$. To this end, we claim that there exists $K_h \in \N$ such that for all $q \in \mathcal{Q}_h^{K_h}$, $q \subset U_h^{\rm good} \setminus Y_{K_h}$, we have 
\begin{align}\label{theclaim}
\mathcal{H}^1(\Gamma_h\cap q') =0\,.
\end{align}
In fact, choose $K_h$ large enough such that $\Lambda2^{-K_h}h$ is smaller than each of the length of the finite number of segments forming $\Gamma_h$,   cf.\ \eqref{h3}(iii). (This is the only point where we use the regularity   and polygonal structure of the  jump set.) Suppose by contradiction that there exists $q \in \mathcal{Q}_h^{K_h}$ \MMM such that \EEE $q \subset U_h^{\rm good} \setminus Y_{K_h}$ \MMM and  \EEE  $\mathcal{H}^1(\Gamma_h\cap q') \neq 0$. Choose $q_* \in \mathcal{Q}_h^{K_h-1}$ with $q_* \supset q$, and observe that $q_* \in \mathcal{Y}_{K_h-1}^{\rm rest}$. As $\Gamma_h$ consists of line segments whose length exceed $\Lambda 2^{-K_h}h$, it is elementary to verify that 
\[\mathcal{H}^1(\Gamma_h\cap q_*') \ge \frac{1}{4} \ell(q) = \frac{1}{8} \ell(q_*)\,,\] which contradicts \eqref{eq: propteries in each step}(iii) as $0< \theta < \frac{1}{16}$. 
	
\medskip
	
We now come to the definition of $\mathcal{W}_h$, cf.\ \eqref{the four sets}. First, $\mathcal{W}^{\rm bdy}_h$ has already been defined, and  we let  
\begin{equation}\label{def_W_h_jump}
 \mathcal{W}^{\rm jump}_h := \bigcup_{j=0}^{K_h} \mathcal{Y}_{j}^{\rm jump} \setminus \mathcal{W}^{\rm bdy}_h\,. \EEE
\end{equation}
Moreover, we introduce the auxiliary collection 
\begin{equation*}
\mathcal{W}^{\rm neigh, aux}_h := \bigcup_{j=0}^{K_h} \mathcal{Y}_{j}^{\rm neigh}\,.\EEE
\end{equation*} 
In view of  \eqref{eq: propteries in each step}(i),(v),  we can cover of $U^{\rm good}_h \setminus Y_{K_h}$ with cubes in $\mathcal{Q}_h^{K_h}$, denoted by $\mathcal{W}^{\rm empt, aux}_h$ such that $\mathcal{W}^{\rm bdy}_h\cup \mathcal{W}^{\rm jump}_h \cup \mathcal{W}^{\rm neigh, aux}_h \cup \mathcal{W}^{\rm empt, aux}_h$ satisfy \eqref{Wit1}, where we use that \eqref{Wit1}(iii) is a simple consequence of \eqref{Wit1}(ii). Eventually, we define 
\begin{equation}\label{theclaim2}
\mathcal{W}^{\rm empt}_h := \mathcal{W}_h^{\rm empt, aux} \cup \big\{ q \in \mathcal{W}^{\rm neigh, aux}_h \colon \mathcal{H}^1(\Gamma_h\cap q') = 0\big\}, \quad 
\mathcal{W}^{\rm neigh}_h  :=    \mathcal{W}_h^{\rm neigh, aux}\setminus \mathcal{W}^{\rm empt}_h.
\end{equation}
We note that the covering $\mathcal{W}_h$ consisting of the four  families  in \eqref{the four sets} still satisfies \eqref{Wit1}.  The collections satisfy the respective properties stated in  \eqref{the four sets} by \eqref{eq: propteries in each step}(ii) and \eqref{theclaim}--\eqref{theclaim2}.
	
It remains to show \eqref{Wit2}--\eqref{Wit2-5}. Property \eqref{Wit2}(i) holds by construction and  \eqref{Wit2}(ii)   follows from \eqref{eq: propteries in each step}(iii). Next,  \eqref{Wit2}(iii) follows by an elementary computation  using \eqref{def:Y_k},  \eqref{def_W_h_jump}--\eqref{theclaim2}, \EEE property \eqref{eq: propteries in each step}(iv), and the fact that $\ell(q'') =  21  \ell (q)$.    Finally, using the property of $\mathcal{W}^{\rm jump}_h$ in   \eqref{the four sets}, as well as  \eqref{eq: badd}--\eqref{Ubadest}
we compute
\begin{align*}
\L^2(W_h^{\rm cov}) &= \sum_{q \in  \mathcal{W}^{\rm bdy}_h \cup  \mathcal{W}^{\rm jump}_h} \L^2(q'')\le Ch  \sum_{q \in  \mathcal{W}^{\rm bdy}_h \cup  \mathcal{W}^{\rm jump}_h} \ell(q)  \\& 
\le    C\theta^{-2} h  \sum_{q \in  \mathcal{W}_h^{\rm  jump }} \mathcal{H}^1(\Gamma_h\cap q') + Ch\mathcal{H}^1(\partial U_h^{\rm good}) \le C\theta^{-2} h \mathcal{H}^1(\Gamma_h)  + Ch\le Ch\,,   
\end{align*}
for a constant  $C>0$  depending on $\Lambda$, $\theta$ \MMM and $\eta$, \EEE where in the penultimate step we have employed \eqref{Wit1}(iii),   the definition of $U_h^{\rm{good}}$ before \eqref{omega-ext0}, and the fact that (by the regularity of $U_\eta$)
\[\H^{1}(\partial U_h)\leq C\H^1(\partial U_\eta)\leq C_\eta\,.\] 
The last step follows from \eqref{h3}(iii). This concludes the proof of the proposition.
\end{proof}
\end{step}

\begin{step}{(5): Auxiliary estimates on the Whitney-type covering}
Before we can come to the definition of the deformations  $(y_h)_{h>0}$, we need some preliminary estimates on the squares of the Whitney-type covering $\mathcal{W}_h$ defined in Proposition \ref{prop:whit} which \OOO allow \EEE us to control the behavior on adjacent squares.   Recalling the notation in \eqref{def:U_eta}--\eqref{def:F_eta_b_eta}, for notational convenience, we define 
\begin{align}\label{eq: for SO2XXX}
a_h :=  h^{-2} \big(|\nabla'z_h-\nabla' y_\eta |+|(F_h,b_h)-(F_\eta,b_\eta)|\big) \in L^2(U_\eta)\,, 
\end{align}
which is bounded in $ L^2(U_\eta) $, uniformly in $h>0$, \EEE by \eqref{h3}.  For every $i\in \OOO\I\EEE$ we apply  Corollary~\ref{replace-cor} on $q_i'$ for the function $(F_h,b_h)$ to find \OOO a set \EEE of finite perimeter $\omega^1_i \subset q_i'$ and \OOO a matrix field \EEE $(F_i,b_i)  \in \R^{3 \times 3}$ such that 
\begin{align}\label{eq: forvh0}
\begin{split}	
{\rm (i)} & \ \ \mathcal{H}^1(\partial^* \omega^1_i) \le C \mathcal{H}^1(\Gamma_h \cap q_i')\,,\\
{\rm (ii)} & \ \  \Vert  (F_h,b_h)   - (F_i,b_i)      \Vert_{L^2(q_i' \setminus \omega^1_i)}  \le C\ell(q_i) \Vert \nabla' (F_h,b_h) \Vert_{L^2(q_i')}\,,  
\end{split}
\end{align}
where we recall the definition of $\Gamma_h$ in \eqref{h3}(iii).
(Clearly, the  objects $\omega_i^1, (F_i,b_i)$ \EEE also depend on $h$ which we do not include in the notation for simplicity.)    In view of \eqref{eq: for SO2XXX}, \eqref{eq: forvh0}, and the fact that $\nabla' y_\eta = F_\eta$, see  \eqref{def:F_eta_b_eta}, \EEE we also get 
\begin{equation}\label{eq: forvh0-}
	\Vert  \nabla' z_h   - F_i      \Vert_{L^2(q_i' \setminus \omega^1_i)}   \le C\ell(q_i) \Vert \nabla' (F_h,b_h) \Vert_{L^2(q_i')}  + Ch^2   \Vert a_h \Vert_{L^2(q_i')}\,. 
\end{equation}
Then, we define
\begin{align}\label{uidef}
u_i(x')  \defas   \chi_{q_i' \setminus \omega_i^1}(x') \big( z_h(x') -  F_i x'   \big) \quad \text{for } x' \in q_i'\,. 
\end{align}
We note that $u_i \in SBV^2(q_i';\R^3)$ with 
\begin{equation*}
	\Vert \nabla' u_i \Vert_{L^2(q_i')} = \Vert  \nabla' z_h   - F_i      \Vert_{L^2(q_i' \setminus \omega^1_i)} \le Ch\Vert \nabla' (F_h,b_h) \Vert_{L^2(q_i')} + Ch^2   \Vert a_h \Vert_{L^2(q_i')}\,,  
\end{equation*}
and
\begin{align}\label{N5}
\mathcal{H}^1(J_{u_i}) \le  C\mathcal{H}^1(\Gamma_h\cap q_i')\,,\EEE 
\,
\end{align}
where we also used that $\ell(q_i) \le Ch$ for a constant $C>0$ depending on $\Lambda$, together with \eqref{eq: forvh0}(i) and the fact that 
\begin{align*}
J_{u_i}\subset (J_{z_h}\cap q'_i)\cup\partial^*\omega^1_i\subset  (\Gamma_h\cap q'_i)\cup\partial^*\omega^1_i\,.
\end{align*}
\EEE By applying Corollary \ref{replace-cor}  on $q_i'$ once more, this time for the function $u_i$, we obtain another set of finite perimeter $\omega^2_i \subset q_i'$ with $\mathcal{H}^1(\partial^* \omega_i^2) \le C \mathcal{H}^1(\Gamma_h \cap q_i') $, see \eqref{N5}, and $c_i \in \R^3$  such that 
\begin{align}\label{eq: forvh0?}
\Vert u_i - c_i    \Vert_{L^{2}(q_i' \setminus \omega_i^2)}     \le C \ell(q_i)\Vert \nabla'  u_i   \Vert_{L^2(q_i')}  \le C\ell(q_i) \big( h\Vert \nabla' (F_h,b_h) \Vert_{L^2(q_i')} + \OOO h^2\EEE   \Vert a_h \Vert_{L^2(q_i')} \big)\,.  
\end{align}
For later reference, we note that the estimates \eqref{eq: forvh0}--\eqref{eq: forvh0-} and \eqref{eq: forvh0?} are true for $\omega_i^1 = \omega_i^2 = \emptyset$, whenever $q_i \in \mathcal{W}^{\rm empt}_h$, see \eqref{the four sets}.  Now, we define the affine function
\begin{align}\label{yiii}
y_i(x') \defas F_i x'  + c_i \OOO\,. \EEE 
\end{align}
Recalling \eqref{uidef}, we note that
\begin{align}\label{eq: forvh}
z_h(x') = u_i (x')+ y_i (x')-c_i    \quad   \text{for } x' \in q_i' \setminus \MMM \omega_i \,,  \EEE 
\end{align}
\MMM where we set $\omega_i:=\omega_i^1 \cup \omega_i^2$. \EEE  By the isoperimetric inequality and \eqref{Wit2}(ii) we get 
\begin{equation}\label{eq:isoper_omega_1_2}
\L^2(\omega_i)\OOO \le\EEE C\big(\mathcal{H}^1(\partial^* \omega_i^1 \cup \partial^*\omega_i^2)\big)^2 \le C(\mathcal{H}^1(\Gamma_h \cap q_i'))^2 \le C\theta^2 \ell(q_i)^2  \le \frac{1}{100}\ell(q_i)^2\,,
\end{equation}
for $\theta  \in (0,1/16) $ sufficiently small. Given $ i \in \OOO \I\EEE$, we define 
\begin{equation}\label{NNNN}
\mathcal{N}_i := \lbrace q_j \colon q_j' \cap q_i' \neq \emptyset\rbrace\,, \quad \quad N(q_i) := \bigcup_{q_j\in\mathcal{N}_i} q_j'\,.
\end{equation}
\OOO It is then easy to deduce that \EEE for each $ i\in \OOO \I\EEE$ 
\begin{align}\label{eq: good estimates}
\Vert   y_i - y_j      \Vert_{L^2(q_i')} &  + h \Vert  (F_i,b_i)   - (F_j,b_j)      \Vert_{L^2(q_i')}  \notag\\
&\le C \ell(q_i) \Big( h\Vert \nabla' (F_h,b_h) \Vert_{L^2(N(q_i))} + \OOO h^2 \EEE   \Vert a_h \Vert_{L^2(N(q_i))} \Big)\quad \text{for all $q_j \in \mathcal{N}_i$}\,.
\end{align}
Indeed, since $y_i$ \MMM is \EEE affine, by Lemma \ref{lemma: rigid motions}, \eqref{eq:isoper_omega_1_2}, \eqref{eq: forvh}, \eqref{eq: forvh0?},\eqref{Wit1}(ii), and $q_i'\cap q_j'\subset N(q_i)$ for all $q_j\in \mathcal{N}_i$, we obtain
\begin{align}\label{eq:y_i-y_j}
\Vert y_i-y_j \Vert_{L^2(q_i')}&\leq C\Vert y_i-y_j \Vert_{L^2((q_i'\cap q_j')\setminus (\omega_i\cup \omega_j))}\leq C\Vert y_i-z_h \Vert_{L^2(q_i'\setminus \omega_i)}+C\Vert z_h-y_j\Vert_{L^2(q_j'\setminus\omega_j)}\notag\\
&\leq C\Vert u_i-c_i \Vert_{L^2(q_i'\setminus \omega_i)}+C\Vert u_j-c_j\Vert_{L^2(q_j'\setminus\omega_j)}\notag\\
&\le C \ell(q_i) \Big( h\Vert \nabla' (F_h,b_h) \Vert_{L^2(N(q_i))} + \OOO h^2\EEE \Vert a_h \Vert_{L^2(N(q_i))} \Big)\,. 
\end{align}
Similarly, using that $(F_i,b_i),(F_j,b_j)$ are constant, \eqref{eq:isoper_omega_1_2}, and \eqref{eq: forvh0}(ii), we get
\begin{align}\label{eq:F_i-F_j}
h \Vert  (F_i,b_i)-(F_j,b_j)\Vert_{L^2(q_i')}&\leq \OOO C\EEE h\Vert (F_i,b_i)-(F_j,b_j)\Vert_{L^2((q_i'\cap q_j')\setminus (\omega_i\cup \omega_j))}\notag\\
&\leq \OOO C\EEE h\Vert (F_i,b_i)-(F_h,b_h)\Vert_{L^2(q_i'\setminus \omega_i)}+\OOO C\EEE h\Vert (F_h,b_h)-(F_j,b_j)\Vert_{L^2(q_j'\setminus \omega_j)}\notag\\
&\le Ch\ell(q_i)\Vert \nabla' (F_h,b_h) \Vert_{L^2(N(q_i))}\,. 
\end{align}
\MMM Combining \EEE \eqref{eq:y_i-y_j} and \eqref{eq:F_i-F_j}, the estimate \eqref{eq: good estimates} follows.

By a similar argument, using \eqref{eq:isoper_omega_1_2}, \OOO that $(F_\eta,b_\eta)\in SO(3)$ (see \eqref{def:F_eta_b_eta} and \eqref{prep3}),  \eqref{eq: forvh0}(ii), and \eqref{eq: for SO2XXX}, \EEE we get
\begin{align}\label{eq: for SO2}
	\Vert \dist( (F_i,b_i), SO(3)) \Vert_{L^2(q_i')}  & \leq C\Vert \dist( (F_i,b_i), SO(3)) \Vert_{L^2(q_i'\setminus \omega^1_i)}\le C\Vert (F_i,b_i)-(F_\eta,b_\eta)\Vert_{L^2(q_i'\setminus \omega^1_i)}
	\EEE \notag\\
	& \le  C\Vert (F_i,b_i) -(F_h,b_h)\Vert_{L^2(q_i' \setminus \omega_i^1)}+C\Vert (F_h,b_h) -(F_\eta,b_\eta)\Vert_{L^2(q_i' \setminus \omega_i^1)}  
	\notag \\ 
	&  \le  Ch \Vert \nabla' (F_h,b_h) \Vert_{L^2(q_i')} + Ch^2  \Vert a_h \Vert_{L^2(q_i')}\,.
\end{align}
For notational convenience,  recalling \eqref{h3},  we also introduce the functions 
\begin{align}\label{eq: bardef}
	\begin{cases}
		\bar{y}_i(x') := z_h(x')\,, \quad  \bar{b}_i := b_h(x') & \text{for }  x'\in q'_i\,,   \ q_i \in \mathcal{W}^{\rm empt}_h\,, \\
		\bar{y}_i(x') := y_i(x'), \quad  \ \ \bar{b}_i := b_i & \text{for }  x'\in q'_i\,,  \ q_i \notin \mathcal{W}^{\rm empt}_h\,. 
	\end{cases}
\end{align}
Note that for every $q_i \notin\mathcal{W}^{\rm empt}_h$, the functions $\bar{y}_i, \bar{b}_i$ \EEE are affine or constant, respectively.  The fact that $\omega_i^1 = \omega_i^2 = \emptyset$ for $q_i \in \mathcal{W}^{\rm empt}_h$, along with \eqref{eq: forvh0}--\eqref{eq: forvh0-}, \eqref{eq: forvh0?}--\eqref{eq: forvh}, and  \eqref{eq: good estimates} also implies, for all $q_j \in \mathcal{N}_i$, 
\begin{align}\label{eq: good estimates2}
\Vert   \bar{y}_i - \bar{y}_j     \Vert_{L^2(q_i')} + h\Vert  (\nabla' \bar{y}_i, \bar{b}_i)    - (\nabla' \bar{y}_j, \bar{b}_j)      \Vert_{L^2(q_i')}  \le C \ell(q_i)  \big(h \Vert \nabla' (F_h,b_h) \Vert_{L^2(N(q_i))} +  \OOO h^2 \EEE  \Vert a_h \Vert_{L^2(N(q_i))}\big)\,.
\end{align}
\smallskip
We are now ready to proceed to the next step, namely the definition of the approximating sequence $(y_h)_{h>0}$ \OOO satisfying \eqref{eq: toshow2}\EEE. Recall the definition of  $U_h^{\rm good}$ and $U_h^{\rm ext}$ before and in \eqref{omega-ext0}. Following the notation in \eqref{eq:V_ident}, we also define
\begin{align}\label{tildeuex}
	\tilde{U}_h^{\rm good} \defas U_h^{\rm good} \times (-\tfrac{1}{2}, \tfrac{1}{2})\,, \quad \quad   \tilde{U}_h^{\rm ext} \defas U_h^{\rm ext} \times (-\tfrac{1}{2} - \Lambda h, \tfrac{1}{2} + \Lambda h)\,,  
\end{align}
where for the second set it will turn out to be convenient to thicken slightly also in the $x_3$-direction. Our next steps (Steps 6 and 7) consist in defining $y_h$ first on  $\tilde{U}_h^{\rm good}$ and then on $\tilde{U}_h^{\rm ext}$. As during the extension we slightly change the function, we denote the functions in Step~6 by $\bar{y}_h$ and in Step~7 by $y_h$ for a better distinction. 
\end{step}

\begin{step}{(6): Definition of $\bar{y}_h$ on $\tilde{U}_h^{\rm good}$}
Recalling the covering $\mathcal{W}_h := ( q_i )_{i \in \OOO \I\EEE}$  provided by Proposition~\ref{prop:whit}, \EEE we choose $(\varphi_i)_{i \in \OOO \I\EEE} \subset C^\infty_{ c}(\R^2; [0,1])$ with 
\begin{equation}\label{unity}
{\rm (i)} \ \ \sum_{i\in \OOO \I\EEE} \varphi_i(x') = 1 \,\ \forall x' \in U_{h}^{\rm good}\,,	 \quad 
{\rm (ii)} \ \  {\rm supp} \, (\varphi_i) \subset q_i'\,, \quad    {\rm (iii)} \ \   \Vert \nabla \varphi_i \Vert_\infty \le C \ell(q_i)^{-1} \ \  \forall  i \in \OOO \I\EEE\,. 
\end{equation}
As the proof of the existence of such a partition is very similar to the construction of a partition
of unity for Whitney coverings (see \cite[Chapter VI.1]{Stein:1970}), we omit it here. We also refer to \cite[Proof of Theorem 4.6]{Friedrich:15-4} for  similar arguments. 

Fix $d \in C^1_0(S;\R^3)$ to be specified later, see  the choice before \eqref{the bulk part}  below. Recalling \eqref{eq: bardef}, we define    $\bar{y}_h \in W^{1,2}(\tilde{U}_h^{\rm good};\R^3) $ by
\begin{equation}\label{bardef}
\bar{y}_h(x)  :=   \sum_{ i\in \OOO \I\EEE}  \varphi_i(x')\big(  \bar{y}_{i}(x')+ hx_3 \bar{b}_{i}(x')  \big)   + h^2 \frac{x_3^2}{2} d(x')\,. 
\end{equation}
We further introduce
\begin{equation}\label{eq:W_bjn} W^{\rm bjn}_h := \bigcup_{q \in \mathcal{W}^{\rm bdy}_h \cup \mathcal{W}^{\rm jump}_h \cup \mathcal{W}^{\rm neigh}_h } q'\,,\quad  \quad \ W^{\rm empt}_h := U_h^{\rm good} \setminus W^{\rm bjn}_h\,. \EEE
\end{equation}
We again use the \OOO notation \EEE $\tilde{W}^{\rm bjn}_h$ and $\tilde{W}^{\rm empt}_h$ for the  corresponding cylindrical sets in $\R^3$,  see also \eqref{tildeuex}. \EEE  We also note that $W^{\rm bjn}_h \subset W^{\rm cov}_h$ by \eqref{Wit2}(iii). Observe that the definition of $U_h^{\rm good}$ in Step~3,  \MMM the definition of $\mathcal{Q}^U_h$ in \eqref{5u}, \EEE and \eqref{Ubadest} shows that $\L^2( U_\eta\EEE
\setminus U_h^{\rm good}) \to 0$ as $h \to 0$. By the definition of $W^{\rm{empt}}_h$ in \eqref{eq:W_bjn}, together with \eqref{Wit2-5} and the previous observation,  we find
\begin{equation}\label{anothersetconv}
	 \L^2(U_\eta \setminus {W}^{\rm empt}_h) \to 0 \ \text{ as } h\to 0\,.\EEE
\end{equation}
In view of \eqref{unity} and \eqref{bardef}, on  $\tilde{W}_h^{\rm empt}$,   the definition of the deformation is similar to the ansatz for the recovery sequence in \cite[Equation (6.24)\EEE]{friesecke2002theorem}, namely, using also \EEE\eqref{eq: bardef}, we have
\begin{align}\label{eq: yh}
	\bar{y}_h(x) =  z_h(x')+ hx_3 b_h(x')     + h^2 \frac{x_3^2}{2} d(x') \ \ \text{for } x'\in \tilde{W}_h^{\rm empt}\,,
\end{align}
with 
\begin{equation}\label{eq: yh2}
	\nabla_h \bar{y}_h (x)   =  \big(R_h(x')  +  hx_3 (\nabla' b_h (x'), d(x'))\big)   + h^2 \frac{x_3^2}{2}( \nabla' d(x'),0)\,, 
\end{equation}
where $R_h(x')\defas  (\nabla' z_h(x'),  b_h(x'))$. Here,  we can repeat the estimate from the purely elastic case \cite[Proof of Theorem 6.1(ii)]{friesecke2002theorem}, which we perform here for the sake of completeness.  In view of \eqref{eq: yh2}, we obtain 
\begin{equation}\label{eq:A_h_def}
	R_\eta^T  \nabla_h \bar{y}_h    - {\rm Id}=   hx_3 R_\eta^T \big(\nabla' b_h, d\big) + h^2 \frac{x_3^2}{2} R^T_\eta  ( \nabla' d,0) + R^T_\eta R_h- {\rm Id}  =: A_h\EEE\,,
\end{equation}
where we set $R_\eta :=( F_\eta, b_\eta) \in SO(3)$\OOO, cf.~\eqref{def:F_eta_b_eta}. \EEE
Define  $\chi_h := \chi_{\tilde{W}^{\rm empt}_h}$ for brevity. In view of \eqref{h3}(i),(ii) and \eqref{anothersetconv}, we get that,  up to subsequences in $h$ (not relabeled), 
\begin{equation}\label{eq: recovery_approx}
\frac{1}{h}\chi_h(R^T_\eta R_h- {\rm Id}) \to 0\, \ \ \text{and } \ \frac{1}{h}\chi_hA_h  \to x_3 R_\eta^T (\nabla' b_\eta, d) \text{ pointwise } \L^3\text{-a.e. on }\Omega \setminus \tilde{V}_\eta\,.
\end{equation} 
Moreover,  \eqref{eq: nonlinear energy}(i), \eqref{eq:A_h_def}, \OOO the growth from above on $W$ in \eqref{eq: nonlinear energy}(v), and \EEE \eqref{h3}(ii) yield \[h^{-2}\chi_h W(\nabla_h \bar{y}_h)   \le Ch^{-2}|A_h|^2 \le  C(|\nabla'  b_{\eta}|^2 + |d|^2 + |\nabla' d|^2) + {\rm O}(h) \ \text{ on } \tilde{W}_h^{\rm empt}\,,\] 
where $ {\rm O}(h)$ has to be understood in the $L^1$-sense. Therefore, by the Dominated Convergence Theorem, a Taylor expansion of $W$, see  \eqref{eq: nonlinear energy} and  \eqref{def:Q_3}, \MMM \eqref{anothersetconv}, \EEE   \eqref{eq:A_h_def}, and \eqref{eq: recovery_approx},  we obtain 
\begin{align}\label{empt-est}
\lim_{h \to 0} h^{-2} \int_{\tilde{W}_h^{\rm empt}}    W(\nabla_h \bar{y}_h(x)) \, {\rm d}x = \frac{1}{2} \int_{\Omega \setminus \tilde{V}_\eta}  x_3^2 \Q_3(R_\eta^T  (\nabla' b_\eta , d)) \, {\rm d} x\,. 
\end{align}
Now, recalling \eqref{def:Q_2} and \eqref{def:second_FF}, we can choose a function $d\in C^1_0(S;\R^3)$ such that  
\begin{align}\label{the bulk part} 
\lim_{h \to 0} h^{-2} \int_{ \tilde{W}_h^{\rm empt}}     W(\nabla_h \bar{y}_h (x)) \, {\rm d}x  \le \frac{1}{24}\int_{S \setminus   {V}_\eta}\Q_2(\mathrm{II}_{y_\eta}(x'))\,\mathrm{d}x' + \eta\,.  
\end{align}

We now come to the integral over $\tilde{W}^{\rm bjn}_h$. On this set, the derivative of $\bar{y}_h$ reads as 
\begin{align}\label{eq: derivyh} 
\nabla_h \bar{y}_h (x) &  = \sum_{ j\in \OOO\I\EEE}    \varphi_j(x') \big( (\nabla' \bar{y}_j(x'),  \bar{b}_j(x'))   +  hx_3 (\nabla' \bar{b}_j(x'), d(x')) \big)   + h^2 \frac{x_3^2}{2}( \nabla' d(x'),0) \notag \\ 
& \ \ \ +   \sum_{ j\in \OOO\I\EEE}    \big(  \bar{y}_j(x')+ hx_3 \bar{b}_j(x')  \big) \otimes  (\nabla'\varphi_j(x'),0)\,. 
\end{align}
Fix $q_i \in \mathcal{W}^{\rm bdy}_h \cup \mathcal{W}^{\rm jump}_h \cup \mathcal{W}^{\rm neigh}_h$, and set $\OOO\tilde{q}'_i:=\EEE q_i'\times (-\frac{1}{2},\frac{1}{2})$.  Since   $\nabla'\big( \sum_{ j\in \OOO\I\EEE} \varphi_{j}\big) = 0 $, see  \eqref{unity}(i), \EEE we get  
\begin{align*}
\Big\Vert  \sum_{ j\in \OOO\I\EEE}   \big(  \bar{y}_j+ hx_3 \bar{b}_j  \big) \otimes    (\nabla'\varphi_j,0)    \Big\Vert_{L^2(\tilde{q}'_i)} & =  \Big\Vert  \sum_{  j\in \OOO\I\EEE}    \big(  (\bar{y}_j - \bar{y}_i)+ hx_3 (\bar{b}_j - \bar{b}_i) \big)  \otimes (\nabla'\varphi_j,0)   \Big\Vert_{L^2(\tilde{q}_i')}\,.
\end{align*}
By  \eqref{eq: good estimates2}  and  the fact that  $\Vert \nabla' \varphi_j \Vert_\infty \le C\ell(q_i)^{-1}$ for all $j \in \OOO\I\EEE$ with ${\rm supp}(\varphi_j) \cap q_i' \neq \emptyset$, see \eqref{Wit1}(ii) and \eqref{unity}(ii),(iii), we thus find
\begin{align*}
\Big\Vert  \sum_{ j\in \OOO\I\EEE}   \big(  \bar{y}_j+ hx_3 \bar{b}_j  \big) \otimes    (\nabla'\varphi_j,0)    \Big\Vert_{L^2(\tilde{q}'_i)}  \le   Ch \Vert \nabla' (F_h,b_h) \Vert_{L^2(N(q_i))} +  Ch^2   \Vert a_h \Vert_{L^2(N(q_i))}\,.
\end{align*}
This along with \OOO \eqref{eq: derivyh}, \EEE  \eqref{eq: bardef}, \eqref{eq: good estimates2}, the fact that $\nabla' \bar{b}_j \in \lbrace 0, \nabla' {b}_h\rbrace$, and $\ell(q_i) \le Ch$,  shows
\begin{align*}
	\Vert \nabla_h \bar{y}_h - (\nabla' \bar{y}_i,\bar{b}_i)   \Vert_{L^2(\tilde{q}_i')} \le  Ch \Vert \nabla' (F_h,b_h) \Vert_{L^2(N(q_i))}  +  Ch^2   \Vert a_h \Vert_{L^2(N(q_i))}  + Ch \Vert d \Vert_{W^{1,2}(q_i')}\,. 
\end{align*}
Then,  by  \MMM \eqref{eq: for SO2XXX}, \EEE  \eqref{yiii},   and \eqref{eq: for SO2}, we get 
\begin{align*}
	\Vert {\rm dist}(\nabla_h \bar{y}_h, SO(3))   \Vert_{L^2(\tilde{q}_i')} \le  Ch \Vert \nabla' (F_h,b_h) \Vert_{L^2(N(q_i))}  + C h^2  \Vert a_h \Vert_{L^2(N(q_i))}  +  Ch \Vert d \Vert_{W^{1,2}(q_i')}\,. 
\end{align*}
Summing over all $q_i \in  \mathcal{W}^{\rm bdy}_h \cup \mathcal{W}^{\rm jump}_h \cup \mathcal{W}^{\rm neigh}_h$ and using  \eqref{eq:W_bjn}, \EEE \eqref{Wit1}(iii), and \eqref{Wit2}(iii), we deduce
\begin{align}\label{finalllest}
\Vert {\rm dist}(\nabla_h \bar{y}_h, SO(3))   \Vert_{L^2(\tilde{W}^{\rm bjn}_h)}^2 \le Ch^2 \Vert \nabla' (F_h,b_h) \Vert^2_{L^2(W_h^{\rm cov})}  +  Ch^2 \Vert d \Vert^2_{W^{1,2}(W_h^{\rm cov})} +  Ch^4  \Vert a_h \Vert^2_{L^2(W_h^{\rm cov})}\OOO\,,\EEE    
\end{align}
where we also used $N(q_i) \subset q_i''$, see \eqref{NNNN}. Now, by \eqref{h3}(ii), \eqref{Wit2-5}, and \eqref{eq: for SO2XXX} we find that 
\[ h^{-2} \Vert {\rm dist}(\nabla_h \bar{y}_h, SO(3))   \Vert_{L^2(\tilde{W}^{\rm bjn}_h)}^2  \to 0 \ \text{ as } h\to 0\,.\]
This, along with  \eqref{eq:W_bjn},  \eqref{the bulk part}, and \eqref{eq: nonlinear energy}(v)  shows
\begin{align*}
	\lim_{h \to 0} h^{-2} \int_{\tilde{U}_h^{\rm good}}    W(\nabla_h \bar{y}_{ h}(x)) \, {\rm d}x  \le \frac{1}{24}\int_{S \setminus V_\eta}\Q_2(\mathrm{II}_{y_\eta}(x'))\,\mathrm{d}x' + \eta\,.  
\end{align*}
\end{step}

\begin{step}{(7): Definition of $y_h$ on $\tilde{U}_h^{\rm ext}$}
We now come to the definition of  $y_h$ on the extended set $\tilde{U}_h^{\rm ext}$ defined  \MMM in \eqref{tildeuex}. \EEE  Recall that $\mathcal{W}^{\rm bdy}_h \subset \mathcal{Q}_h^0$, see \eqref{Wit2}(i). We can thus extend the Whitney-type covering $\mathcal{W}_h$ given by Proposition \ref{prop:whit}  to a new covering, denoted by $\mathcal{W}^{\rm ext}_h$, by adding all squares of $\mathcal{Q}_h^{\rm ext}$, see \eqref{cubes-ext0}. On each of these squares $q_i \in \mathcal{W}^{\rm ext}_h \setminus \mathcal{W}_h$, we pick  one of the \EEE squares $q_j \in \mathcal{W}^{\rm bdy}_h$ \EEE which is closest to $q_i$ 
and define $F_i := F_j$ and $b_i := b_j$, where $F_j$ and $b_j$ are given in \eqref{eq: forvh0}. Accordingly, we also define  the affine function $y_i$, see \eqref{yiii}, and  as in \eqref {eq: bardef} we introduce the notation $\bar{y}_i(x') := y_i(x')$ and $\bar{b}_i := b_i$. Exploiting the fact that for neighboring squares the difference of these objects can be controlled, see \eqref{eq: good estimates} and its justification \MMM in \eqref{eq:y_i-y_j} and \eqref{eq:F_i-F_j}, \EEE it is elementary to check that \eqref{eq: good estimates2} still holds for the extended covering.  

We let $(\varphi_i)_{ i\in \OOO \I'\EEE}$ be a partition of unity related to $\mathcal{W}^{\rm ext}_h$ satisfying \eqref{unity}. Then, choosing  a field \EEE $d \in C^1_0(S;\R^3)$ as in Step 6, we define    ${y}_h \in W^{1,2}(\tilde{U}_h^{\rm ext};\R^3   ) $ by
\begin{align}\label{yh12}
{y}_h(x)  :=   \sum_{ i\in \OOO \I'\EEE}    \varphi_{ i\EEE}\big(  \bar{y}_{ i\EEE}(x')+ hx_3 \bar{b}_{ i\EEE}(x')  \big)   + h^2 \frac{x_3^2}{2} d(x')\,. 
\end{align}
The estimate \eqref{empt-est} holds still true for $y_h$ in place of $\bar{y}_h$ and ${W}_h^{\rm empt} \times (-\tfrac{1}{2} - \Lambda h, \tfrac{1}{2} + \Lambda h)$ in place of $\tilde{W}_h^{\rm empt}$. In particular, the limit is not affected by the thickening in the $x_3$-direction.  In a similar fashion, arguing as in Step 6, by replacing the estimate on $\tilde{W}^{\rm bjn}_h$ in \eqref{finalllest}  accordingly by a calculation on $W_h^* \defas ( {W}^{\rm bjn}_h \cup (U_h^{\rm ext} \setminus U_h^{\rm good}))  \times (-\tfrac{1}{2} - \Lambda h, \tfrac{1}{2} + \Lambda h)$, we get
\begin{equation*}
h^{-2}\Vert {\rm dist}(\nabla_h  {y}_h, SO(3))   \Vert_{L^2(W_h^*)}^2 \le C \Vert \nabla' (F_h,b_h) \Vert^2_{L^2(W_h^{\rm cov, *})}  +  C \Vert d \Vert^2_{W^{1,2}(W_h^{\rm cov, *})} +  Ch^2  \Vert a_h \Vert^2_{L^2(W_h^{\rm cov, *})} \,,
\end{equation*}
where we set $W_h^{\rm cov, *} \defas W_h^{\rm cov} \cup (U_h^{\rm ext} \setminus U_h^{\rm good})$. Hence, repeating verbatim the argument after \eqref{finalllest}, we obtain
\begin{equation}\label{finalllestNEW}
h^{-2}\Vert {\rm dist}(\nabla_h  {y}_h, SO(3))   \Vert_{L^2(W_h^*)}^2\to 0 \ \text{as } h\to \infty\,.
\end{equation}
\OOO A combination of these estimates as before, \EEE shows  
\begin{align}\label{the bulk part-neu} 
\lim_{h \to 0} h^{-2} \int_{\tilde{U}_h^{\rm ext}}    W(\nabla_hy_{ h}(x)) \, {\rm d}x  \le \frac{1}{24}\int_{S \setminus V_\eta}\Q_2(\mathrm{II}_{y_\eta}(x'))\,\mathrm{d}x' + \eta\,.  
\end{align}
\end{step}
 
\begin{step}{(8):  Conclusion}
We define \MMM $y_h\colon \Omega\to \R^3$ by $y_h = T_h(\mathrm{id})$ on $\tilde{V}_h$  \EEE (recall its definition before \eqref{eq: toshow1.5}) and otherwise as the restriction of the function in \eqref{yh12} to $\Omega \setminus \tilde{V}_h$. Here, we use \eqref{extensnion}, recall also \eqref{def:U_eta}, to ensure that $\Omega \setminus \tilde{V}_h \subset \tilde{U}_h^{\rm ext}$. We first treat the case that $\Vert y\Vert_{L^\infty(S)} < M$.  In view of \eqref{prep3}, \eqref{h3}(iv), and \eqref{yh12}, we find $\Vert y_h \Vert_{L^\infty(\Omega)} \le M$ for $h$ sufficiently small, i.e., \eqref{eq: toshow2}(iv) holds. Here, recalling the estimates in Step 5 it is indeed  not restrictive to assume that $\Vert \bar{y}_i \Vert_\infty \le \Vert z_h \Vert_\infty$ and $\Vert \bar{b}_i \Vert_\infty \le \Vert b_h \Vert_\infty$. From the representation of $y_h$ and $\nabla_h y_h$ on  $\tilde{W}_h^{\rm empt}$, see  \eqref{eq: yh}--\eqref{eq: yh2}, by  using \eqref{eq: toshow1.5}(i), \eqref{h3}, and the fact that  $\L^3((\Omega \setminus \tilde{V}_h)\setminus \tilde{W}_h^{\rm empt}) \to 0$, see \eqref{anothersetconv}, we find 
\[\chi_{\tilde{W}_h^{\rm empt}}y_h \to \CCC\tilde y_\eta\EEE \ \text{in } L^1(\Omega;\R^3)\,,\ \text{ and }\ \chi_{\tilde{W}_h^{\rm empt}}\nabla_h y_h \to  (\nabla'  \CCC\tilde y_\eta\EEE,  \partial_1 \CCC\tilde y_\eta\EEE \wedge \partial_2  \CCC\tilde y_\eta\EEE) \text{ strongly  in } L^2(\Omega;\R^{3\times 3})\,.
\]  This together with $\Vert y_h \Vert_{L^\infty(\Omega)} \le M$ and \eqref{finalllestNEW}  shows \eqref{eq: toshow2}(i),(ii). Eventually, \eqref{eq: toshow2}(iii) follows from \eqref{the bulk part-neu}.  

We close the proof by explaining the necessary adaptations in the case that $\Vert y\Vert_{L^\infty(S)} = M$. Note that the sequence $(y_h)_{h>0}$ as defined in Step 7 might not satisfy $\|y_h\|_{L^\infty}\leq M$ in this case.   Using  \eqref{extensnion} and recalling   the definition of $ \tilde{U}_h^{\rm ext}$ in \eqref{tildeuex} we find  
$$(\Omega \setminus \tilde{V}_h)_{(2-\sqrt{2})\Lambda h} \subset \tilde{U}_h^{\rm ext}. $$
Then, choosing a universal $C >0$ \OOO large enough \EEE such that $\Omega \subset B_C(0) \subset \R^3$,  and defining  
\begin{equation}\label{def:sigma_h}
	 \sigma_h \defas 1+ \frac{1}{2C}\Lambda h\,,\EEE
\end{equation}
it is elementary to check that 
\begin{align}\label{dila}
	\sigma_h x  \in   \tilde{U}_h^{\rm ext}  \quad \text{for all $x \in \Omega \setminus \tilde{V}_h$}\,.  
\end{align}
We define the sequence deformations \OOO $(\hat{y}_h)_{h>0}$ with \EEE $\hat{y}_h \in W^{1,2}(\OOO\Omega\setminus \tilde V_h\EEE; \R^3)$, by $\hat{y}_h :=\MMM T_h(\mathrm{id}) \EEE $ on $\tilde{V}_h$ and
$$\hat{y}_h(x) :=  \sigma_h^{-1} y_h ( \sigma_h x) \qquad \text{on $\Omega \setminus \tilde{V}_h$}, $$
where $y_h$ is given in \eqref{yh12}. By \eqref{dila} this is well defined. 
From \eqref{yh12} and \eqref{h3}(iv) we obtain $\Vert y_h \Vert_{L^\infty(\Omega)} \le M + Dh$, where $D:= \Vert d \Vert_\infty +   \Vert b_\eta \Vert_{\infty}$.  Recalling  the choice \eqref{def:sigma_h} \EEE
and choosing further
\begin{align}\label{Lambda-ref}
	\Lambda \ge \frac{2C\cdot D}{M}\,, 
\end{align}
which clearly only depends on $\eta$, we find $\Vert \hat{y}_h \Vert_{L^\infty(\Omega)} \le M$. This shows \eqref{eq: toshow2}(iv). As $\sigma_h \to 1$ for $h \to 0$, we easily get that also \eqref{eq: toshow2}(i)--(iii) are satisfied. This concludes the proof. 
\end{step}

\appendix

\section{Proofs of Proposition \ref{summary_estimates_proposition} and Corollary  \ref{difference_rigid-motions}}\label{sec: aux estimates}
 \begin{proof}[Proof of Proposition \ref{summary_estimates_proposition}] 
We recall once again that  by $C>0$ we denote generic constants which are independent of $h, \rho$. We fix $i \in I^h_{\rm g}$.
Let
\begin{equation*}
\mathcal P_{i,h}:=  \Big\{(P^j_{i,h})_j  \text{ \OOO the \EEE connected components of } \hat Q_{h,\rho}(i)\setminus   \partial E^*_h \Big\}\,,
\end{equation*}
with the enumeration being such that $\L^3(P^1_{i,h})$ is always maximal. We can use the maximality of $P^1_{i,h}$ in terms of its volume, the relative isoperimetric inequality, \OOO and \eqref{good_cuboids} \EEE to estimate
	
\begin{align}\label{eq:volume_of_maximal}
\begin{split}	
\mathcal{L}^3(\hat Q_{h,\rho}(i)\setminus P^1_{i,h})&=\sum_{j\geq 2}\mathcal{L}^3(P^j_{i,h})\leq c_{\OOO\rm{isop}\EEE}\sum_{j\geq 2}[\mathcal{H}^2(\partial P^j_{i,h}\OOO\cap \hat Q_{h,\rho}(i)\EEE)]^{3/2}\\
&\leq c_{\OOO\rm{isop}\EEE}\alpha^{1/2}h\sum_{j\geq 2}\mathcal{H}^2(\partial \EEE P^j_{i,h}\OOO\cap \hat Q_{h,\rho}(i)\EEE)\leq Ch\mathcal{H}^2\big(\partial E_h^*\cap \hat Q_h(i)\big)\,,		
\end{split}
\end{align}
for $C:=2c_{\OOO\mathrm{isop}\EEE}\alpha^{1/2}$. 
Furthermore, by \OOO \eqref{eq:Q_h}, \eqref{eq:Q_h_rho}, \EEE\eqref{good_cuboids}, \eqref{eq:volume_of_maximal}, and the choice of $\alpha$ in \eqref{eq: T2} and $\rho$ in \eqref{eq:choice_of_rho},
\begin{align}\label{eq: first on domi} 
\L^3(\hat Q_{h}(i)\setminus P^1_{i,h})  \le \L^3(\hat Q_{h}(i) \setminus \hat Q_{h,\rho}(i))+\OOO 2\EEE c_{\OOO\rm isop\EEE}\alpha^{3/2}h^3 \le \frac{9h^3 }{128} + \frac{9h^3}{128}  =  \frac{1}{64}  \mathcal{L}^3(\hat Q_{h}(i))\,.
\end{align}
We now distinguish between the two cases, \MMM namely \EEE
$${\rm (a)} \ \ P^1_{i,h} \subset E_h^*\,, \qquad {\rm (b)} \ \ P^1_{i,h} \cap  E_h^* = \emptyset\,. $$
	
\emph{Case ${\rm(a)}$}: If $P^1_{i,h} \subset E_h^*$, we just define  $D_{i,h} := P^1_{i,h}$,  $R_{i,h}:= {\rm Id}$, $b_{i,h}:= 0$, and $z_{i,h}\in W^{1,2}(\hat Q_{h,\rho}(i);\R^3)$ by $z_{i,h}:= {\rm id}$. Then  \eqref{big_volume_of_dominant_set} holds by \eqref{eq:volume_of_maximal} and \eqref{eq: first on domi}, while \eqref{big_surface_of_dominant_set}, \eqref{L2_gradient_estimates}, and \eqref{properties_of_Sobolev_replacement} are trivially satisfied, \OOO recalling \EEE\eqref{eq: *mod}. 
	
\medskip
	
\emph{Case ${\rm(b)}$}: Suppose now \MMM that \EEE  $P^1_{i,h}\cap  E_h^* = \emptyset$. Then,  we apply Theorem \ref{prop:rigidity} to the map $v^*_h$ for $\rho>0$, $\gamma :={\kappa}_h/h^2$, and $\eta_h\to 0$ satisfying \eqref{parameters_for_uniform_bounds}. Property \eqref{eq: main rigitity}  provides a rotation $R^1_{i,h} \in SO(3)$ and $b_{i,h}^1 \in \R^3$ such that 
\begin{align*}
\hspace{-1em}\begin{split}
{{\rm (i)}} & \ \  \int_{P^1_{i,h}}\big|{\rm sym}\big((R^1_{i,h})^T \nabla v^*_h-\mathrm{Id}\big)\big|^2\,\mathrm{d}x 
\leq C\big(1 +  C_{\eta_h} (h^{-2} {\kappa}_h)^{-15/2}h^{-3}\eps_{i,h}\big)\eps_{i,h}\,,
\\
{{\rm (ii)}} & \ \     h^{-2}\int_{P^1_{i,h}}|v^*_h-(R^1_{i,h}x+b^1_{i,h})|^2\,\mathrm{d}x+\int_{P^1_{i,h}}\big|(R^1_{i,h})^T \nabla  v^*_h-\mathrm{Id}\big|^2\,\mathrm{d}x
\leq  C_{\eta_h}(h^{-2}{\kappa}_h)^{\AAA-5\EEE}\eps_{i,h}\,,
\end{split}
\end{align*}
where we recall \eqref{eq: localized_elastic} and we have set $\gamma ={\kappa}_h/h^2 $. Our choice of $(\eta_h)_{h >0}$ in \eqref{parameters_for_uniform_bounds}, the definition in \eqref{good_cuboids}, and \eqref{rate_1_gamma_h} ensure that, for $h>0$ small enough depending on $\rho$,
\begin{align}\label{eq: main rigidity}
\hspace{-1em}\begin{split}
{{\rm (i)}} & \ \    \int_{P^1_{i,h}}\big|{\rm sym}\big((R^1_{i,h})^T \nabla v^*_h-\mathrm{Id}\big)\big|^2\,\mathrm{d}x 
\leq C_0\eps_{i,h}\,,
\\
{{\rm (ii)}} & \ \   h^{-2}\int_{P^1_{i,h}}|v^*_h-(R^1_{i,h}x+b^1_{i,h})|^2\,\mathrm{d}x
+\int_{P^1_{i,h}}\big|(R^1_{i,h})^T \nabla  v^*_h-\mathrm{Id}\big|^2\,\mathrm{d}x \leq C_0 h^{-2/5} \eps_{i,h}\,,
\end{split}
\end{align}
for a universal constant $C_0>0$. It is also easy to verify that 
\begin{align}\label{eq: bbb}
|b^1_{i,h}| \le  CM 
\end{align}  
for a universal constant $C>0$ that is independent of $h>0$. Indeed, since $\|v_h\|_{L^\infty(\Omega_h)} \leq M$ for $M \ge 1$, using the triangle inequality we obtain
\begin{align}\label{eq: argu1}
\L^3(P^1_{i,h})|b_{i,h}^{1}|^2 \le  C \int_{P^1_{i,h}}|v_{h}^*(x)-  (R^1_{i,h}x+b^1_{i,h}) |^2\OOO + \EEE C\L^3(P^1_{i,h}) \big( \|v_h\|_{L^\infty(\Omega_h)}^2 + ({\rm diam}(\Omega_h))^2   \big)\,.  
\end{align}
Using further \eqref{eq: first on domi},  \eqref{eq: main rigidity}(ii), and   \eqref{good_cuboids}, we get 
\begin{align}\label{eq: argu2}
|b_{i,h}^{1}|^2 \le Ch^{-3} h^2  h^{-2/5} \eps_{i,h}  +   C(M^2 + C) \le C(M^2 + C)\,, 
\end{align}
hence $|b_{i,h}^{1}| \le CM$. 
	
We now show that we can use Theorem \ref{th: kornSBDsmall} to obtain a Sobolev function satisfying \eqref{properties_of_Sobolev_replacement}. Introducing the function  $u_{i,h}\in SBV^2(\hat Q_{h,\rho}(i);\R^3)$ by
\begin{equation}\label{rotated_sbv_displacements}
u_{i,h}(x):=  \chi_{P^1_{i,h}}(x)\big[(R^1_{i,h})^T v_h^*(x)-x-(R^1_{i,h})^T b^1_{i,h}\big]\,,
\end{equation}
we observe that  $J_{u_{i,h}}\subset \partial E^*_h\cap \hat Q_{h,\rho}(i)$.    Now,  \eqref{rotated_sbv_displacements}, \eqref{eq: main rigidity}, and \eqref{good_cuboids}  imply that 
\begin{align}\label{cor_main rigidity}
\begin{split}
{\rm{(i)}} & \ \ \int_{\hat Q_{h,\rho}(i)}|{\rm sym}(\nabla u_{i,h})|^2\,\mathrm{d}x 
\leq C\eps_{i,h}\,,
\\
{\rm{(ii)}} & \ \  h^{-2}\int_{\hat Q_{h,\rho}(i)} |u_{i,h}|^2\,\mathrm{d}x+\int_{\hat Q_{h,\rho}(i)} |\nabla u_{i,h}|^2\,\mathrm{d}x \leq C\eps_{i,h}^{9/10}\,.
\end{split}
\end{align}
Applying Theorem \ref{th: kornSBDsmall} to the map $u_{i,h}$, in view of Remark \ref{eq:easy_rem} and \eqref{good_cuboids}, we obtain a set of finite perimeter $\omega_{i,h}\subset \hat Q_{h,\rho}(i)$ that satisfies
\begin{equation}\label{eq:set_of_finite_perimeter_control_area_volume}
\H^2(\partial^*\omega_{i,h})\leq c_{\rm{KP}} \mathcal{H}^2(J_{u_{i,h}}) \le  c_{\rm{KP}}\mathcal{H}^2(\partial E^*_h\cap \hat Q_{h,\rho}(i))\,,		
\end{equation}
and 
\begin{align}\label{set_of_finite_perimeter_control_area_volume}
\begin{split}	
\L^3(\omega_{i,h})&\leq   c_{\rm{KP}} \big(\mathcal{H}^2(J_{u_{i,h}}) \big)^{3/2} \le  c_{\rm{KP}}\big(\mathcal{H}^2(\partial E^*_h\cap \hat Q_{h,\rho}(i))\big)^{3/2}\\
&\leq  c_{\rm{KP}}\alpha^{1/2}h\mathcal{H}^2(\partial E^*_h\cap \hat Q_{h,\rho}(i))\le  c_{\rm{KP}}\alpha^{3/2}h^3 \le \frac{9}{128} h^3 <  \frac{1}{64}  \L^3(\hat Q_h(i))\,,
\end{split}
\end{align}
where we \OOO used again \EEE \eqref{eq: T2}. \OOO Theorem \EEE\ref{th: kornSBDsmall} provides also a Sobolev function $\zeta_{i,h}\in W^{1,2}(\hat Q_{h,\rho}(i);\R^3)$  such that  
\begin{align}\label{properties_of_Sobolev_replacementXXX}
\begin{split}
{\rm (i)}&\quad \zeta_{i,h}\equiv  u_{i,h} \quad \text{on }\ \hat Q_{h,\rho}(i) \setminus  \omega_{i,h}\,,\\
{\rm (ii)} & \quad \Vert {\rm sym}(\nabla  \zeta_{i,h}) \Vert_{L^2(\hat Q_{h,\rho}(i) ) } \le c_{\rm{KP}} \Vert {\rm sym}(\nabla  u_{i,h}) \Vert_{L^2(\hat Q_{h,\rho}(i) ) }\,,\\
{\rm (iii)} & \quad \Vert \zeta_{i,h} \Vert_\infty \le \Vert  u_{i,h}  \Vert_\infty\le CM\EEE\,,
\end{split}
\end{align}
where the final estimate in \OOO\eqref{properties_of_Sobolev_replacementXXX}\EEE$\rm{(iii)}$ follows from $\Vert v_h^* \Vert_\infty \le M$, \eqref{eq: bbb}, and the definition of $u_{i,h}$ in \eqref{rotated_sbv_displacements}.
	
We then define as dominant component the set
\begin{align}\label{eq: last dom def}
D_{i,h}:= P^1_{i,h} \setminus \omega_{i,h}\,,
\end{align}
so that by \eqref{eq:volume_of_maximal}, \eqref{eq: first on domi},   and \eqref{set_of_finite_perimeter_control_area_volume} we indeed verify that  \eqref{big_volume_of_dominant_set} holds. The estimate \eqref{big_surface_of_dominant_set} follows directly from the definition \eqref{eq: last dom def}, the fact that $\partial P^1_{i,h}\OOO\cap \hat Q_{h,\rho}(i)\EEE \subset \partial E_h^*\cap \hat Q_{h,\rho}(i)$, and \eqref{eq:set_of_finite_perimeter_control_area_volume}.
	
Applying the classical Korn's inequality  in $W^{1,2}$, we find $A_{i,h}\in \R^{3\times 3}_{\mathrm{skew}}$ such  that 
\begin{equation}\label{Korn_Poincare_inequality_first}
\int_{\hat Q_{h,\rho}(i)}|\nabla \zeta_{i,h}-A_{i,h}|^2 \,\mathrm{d}x
\leq C_{\rm{KP}}\int_{\hat Q_{h,\rho}(i)}|\mathrm{sym}(\nabla u_{i,h})|^2 \,\mathrm{d}x \le CC_{\rm{KP}} \eps_{i,h}
\end{equation}
for a universal $C_{\rm{KP}}>0$, where we used \eqref{properties_of_Sobolev_replacementXXX}(ii) and \eqref{cor_main rigidity}(i). We \MMM now \EEE  set  
\begin{equation}\label{eq:z_definition}
z_{i,h}:= R^1_{i,h} \zeta_{i,h} + R^1_{i,h}{\rm id} + b_{i,h}^1 \in  W^{1,2}(\hat Q_{h,\rho}(i);\R^3)\,,
\end{equation}
and observe that by \eqref{rotated_sbv_displacements}, \eqref{properties_of_Sobolev_replacementXXX}(i), and \eqref{eq: last dom def} it holds that $z_{i,h} \equiv v^*_h$ on $D_{i,h}$, which in particular implies \eqref{properties_of_Sobolev_replacement}(i).  Moreover, \eqref{properties_of_Sobolev_replacement}(iii) follows from \eqref{properties_of_Sobolev_replacementXXX}(iii) and \eqref{eq: bbb}.
	
It remains to show \eqref{L2_gradient_estimates} and \eqref{properties_of_Sobolev_replacement}(ii). In view of \eqref{Korn_Poincare_inequality_first} and \eqref{eq:z_definition}, we have
\begin{equation}\label{ineq_for_skew_part}
\int_{\hat Q_{h,\rho}(i)}|\nabla z_{i,h}-R^1_{i,h}(\mathrm{Id}+A_{i,h})|^2 \,\mathrm{d}x\leq C\eps_{i,h}\,.
\end{equation} 
Next, we substitute $R^1_{i,h}(\mathrm{Id}+A_{i,h})$ in \eqref{ineq_for_skew_part} by a suitable rotation. For this purpose, we prove that there exists $R_{i,h}\in SO(3)$ such that
\begin{equation}\label{optimal_rotation_i}
\L^3(\hat Q_{h,\rho}(i))|R^1_{i,h}(\mathrm{Id}+A_{i,h})-R_{i,h}|^2\leq C \eps_{i,h}\,.
\end{equation}
Indeed, by \eqref{big_volume_of_dominant_set}  together with \eqref{properties_of_Sobolev_replacementXXX}(i), \eqref{eq: last dom def},  \eqref{cor_main rigidity}(ii), and \eqref{Korn_Poincare_inequality_first}, we get 
\begin{align*}
\MMM h^3 \EEE |A_{i,h}|^2&\leq\L^3(D_{i,h})| A_{i,h}|^2= \int_{D_{i,h}}\big|\nabla u_{i,h}+ A_{i,h}-\nabla \zeta_{i,h}\big|^2\, \mathrm{d}x\\
&\leq 2\Big(\int_{D_{i,h}}|\nabla u_{i,h}|^2 \,\mathrm{d}x+\int_{D_{i,h}}|\nabla \zeta_{i,h} - A_{i,h}|^2\,\mathrm{d}x\Big)\leq C(\eps^{9/10}_{i,h}+\eps_{i,h})\,.
\end{align*}
Using that $\eps_{i,h} \le h^4$, see \eqref{good_cuboids}, \MMM  we \EEE obtain  
\begin{equation*}
|A_{i,h}|^2\leq Ch^{-3}\eps_{i,h}^{9/10} \le Ch^{-7/5}\eps_{i,h}^{1/2}\,.
\end{equation*}
A standard Taylor expansion, \OOO cf.~\EEE\cite[Equation (33)]{friesecke2002theorem}, gives
$${\rm dist}(G,SO(3)) = |{\rm sym}(G) - {\rm Id}| + {\rm O}(|G - {\rm Id}|^2)\,,$$
which further yields 
\begin{align*}
\mathrm{dist}^2\big(\OOO\mathrm{Id}+A_{i,h}\EEE,SO(3)\big) \leq C|A_{i,h}|^4\leq Ch^{-14/5}\eps_{i,h}\,,
\end{align*}
i.e., there exists $R_{i,h}\in SO(3)$ such that
\begin{equation*}
|R^1_{i,h}(\mathrm{Id}+A_{i,h})- R_{i,h}|^2  \leq Ch^{1/5}h^{-3}\eps_{i,h}\leq Ch^{-3}\eps_{i,h}  \le C (\L^3(\hat Q_{h,\rho}(i)))^{-1}  \eps_{i,h}  \,.
\end{equation*}
This proves \eqref{optimal_rotation_i} \OOO which, combined with \eqref{ineq_for_skew_part}, gives \EEE 
\begin{equation*}
\int_{\hat Q_{h,\rho}(i)}|\nabla z_{i,h}- R_{i,h}|^2\,\mathrm{d}x\leq C\eps_{i,h}\,.
\end{equation*}
\MMM This \EEE yields the second part of \eqref{properties_of_Sobolev_replacement}(ii). Applying the Poincar\'e  inequality on $W^{1,2}(\hat Q_{h,\rho}(i);\R^3)$ we obtain a vector $b_{i,h} \in \R^3$ such that the rigid motion 
$r_{i,h}(x):=  R_{i,h}x+ b_{i,h}$ satisfies
\begin{equation*}
h^{-2} \int_{\hat Q_{h,\rho}(i)}|z_{i,h}(x)- r_{i,h}(x)|^2\,\mathrm{d}x\leq C\eps_{i,h}\,,
\end{equation*}
concluding the proof of  \eqref{properties_of_Sobolev_replacement}(ii). Moreover, \eqref{L2_gradient_estimates} is an immediate  consequence of \eqref{properties_of_Sobolev_replacement}(i),(ii). Finally, by repeating verbatim  the argument in \eqref{eq: argu1}--\eqref{eq: argu2} with $b_{i,h}$ in place of $b_{i,h}^1$, we also obtain that $|b_{i,h}| \le CM$. This concludes the proof.     
\end{proof}

\begin{remark}\label{rem:choice_h_4}
\normalfont
Note  again  that the indices considered  in $I^h_{ \rm{g}}$ are related to cuboids for which $\eps_{i,h}\leq h^4$. In \cite[\OOO Equation \EEE (3.47)]{KFZ:2023} this additional requirement was not necessary since the global elastic energy scaling was $h^4$. In contrast, in the present setting,  the global elastic energy scaling  is $h^3$, and an additional control is needed for the following reason: Since the proof of Proposition \ref{summary_estimates_proposition}  relies on applying Theorem \ref{prop:rigidity} on $\hat Q_{h,\rho}(i)$, in order to ensure that the constant in \eqref{eq: main rigitity}{\OOO\rm(i)\EEE} can be chosen independently of $h>0$, it is essential that $\eps_{i,h}\ll h^3$. Thus, using the global energy bound $\eps_{i,h}\leq h^3$ would not be sufficient for this purpose. Yet, considering any bound of the form $\eps_{i,h}\leq h^{3+\OOO\mu\EEE}$, for $\OOO\mu\EEE>0$, would be sufficient, up to adjusting the curvature regularization parameter $\kappa_h$ in \eqref{rate_1_gamma_h}.

The choice $\OOO\mu\EEE=1$ is canonical, since the cardinality of indices $i$ such that $\eps_{i,h}>h^4$ is of the same order (namely $h^{-1}$) as the one of the indices $i'$ for which $ \H^2(\partial E^*_h\cap  \hat Q_{h,\rho}(i'))> \alpha h^2$, i.e., for which the surface area condition in the definition \eqref{good_cuboids} is violated.  \EEE 	
\end{remark}

\begin{proof}[Proof of Corollary \ref{difference_rigid-motions}]
By \eqref{L2_gradient_estimates} and the triangle inequality we can estimate 
\begin{align}\label{eq: the diff}
\int_{D_{i,h} \cap D_{i',h}}\big|r_{i,h}-r_{i',h}\big|^2\,\mathrm{d}x &\le 2 \int_{D_{i,h}}\big|v^*_h(x)-r_{i,h}(x)\big|^2\,\mathrm{d}x + 2\int_{D_{i',h}}\big|v^*_h(x)-r_{i',h}(x)\big|^2\,\mathrm{d}x \notag \\ 
& \le Ch^2(\eps_{i,h}+ \eps_{i',h})\,.
\end{align}
Note that $\mathcal{L}^3(\hat Q_{h}(i)  \cap \hat Q_{h}(i')) >\frac{9}{8} h^3$   and  $\mathcal{L}^3( \hat Q_{h}(j)  \setminus  D_{j,h}  ) \le \frac{9}{32} h^3$  by \eqref{big_volume_of_dominant_set} for $j=i,i'$. This yields 
$$\mathcal{L}^3(D_{i,h} \cap D_{i',h}) \ge \mathcal{L}^3(\hat Q_{h}(i)  \cap \hat Q_{h}(i')) -\mathcal{L}^3(\hat Q_{h}(i)    \setminus  D_{i,h}  ) -\mathcal{L}^3( Q_{h}(i') \setminus  D_{i',h}  ) \ge  \frac{9}{16}h^3\,.$$
Moreover, $\hat Q_{h}(i)  \cup \hat Q_{h}(i')$ is contained in a ball of radius $r = ch$ for a universal constant $c>0$. This along with \eqref{eq: the diff} and Lemma \ref{lemma: rigid motions} shows  the first part of \eqref{general_differences_rotations2}. The estimate for $|R_{i,h}-R_{i',h}|^2$ therein follows exactly in the same fashion, using again \eqref{L2_gradient_estimates}. 
\end{proof} 

\section{A linearization argument for the elastic energy liminf.}\label{sec: standard_liminf}
We \OOO detail \EEE the by now classical linearization argument to obtain \eqref{aux_lower_bound}. 
Let $(\lambda_h)_{h>0}\subset (0,\infty)$ be such that \begin{equation}\label{lambda_h_sequence}
	\lambda_h\to \infty\,,\  h\lambda_h  \to 0 \ \ \text{as } h\to 0\,,
\end{equation}
and  define
\begin{equation}\label{set_of_big_gradient}
\Theta_{h,\rho}:= \Omega_{1,\rho}\cap \lbrace \tilde{w}_h= y_h\rbrace \cap \{|G_h|\leq \lambda_h\}\,.
\end{equation}
Note that $\mathcal{L}^3(\lbrace \tilde{w}_h \neq y_h \rbrace\OOO\cap \Omega_{1,\rho}\EEE) \to 0$  by \eqref{eq:w_h_rescaling}, \eqref{from_v_to_y}, \eqref{w_h_almost_the same}(i), and a scaling argument. Combining this with the fact that  $\sup_{h>0}\Vert  G_h  \Vert_{L^2(\Omega_{1,\rho})} \le C$, see \eqref{eq:rescaled_fields_convergence}, $\lambda_h \to + \infty$, and Chebyshev's inequality, we obtain 
\begin{align}\label{eq: cheby}
	\mathcal{L}^3(\Omega_{1,\rho} \setminus \Theta_{h,\rho}) \to 0 \quad \text{as $h\to 0$}\,, 
\end{align}
i.e., $\chi_{\Theta_{\OOO h,\rho\EEE}}\to 1$ boundedly in measure \OOO on \EEE $\Omega_{1,\rho}$ as $h\to 0$. By   \eqref{admissible_configurations_h_level}, $W(\mathrm{Id})=0$, $W \ge 0$, and the definition of $\Theta_{\OOO h,\rho\EEE}$, we get 
\begin{align*}
	\liminf_{h \to 0} \Big(h^{-2}\int_{\Omega \setminus \overline{V_h}} W(\nabla_hy_h)\, {\rm d}x\Big) & = \liminf_{h \to 0} \Big(h^{-2}\int_{\Omega} W(\nabla_hy_h)\, {\rm d}x\Big)  \\
	& \ge  \liminf_{h \to 0} \Big(h^{-2}\int_{\Omega_{1,\rho}} \chi_{\Theta_{h,\rho}} W(\nabla_h \tilde{w}_h)\, {\rm d}x\Big)\,.
\end{align*}
The regularity and the structural hypotheses on $W$, recall \eqref{eq: nonlinear energy}, imply that
$$W({\rm Id}+F) = \tfrac{1}{2}\mathcal{Q}_3(F) + \Phi(F)\,,$$
where $\Phi\colon\R^{3 \times 3}\to  \R $ satisfies 
\begin{equation}\label{PPPhi}
\sup \big\{ \tfrac{|\Phi(F)|}{|F|^2} \colon \, |F| \le \sigma \big\} \to 0 \ \ \text{ as $\sigma \to 0$\,.}
\end{equation} 
Together with the definition of $G_h$ in  \eqref{eq:rescaled_fields_convergence}, we obtain
\begin{align}\label{liminf_first_inequality}
\liminf_{h \to 0} \Big(h^{-2}\int_{\Omega\setminus \overline{V_h}} W(\nabla_hy_h)\, {\rm d}x\Big) & \ge \liminf_{h\to 0}  \Big(h^{-2}\int_{\Omega_{1,\rho}} \chi_{\Theta_{h,\rho}} W({\rm Id} + hG_h) \, {\rm d}x\Big) \nonumber
	\\
	&\ge   \liminf_{h\to 0} \int_{\Omega_{1,\rho}} \chi_{\Theta_{h,\rho}}\Big( \tfrac{1}{2}\Q_3(G_h) +h^{-2} \Phi(hG_h)  \Big) \, {\rm d}x \nonumber
	\\
	&= \liminf_{h\to 0} \frac{1}{2}\int_{\Omega_{1,\rho}} \chi_{\Theta_{h,\rho}} \Q_3(G_h) \,{\rm d}x\,.
\end{align}
In the passage to the last line above, we made use of the fact that 
$${\limsup_{h\to 0}  \int_{\Omega_{1,\rho}} \chi_{\Theta_{h,\rho}} h^{-2} |\Phi(hG_h)|  \, {\rm d}x \le \limsup_{h \to 0}\left( \sup\Big\{ \tfrac{|\Phi(hG_h)|}{|h G_h|^2} \colon \, |hG_h| \le  h\lambda_h  \Big\}   \int_{\Omega_{1,\rho}} \chi_{\Theta_{h,\rho}} |G_h|^2 \, {\rm d}x\right)= 0\,,    }$$
which follows from the fact that $(G_h)_{h>0}$ is bounded in $L^2(\Omega_{1,\rho};\R^{3 \times 3})$,  \eqref{set_of_big_gradient}, \eqref{PPPhi}, and   $h\lambda_h\to 0$\,, \OOO cf.~\eqref{lambda_h_sequence}\EEE.  Hence, \OOO \eqref{liminf_first_inequality}, \EEE\eqref{eq:rescaled_fields_convergence}, the fact that $\chi_{\OOO \Theta_{h,\rho}\EEE}\to 1$ boundedly in measure in $\Omega_{1,\rho}$, see \eqref{eq: cheby}, and the convexity of $\Q_3$ imply that 
\begin{equation*}
\OOO\liminf_{h \to 0} \Big(h^{-2}\int_{\Omega\setminus \overline{V_h}} W(\nabla_hy_h)\, {\rm d}x\Big)\EEE	\geq \frac{1}{2}\int_{\Omega_{1,\rho}} \Q_3(G)\,{\rm d}x\geq \frac{1}{2}\int_{\Omega_{1,\rho}}\Q_2\big(G'\big)\,\mathrm{d}x\,,\\
\end{equation*}
\OOO which is exactly \eqref{aux_lower_bound}\,. \EEE
\medskip

\section*{Acknowledgements} 
\MMM MF was supported by the RTG 2339 “Interfaces, Complex Structures, and Singular Limits”
of the German Science Foundation (DFG). \EEE The research of LK was supported by the DFG through the Emmy Noether Programme (project number 509436910). KZ was supported by the Sonderforschungsbereich 1060 and the Hausdorff Center for Mathematics (HCM) under Germany's Excellence Strategy -EXC-2047/1-390685813.

\typeout{References}

 \end{document}